\documentclass[11pt]{article}
\usepackage{amsmath, amssymb, amsthm}
\usepackage{graphicx} 
\usepackage{enumitem}                                 
\usepackage[margin=1.0in]{geometry}

\AtEndDocument{ \bigskip{\footnotesize%
\textsc{Department of Mathematics, The University of Haifa at Oranim, Tivon 36006, Israel} \par 
  \textit{E-mail address} \texttt{ amir.algom@math.haifa.ac.il}
}

\bigskip{\footnotesize%
\textsc{Department of Mathematics, the Pennsylvania State University, University
Park, PA 16802, USA} \par 
\textit{E-mail address} \texttt{ hertz@math.psu.edu
} \par 
\textit{E-mail address} \texttt{ zhirenw@psu.edu
}
}}

\newtheorem{theorem}{Theorem}[section]
\newtheorem{Remark} [theorem]{Remark}

\newtheorem{Counter-example}[theorem]{Counter example}
\newtheorem{Claim}[theorem]{Claim}

\newtheorem{Lemma}[theorem]{Lemma}
\newtheorem{Proposition}[theorem]{Proposition}

\newtheorem{Definition}[theorem]{Definition}
\newtheorem{Corollary}[theorem]{Corollary}

\newtheorem*{theorem*}{Theorem}

\newcommand{\supp}{\text{supp}}
\newcommand{\diam}{\text{diam}}

\title{Polynomial Fourier decay and a cocycle version of Dolgopyat's method  for   self conformal measures}
\author{Amir Algom, Federico Rodriguez Hertz, and Zhiren Wang}
\date{}
\begin{document}
\maketitle
\begin{abstract}
We show that every self conformal measure with respect to a $C^2 (\mathbb{R})$ IFS $\Phi$  has polynomial Fourier decay under some mild and natural non-linearity conditions. In particular, every such measure  has polynomial  decay if $\Phi$ is  $C^\omega (\mathbb{R})$ and contains a non-affine map.

A key ingredient in our argument is a cocycle version of Dolgopyat's method, that does not require the cylinder covering of the attractor to be a Markov partition.  It is used to obtain spectral gap-type estimates for the transfer operator, which in turn imply a renewal theorem with an exponential error term in the spirit of Li (2022).  
 \end{abstract}

\section{Introduction}
\subsection{Background and main results}
Let $\nu$ be a Borel probability measure on $\mathbb{R}$. For every $q\in \mathbb{R}$   the Fourier transform of $\nu$ at $q$ is defined by 
\begin{equation*} 
\mathcal{F}_q (\nu) := \int \exp( 2\pi i q x) d\nu(x).
\end{equation*} 
The measure $\nu$ is called a \textit{Rajchman measure} if $ \mathcal{F}_q(\nu)= o(1)$ as $|q|\rightarrow \infty$. By the Riemann-Lebesgue Lemma, if $\nu$ is absolutely continuous then it is Rajchman. On the other hand, by  Wiener's Lemma  if $\nu$ has an atom then it is not  Rajchman. For measures that are both continuous (no atoms) and singular, determining whether or not $\nu$ is a Rajchman measure may be a challenging problem even for well structured measures. The Rajchman property has various geometric consequences on the measure $\nu$ and its support, e.g. regarding the uniqueness problem    \cite{li2019trigonometric}.

Further information about the rate of decay of $\mathcal{F}_q(\nu)$ has even stronger geometric consequences. For example,  by a classical Theorem of Davenport-Erd\H{o}s-LeVeque \cite{Davenport1964Erdos}, if $ \mathcal{F}_q(\nu)$ decays at a logarithmic rate then $\nu$-a.e. point is normal to all integer bases (see also \cite{Velani2022Aga}). Wide ranging geometric information can be derived  if $ \mathcal{F}_q(\nu)$ decays at a polynomial rate, that is, if there exists some $\alpha>0$ such that
$$ \mathcal{F}_q(\nu) = O\left( \frac{1}{|q|^\alpha}\right).$$
For example, by a  result of Shmerkin \cite{Shmkerin2014Abs}, it implies that  the convolution of $\nu$ with any measure of dimension $1$ is  absolutely continuous. We refer to Mattila's recent book \cite{Mattila2015new} for various further applications of Fourier decay in geometric measure theory and related fields.

The goal of this paper is to prove  that every measure in a fundamental class of fractal measures  has polynomial Fourier decay, as long as it satisfies some very mild non-linearity conditions; Furthermore, these conditions can be easily verified in concrete examples. To define this class of measures, let $\Phi= \lbrace f_1,...,f_n \rbrace$ be a finite set of strict contractions of a compact interval $I\subseteq \mathbb{R}$ (an \textit{IFS} - Iterated Function System), such that every $f_i$ is differentiable.  We say that  $\Phi$ is $C^\alpha$ smooth if every $f_i$ is at least $C^\alpha$ smooth for some $\alpha\geq 1$.  It is well known that there exists a unique compact set $\emptyset \neq K=K_\Phi \subseteq I$ such that
\begin{equation} \label{Eq union}
K = \bigcup_{i=1} ^n f_i (K).
\end{equation}
The set $K$ is called  the \textit{attractor} of the IFS $\lbrace f_1,...,f_n \rbrace$.  We always assume that there exist $i,j$ such that  the fixed point of $f_i$ does not equal the fixed point of $f_j$. This ensures that $K$ is infinite. We call $\Phi$  \textit{uniformly contracting} if 
$$0< \inf \lbrace |f '(x)|:\, f\in \Phi, x\in I \rbrace \leq \sup \lbrace |f '(x)| :\, f\in \Phi, x\in I  \rbrace <1.$$
Next, writing $\mathcal{A}= \lbrace 1,...,n\rbrace$, for every $\omega \in \mathcal{A} ^\mathbb{N}$ and $m\in \mathbb{N}$ let
$$f_{\omega|_m} := f_{\omega_1} \circ \circ \circ f_{\omega_m}.$$
Fix $x_0 \in I$.  Then we have a surjective coding map $\pi: \mathcal{A} ^\mathbb{N} \rightarrow K$ given by
\begin{equation} \label{Eq coding}
\omega \in \mathcal{A}^{\mathbb N} \mapsto x_\omega:= \lim_{m\rightarrow \infty}  f_{\omega|_m}  (x_0),
\end{equation}
which is  well defined  because of uniform contraction (see e.g. \cite[Section 2.1]{bishop2013fractal}). 

Let $\textbf{p}=(p_1,...,p_n)$ be a strictly positive probability vector, that is, $p_i >0$ for all $i$ and $\sum_i p_i =1$, and let $\mathbb{P}=\mathbf{p}^\mathbb{N}$ be the corresponding Bernoulli measure on $\mathcal{A}^\mathbb{N}$. We call the measure $\nu=\nu_\mathbf{p} = \pi \mathbb{P}$ on $K$ the \textit{self conformal measure} corresponding to $\mathbf{p}$, and note that our assumptions are known to imply that it is non-atomic.  Equivalently, $\nu_\mathbf{p}$ is the unique Borel probability  measure on $K$ such that
$$\nu = \sum_{i=1} ^n p_i\cdot  f_i\nu,\quad \text{ where } f_i \nu \text{ is the push-forward of } \nu \text{ via } f_i.$$ 
When all the maps in $\Phi$ are affine we call $\Phi$ a \textit{self-similar IFS} and $\nu$ a \textit{self-similar measure}.

Next,  we say that a $C^2 (\mathbb{R})$ IFS $\Psi$ is \textit{linear} if $g''(x) = 0$  for every  $x\in K_\Psi$  and $g\in \Psi$. This notion was introduced in our previous work \cite{algom2021decay}, though it  implicitly appeared in the literature prior to that, notably in the work of Hochman-Shmerkin \cite{hochmanshmerkin2015}. It is clear that if $\Psi$ is $C^\omega (\mathbb{R})$ and linear  then  it must be self-similar. While we believe such IFSs should exist, we are not aware of any known example of a  linear $C^r (\mathbb{R})$ smooth IFS that is not self-similar for  $r\geq 1$. A major source of examples of  non-linear IFSs arise in smooth dynamics from certain homoclinic intersections, see e.g. \cite[Section I]{Moreira1996stable}.

We can now state the main result of this paper. We say that a $C^r$ IFS $\Phi$ is  conjugate to an IFS $\Psi$ if there is a $C^r$ diffeomorphism $h$ such that  $\Phi = \lbrace h\circ g \circ h^{-1}\rbrace_{g\in \Psi}$. 
\begin{theorem} \label{Main Theorem}
Let $\Phi$ be a uniformly contracting $C^{r} (\mathbb{R})$   IFS where $r\geq 2$. If $\Phi$ is not conjugate to a linear IFS then every non-atomic self-conformal measure $\nu$ admits   some $\alpha=\alpha(\nu)>0$ such that
\begin{equation*}
\left| \mathcal{F}_q \left( g \nu \right) \right| = O\left(\frac{1}{|q|^\alpha} \right).
\end{equation*} 
\end{theorem}
Moreover, our argument yields easily verifiable conditions when Theorem \ref{Main Theorem} may be applied (without directly checking whether the IFS is conjugate to linear); As we discuss later, it suffices for a $C^2 (\mathbb{R})$ IFS to satisfy conditions \eqref{Add property 1} and \eqref{Add property 2} below for the conclusion of Theorem \ref{Main Theorem} to hold true. 


A great deal of attention has been given to the case when IFS in question is real analytic. This is mainly because such IFSs arise naturally  in number theory, e.g. as (finitely many) inverse branches of the Gauss map \cite{Sahl2016Jor}, as Furstenberg measures for some $\text{SL}(2,\mathbb{R})$ cocycles \cite{Yoccoz2004some, Avila2010jairo}, and are closely related to Patterson-Sullivan measures on limit sets of some Schottky groups \cite{Li2021Naud, Naud2005exp, Bour2017dya}, among others.  Combining Theorem \ref{Main Theorem} with a new method from a paper of Algom et al. \cite{Algom2023Wu}, we can derive the following Corollary regarding Fourier decay in  this setting:
\begin{Corollary} \label{Main Corollary}
Let $\Phi$ be a $C^\omega (\mathbb{R})$ IFS. If $\Phi$ contains a non-affine map then every non-atomic self-conformal measure $\nu$ admits some $\alpha=\alpha(\nu)>0$ such that
\begin{equation*}
\left| \mathcal{F}_q \left( \nu \right) \right| = O\left(\frac{1}{|q|^\alpha} \right).
\end{equation*}
\end{Corollary}

We emphasize that  no separation conditions are imposed on the IFS in both Theorem \ref{Main Theorem} and Corollary \ref{Main Corollary}. Thus, as we discuss below, these are essentially the first instances where polynomial decay is obtained for such measures without an underlying Markov partition, or a more specialized algebraic or dynamical setup. We remark that, simultaneously and independently of our work, Baker and Sahlsten \cite{Baker2023Sahl} obtained similar results but with a different method. See Remark \ref{Rmk BS23} for more details. Also, we note that Baker and Banaji \cite{baker2024polynomial} independently obtained a proof of Corollary \ref{Main Corollary}, with a method that differs considerably from that of \cite{Algom2023Wu}. We refer to \cite{Algom2023Wu, baker2024polynomial} for more discussion and comparisons between the two techniques.

 Relying on his own previous work \cite{Li2018decay}  and on the work of Bourgain-Dyatlov \cite{Bour2017dya}, Li \cite{li2018fourier} proved polynomial Fourier decay for Furstenberg measures for $\text{SL}(2,\mathbb{R})$ cocycles under mild assumptions. Corollary \ref{Main Corollary} gives a new proof of these results when the measure is self-conformal (see  \cite{Yoccoz2004some, Avila2010jairo} for  conditions that ensure this happens).  We  remark that, as we discuss in Section \ref{Section method} below, the renewal theoretic parts of our argument are closely related to Li \cite{Li2018decay, li2018fourier}.  Sahlsten-Stevens \cite[Theorem 1.1]{sahlsten2020fourier} proved polynomial Fourier decay for a class of stationary measures that includes self-conformal measures, with respect to  totally non-linear $C^\omega$ IFSs with strong separation (i.e.  the union \eqref{Eq union} is disjoint). Their methods also apply when the IFS is only $C^2(\mathbb{R})$ under the additional assumptions that $K_\Phi=[0,1]$ and some further separation conditions on $\Phi$. For self-conformal measures,  Theorem \ref{Main Theorem} improves these results since we  do not require any separation conditions on the IFS, or that the attractor is an interval. In fact, Corollary \ref{Main Corollary} removes virtually all the assumptions made in the corresponding result of Sahlsten-Stevens \cite[Theorem 1.1]{sahlsten2020fourier}, except for the existence of an analytic non-affine map in the IFS. Theorem \ref{Main Theorem} also improves our previous result \cite[Theorem 1.1]{algom2021decay} by upgrading the rate of decay in the non-conjugate to linear setting from logarithmic to polynomial. Finally, we point out that when the IFS in question is  conjugate to self-similar via a non-linear $C^2$ map, Kaufman \cite{Kaufman1984ber} (for some Bernoulli convolutions) and later Mosquera-Shmerkin \cite{Shmerkin2018mos} (for homogeneous self-similar measures) proved polynomial Fourier decay for all self-conformal measures. The only concrete examples of polynomial Fourier decay in the fully self-similar setup were given by Dai-Feng-Wang \cite{Dai2007Feng} and Streck  \cite{streck2023absolute} for some homogeneous IFSs that enjoy  nice number theoretic properties. It is known, though, that very few self-similar IFSs should fail this property; See Solomyak \cite{Solomyak2021ssdecay}.

Finally, combining Corollary \ref{Main Corollary} with the recent works of Br\'{e}mont \cite{bremont2019rajchman} and Li-Sahlsten \cite{li2019trigonometric} we obtain the following complete characterization of the Rajchman property for $C^\omega (\mathbb{R})$ IFSs:  Let  $\Phi$ be a $C^\omega (\mathbb{R})$ IFS. If $\Phi$ admits a non-atomic self-conformal measure that is not Rajchman, then $\Phi = \lbrace r_1\cdot x+t_1,...,r_n\cdot x+t_n\rbrace$ is self-similar and  there exists a Pisot number $r^{-1}$  such that $r_i = r^{\ell_i}$ for some $\ell_i \in \mathbb{Z}_+$. Furthermore, $\Phi$ is affinely conjugated to an IFS  that has all of its translations in  $\mathbb{Q}(r)$.  For more recent results  on Fourier decay for self similar measures we refer to \cite{Solomyak2021ssdecay, rapaport2021rajchman, varju2020fourier, algom2020decay, Shmerkin2018mos, Buf2014Sol, Dai2007Feng, Dai2012ber, streck2023absolute} and references therein.

\subsection{Outline of proof: A cocycle version of Dolgopyat's method} \label{Section method}
We proceed to discuss the method of proof of Theorem \ref{Main Theorem}. There is no loss of generality in assuming our IFS is $C^2 ([0,1])$. Fix a  self-conformal measure $\nu = \nu_\mathbf{p}$, and assume the conditions as in Theorem \ref{Main Theorem} are met. We aim to show that $\nu$ has polynomial Fourier decay. Our proof consists of four steps:

The first and most involved step  is arguing that the transfer operator corresponding to the derivative cocycle and $\mathbf{p}$ (Definition \ref{Def transfer operator}) satisfies a spectral gap-type estimate (Theorem \ref{Theorem spectral gap}). Both our result and  method of proof are strongly related to the works of Dolgopyat \cite{Dol1998annals, Dolgopyat2000Mix2}, Naud \cite{Naud2005exp}, and Stoyanov \cite{Stoyanov2011spectra, Stoyanov2001decay} (see also \cite{Araujo2016Melbourne, Avila2006yoccoz, Baladi2005Valle}) .  These papers utilize a proof technique that originates from the  work of Dolgopyat \cite{Dol1998annals}, commonly known  as Dolgopyat's method. It is used to obtain spectral gap-type estimates by, roughly speaking, a reduction to an $L^2$ contraction estimate (in the spirit of Proposition \ref{Prop. 5.3 Naud}), and then a proof of this estimate via the construction of  so-called Dolgopyat operators (in the spirit of Lemma \ref{Lemma 5.4 Naud}), that  control  cancellations of the transfer operator.

There are, among others, three critical  properties of $\nu$ and $\Phi$ that are used in Naud's version \cite{Naud2005exp} of Dolgopyat's method, which is the one we ultimately rely on here: That the union \eqref{Eq union} corresponds to a Markov partition (i.e. for $i\neq j$,  $f_i([0,1])$ and $f_j([0,1])$ may intersect only at their endpoints),  that the measure $\nu$ satisfies the Federer property (see e.g. Theorem \ref{Theorem disint} Part (6)), and that $\Phi$ satisfies the uniform non-integrability condition (UNI in short -  see Claim \ref{Claim UNI}). In our setting, since we assume no separation conditions (in particular, the cylinder covering is not a Markov partition),  we cannot assume $\nu$ has the Federer property. As for (UNI), in our previous work \cite[proof of Claim 2.13]{algom2021decay} we showed that if $\Phi$ is not conjugate to linear then (UNI) does hold true (regardless of any separation assumptions). In fact, our non-linearity assumption is only used to obtain the (UNI) condition.

One of the main innovations of this paper is the introduction a cocycle version of Dolgopyat's method, that we use get around these issues: First, we express $\nu$ as an integral over a certain family of random measures, and express the transfer operator as a corresponding convex combination of smaller "pieces" of itself. This is based upon a technique that first appeared in \cite{Galcier2016Lq}, and was subsequently applied in several other papers. e.g.  \cite{Antti2018orponen, Shmerkin2018Solomyak, solomyak2023absolute}. In our variant, that is  closest to the construction in \cite{Algom2022Baker},    these random measures typically satisfy a certain dynamical self-conformality relation with \textit{strong separation}, and  satisfy the Federer property. Critically, we are able to preserve the (UNI) condition into all the different pieces of the transfer operator corresponding to this decomposition. An important preliminary step is the construction of a special  covering of our IFS by (possibly overlapping) sub-IFSs (Claim \ref{Claim key}), followed by the construction of this disintegration in Theorem \ref{Theorem disint}. We then proceed to state and prove our spectral gap-type estimate (Theorem \ref{Theorem spectral gap}) by first reducing to a certain randomized $L^2$ contraction estimate (Proposition \ref{Prop. 5.3 Naud}), and then prove  this estimate via the construction of corresponding randomized  Dolgopyat operators (Lemma \ref{Lemma 5.4 Naud}). All of this discussion takes place in Section \ref{Section non-linearity implies s-gap}.

In the second step of our argument, we deduce from Theorem \ref{Theorem spectral gap} (our spectral gap estimate) a renewal theorem with an exponential error term. This is Proposition \ref{Proposition renewal} in Section \ref{Section renewal}. Such strong renewal Theorems were first proved by Li \cite{li2018fourier} for random walks arising from certain random matrix products. Our proof technique is based on that  of Li \cite{Li2018decay, li2018fourier}, and Boyer \cite{boyer2016rate}, relating the error term in the renewal operator to the resolvent of the transfer operator, on which we have good control thanks to  Theorem \ref{Theorem spectral gap}.

In the third step of our proof, we deduce from Proposition \ref{Proposition renewal} (our renewal Theorem) an effective equidistribution result for certain random walks driven by the derivative cocycle. This is Theorem \ref{Theorem equi}, that critically holds with an exponential error term.  The exponent we obtain is related to the size of the strip where we have spectral gap (the $\epsilon$ from Theorem \ref{Theorem spectral gap}). Finally, in Section \ref{Section Fourier decay} we obtain the desired Fourier decay bound on $\mathcal{F}_q (\nu)$ by combining Theorem \ref{Theorem equi} with our previous proof scheme from \cite[Section 4.2]{algom2020decay}, that relies on delicate linerization arguments and estimation of certain oscillatory integrals as in \cite{Hochman2020Host}.

As for the proof of Corollary \ref{Main Corollary}, our main  ingredient is that  the following dichotomy holds for any given $C^\omega (\mathbb{R})$ IFS:  Either it is not $C^2$ conjugate to linear, or it is $C^\omega$ conjugate to a self-similar IFS. This is Claim \ref{Claim JLMS}, that critically relies on the Poincar\'{e}-Siegel Theorem \cite[Theorem 2.8.2]{Katok1995Hass}. Corollary \ref{Main Corollary} then follows from Theorem \ref{Main Theorem} (if the IFS is not conjugate to linear), or from a result of Algom et al. \cite{Algom2023Wu} (if the IFS is analytically conjugate to self-similar).


\begin{Remark} \label{Rmk BS23} Let us now revisit the parallel project of Baker and Sahlsten \cite{Baker2023Sahl}  where they obtain similar results to ours. Roughly speaking, they extend the  arguments of Sahlsten-Stevens \cite{sahlsten2020fourier} that rely on the additive combinatorial approach of Bourgain and Dyatlov \cite{Bour2017dya}. This is a significant difference between the two papers, as we make no use of additive combinatorics in our work.   Baker and Sahlsten \cite{Baker2023Sahl} also require a spectral gap type estimate for the transfer operator without an underlying Markov partition. To this end, they utilize a disintegration technique as in \cite{Algom2022Baker}, and combine it with the proof outline of Naud \cite{Naud2005exp}. This has some similarities to our  corresponding argument. However, even in this part there are some essential differences; We invite the interested reader to compare our Section \ref{Section non-linearity implies s-gap} with \cite[Section 4]{Baker2023Sahl}.

\end{Remark}

\section{From non-linearity to spectral gap} \label{Section non-linearity implies s-gap} The main goal of this Section is to establish the spectral gap-type estimate Theorem \ref{Theorem spectral gap}. To do this, we  construct our random model in Theorem \ref{Theorem disint} in Section \ref{Section dis}, based on  Claim \ref{Claim key} in Section \ref{Section induced IFS}. This construction requires moving to an induced IFS of higher generation, but standard considerations show that this is allowed in our setting. We then proceed to state Theorem \ref{Theorem spectral gap}, and then  reduce it to a randomized $L^2$-contraction estimate Proposition \ref{Prop. 5.3 Naud}, in Section \ref{Section reduction}. We then reduce Proposition \ref{Prop. 5.3 Naud} to the key Lemma \ref{Lemma 5.4 Naud} in Section \ref{Section key Lemma}, and spend the rest of the Section proving it. The proof of Lemma \ref{Lemma 5.4 Naud} can be seen as a randomized version of Naud's arguments in \cite[Sections 5, 6, and  7]{Naud2005exp}.
\subsection{An induced IFS} \label{Section induced IFS}
Fix a $C^{2} (\mathbb{R})$ IFS $\Phi = \lbrace f_1,...,f_n \rbrace$ and write $\mathcal{A}= \lbrace 1,...,n\rbrace$. We  assume without the loss of generality that  for every $1\leq a\leq n$ the map $f_a$ is a self map of $[0,1]$, and that
\begin{equation} \label{Assumption on endpoints}
K=K_\Phi \text{ does not contain the points } 0,1.
\end{equation}
Indeed, this follows since we may assume the IFS acts on an open interval, and since Fourier decay estimates like in Theorem \ref{Main Theorem} are invariant under conjugation of the IFS by a non-singular affine map. Let
$$\tilde\rho:= \sup_{f\in \Phi} ||f'||_\infty \in (0,1),$$
and let
$$\rho_{\min} := \min_{i\in \mathcal{A}, x\in [0,1]} |f_i ' (x)|.$$
Let us recall the bounded distortion property \cite[Lemma 2.1]{algom2020decay}: There is some $L=L(\Phi)$ such that for every $\eta \in \mathcal{A}^*$ we have
\begin{equation} \label{Eq bdd distortion}
L^{-1} \leq \frac{\left|f_\eta ' (x)\right|}{\left|f_\eta ' (y)\right|}\leq L \text{ for all } x,y\in [0,1].
\end{equation}
Next, note that
\begin{equation*} 
\sup_{x\in [0,1], \xi\in \mathcal{A}^\mathbb{N}, n\in \mathbb{N}}  \left| \frac{d}{dx}\left( \log f' _{\xi|_n} \right)(x) \right| = \sup_{x\in [0,1], \xi\in \mathcal{A}^\mathbb{N}, n\in \mathbb{N}} \left| \sum_{k=1} ^n \frac{f_{\xi_k}''\circ f_{\sigma^k(\xi)|_{n-k}}(x)\cdot f_{\sigma^k(\xi)|_{n-k}}'(x)}{f_{\xi_{k}} '(x)\circ f_{\sigma^k(\xi)|_{n-k}}(x)} \right|$$
$$ \leq \sup_{x\in [0,1], f\in \Phi} \left| \frac{f''(x)}{f'(x)}\right| \cdot \frac{1}{1-\tilde{\rho}}.
\end{equation*}
So, there is a uniform constant $\tilde{C}:=2\cdot \sup_{x\in [0,1], f\in \Phi} \left| \frac{f''(x)}{f'(x)}\right|$, such that, assuming as we may (see the next paragraph) that $\tilde{\rho}<\frac{1}{2}$,
\begin{equation} \label{Eq tilde C}
\sup_{x\in [0,1], \xi\in \mathcal{A}^\mathbb{N}, n\in \mathbb{N}}  \left| \frac{d}{dx}\left( \log f' _{\xi|_n} \right)(x) \right| \leq \tilde{C}.
\end{equation}

For all $N$ let $\Phi^N$ be the corresponding induced IFS of generation $N$, that is,
$$\Phi^N = \lbrace f_{\eta}:\, \eta\in \mathcal{A}^N\rbrace.$$
Since $\Phi^N$ is finite, we may order its functions ($f_1,f_2$ and so on) at our will. Note that the same constants $L$ from \eqref{Eq bdd distortion} and  $\tilde{C}$ from  \eqref{Eq tilde C} will work for the induced IFS $\Phi^N$ for all $N$.


Our first step is, assuming $\Phi$ is not conjugate to linear, to construct an induced IFS that has useful separation and non-linearity properties:
\begin{Claim} \label{Claim key}
Let $\Phi$ be a  $C^2([0,1])$ IFS that is not conjugate to linear. Then there exists $N\in \mathbb{N}$ such that $\Phi^N$ admits $f_1,f_2,f_3,f_4$  satisfying:
\begin{enumerate}
\item For every $k\geq 5$ there exists $i\in \lbrace 1,3\rbrace$ such that 
$$f_i ([0,1]) \cup  f_{i+1} ([0,1]) \cup f_k ([0,1])$$
is a disjoint union. In particular, the union $f_i ([0,1]) \cup  f_{i+1} ([0,1])$ is disjoint for $i=1,3$.

\item There exists  $m',m>0$ such that: For every $x\in [0,1]$, for both $i=1,3$,
$$m\leq \left|\frac{d}{dx}  \left(  \log f_{i } ' - \log f_{i+1} ' \right) \left(x\right) \right|\leq m'.$$

\item We have
$$m-2\cdot \tilde{C}\cdot  \tilde{\rho}^N >0.$$
\end{enumerate}
\end{Claim}
We first remark that the existence of  an upper bound $m'$ as in Part (2) holds for all $N$ and every $f_i,f_j \in \Phi^N$; This is a standard fact, see e.g. \cite[Section 4]{Naud2005exp} or \eqref{Eq tilde C} for similar estimates. We thus focus on the lower bound, which is much harder to obtain, in our proof of Claim \ref{Claim key}.

We  require the following Claim:
\begin{Claim} \label{Claim UNI}
 There exist some $c,m',N_0>0$ such that for all $n>N_0$ there exist $\xi,\zeta \in \mathcal{A}^n$  such that for  every $x\in [0,1]$,
$$c< \left|\frac{d}{dx}  \left(  \log f_{\xi }  ' - \log  f_{\zeta}   ' \right) \left(x\right) \right|\leq m'.$$
\end{Claim}
\begin{proof}
By \cite[proof of Claim 2.13]{algom2021decay}, if for all $\sigma$-periodic\footnote{Observe that for the proof of \cite[Claim 2.13]{algom2021decay} to follow through it is enough to consider only $\sigma$-periodic elements. Furthermore, for such elements the limits in \eqref{Eq AGY 2}  exist by \eqref{Eq tilde C}. We remark that  similarly, in \cite[Claim 2.12]{algom2021decay} and its proof it is enough to consider only $\sigma$-periodic elements.}  $\xi,\zeta \in \mathcal{A}^\mathbb{N}$ and $x\in K$ we have
\begin{equation} \label{Eq AGY 2}
\lim_n \frac{d}{dx} \log f_{\xi|_n} ' (x) = \lim_n \frac{d}{dx} \log f_{\zeta|_n} ' (x),
\end{equation}
then $\Phi$ is $C^2$ conjugate to a linear IFS. Thus, since $\Phi$ is not conjugate to linear, there exist some $c'>0$, $x_0 \in K$ and $\sigma$-periodic $\xi,\zeta \in \mathcal{A}^\mathbb{N}$ such that for infinitely many $n$,
$$c'< \left|\frac{d}{dx}  \left(  \log  f_{\xi|_n }  ' - \log   f_{\zeta|_n }   ' \right) \left(x_0\right) \right|.$$
Recalling our coding map \eqref{Eq coding}, let $\omega \in \mathcal{A}^\mathbb{N}$ be such that $x_\omega = x_0$. Using $\sigma$-periodicity, by the uniform convergence (\cite[Claim 2.12]{algom2021decay} or \eqref{Eq tilde C}) as $n\rightarrow \infty$   of 
$$ \left|\frac{d}{dx}  \left(  \log  f_{\xi|_n }  ' - \log   f_{\zeta|_n }   ' \right) \left(\cdot \right) \right|$$
 there is some $N_1$ and some $k=k(N_1,c')$ such that for every $n>N_1$
$$\left|\frac{d}{dx}  \left(  \log  f_{\xi|_{n} }  ' - \log  f_{\zeta|_{n} }  ' \right) \left( x_0 \right) \right| -\frac{c'}{2} \leq \left|\frac{d}{dx}  \left(  \log  f_{\xi|_{n} }  ' - \log  f_{\zeta|_{n} }  ' \right) \left( f_{\omega|_k}(x)\right) \right| $$
Let $n>\max\lbrace  N_1,N_0 \rbrace$, and let $k=k(N_1,c')$ be as above. Then for all $x\in [0,1]$, 
$$\frac{c'}{2} \cdot  \rho_{\min} ^k \leq  \left|\frac{d}{dx}  \left(  \log  f_{\xi|_{n} }  ' - \log  f_{\zeta|_{n} }  ' \right) \left( f_{\omega|_k}(x)\right) \right| \cdot \left| f_{\omega|_k}'(x) \right|$$
$$=  \left|\frac{d}{dx}  \left(  \log \left( f_{\xi|_{n} } \circ f_{\omega|_k} \right) ' - \log \left(  f_{\zeta|_{n} } \circ f_{\omega|_k}  \right) ' \right) \left(x\right) \right|.$$
So, for $c=\frac{c'}{2} \cdot  \rho_{\min} ^k$, the  words $\xi' = \xi|_{n}*\omega|_{k}$ and $\zeta' = \zeta|_{n} *\omega|_{k}$ satisfy the claimed bound from below for all $n>\max\lbrace  N_1,N_0 \rbrace$. 
\end{proof}
$$ $$

\noindent{ \textbf{Proof of Claim \ref{Claim key}, Case 1:}}   In this case we assume that, with the notations of Claim \ref{Claim UNI}, that  there exist arbitrarily  large $n$ such that the $\xi,\zeta \in \mathcal{A}^n$ as in Claim \ref{Claim UNI} satisfy
$$\text{dist} \left( f_{\xi} ([0,1]),\, f_{\zeta} ([0,1]) \right) > 3 (\tilde\rho)^n.$$
Note that there exist some $k$ and $\eta_1,\eta_2 \in \mathcal{A}^k$ such that
$$\text{dist} \left( f_{\eta_1} ([0,1]),\, f_{\eta_2} ([0,1]) \right) \geq \frac{\diam (K)}{2}.$$
For example, $\eta_1,\eta_2$ can be chosen to be the $k$-prefixes of  codes of the endpoints of the interval $\overline{\text{conv}(K)}$.  Clearly  every large enough $k$ will work. We thus assume that $n$ is large enough so that $n>N_0$ as in Claim \ref{Claim UNI}, the first displayed equation holds, and 
\begin{equation} \label{Eq new assumption}
2 \tilde\rho ^n < \diam (K)/2.
\end{equation}

. 

We now define $N=n+k$ and our maps in $\Phi^N$ via
$$f_1:= f_{\eta_1} \circ f_{\xi} ,\quad f_2 := f_{\eta_1} \circ f_{\zeta}, \quad f_3:=f_{\eta_2}\circ f_{\xi},\quad f_4: =  f_{\eta_2}\circ  f_{\zeta}.$$
Let us show that all parts of the Claim hold: First, for part (2),  for all $x\in [0,1]$, 
$$\left|\frac{d}{dx}  \left(  \log f_{1 } ' - \log f_{2} ' \right) \left(x\right) \right|= \left|\frac{d}{dx}  \left(  \log \left( f_{\eta_1} \circ f_{\xi} \right) ' - \log \left(  f_{\eta_1} \circ f_{\zeta} \right) ' \right) \left(x\right) \right|$$ 
$$ = \left|\frac{d}{dx}  \left(  \log  f_{\xi}'   - \log  f_{\zeta} '  + \log \left(  f_{\eta_1}' \circ f_{\xi} \right) - \log \left(  f_{\eta_1}' \circ f_{\zeta} \right)  \right)  \left(x\right) \right| $$
$$ \geq \left|\frac{d}{dx}  \left(  \log  f_{\xi}'   - \log  f_{\zeta} ' \right)  \left( x \right) \right|  - \left| \frac{d}{dx}  \left(  \log \left(  f_{\eta_1}' \circ f_{\xi} \right) - \log \left(  f_{\eta_1}' \circ f_{\zeta} \right)  \right)  \left(x\right) \right|.$$
Now, by  Claim \ref{Claim UNI} since $n>N_0$,
$$\left|\frac{d}{dx}  \left(  \log  f_{\xi}'   - \log  f_{\zeta} ' \right)  \left( x \right)\right| \geq c.$$
On the other hand, arguing similarly to \eqref{Eq tilde C},
$$\left| \frac{d}{dx}   \log \left(  f_{\eta_1}' \circ f_{\xi} \right)(x) \right|\leq \tilde{C}\cdot \left| f_{\xi}' (x) \right|\leq \tilde{C}\cdot \tilde{\rho}^n.$$
Since we may assume $n$ is also large depending on $\tilde{C},c$ (that are both known a-priori), we conclude that
$$\left|\frac{d}{dx}  \left(  \log f_{1 } ' - \log f_{2} ' \right) \left(x\right) \right| \geq c-2\cdot \tilde{C}\cdot \tilde{\rho}^n>0.$$
The same calculation holds for $f_3,f_4$. This shows Part (2) holds true for our functions with
$$m:=c-2\cdot \tilde{C}\cdot \tilde{\rho}^n.$$
Part (3) essentially follows from the same argument since, assuming $n$ is sufficiently large depending on $c,\tilde{C}$, we can get that also
$$c-2\cdot \tilde{C}\cdot \tilde{\rho}^n -2\cdot\tilde{C}\cdot  \tilde{\rho}^{n+k}>0.$$

For Part (1), for every $x\in f_1([0,1])$ and $y\in f_2 ([0,1])$  there is some $z$ with
$$|f_1(x)-f_2(y)| = |f_{\eta_1} ' (z)| \cdot |  f_{\xi}(x)- f_{\zeta}(y)| \geq  \rho_{\min} ^k \cdot \text{dist} \left( f_{\xi} ([0,1]),\, f_{\zeta} ([0,1]) \right)$$
$$   \geq   \rho_{\min} ^k \cdot 3\cdot \tilde\rho^{n} > 0.$$
This shows that
\begin{equation*} 
\text{dist} \left( f_{1} ([0,1]),\, f_{2} ([0,1]) \right), \quad \text{dist} \left( f_{3} ([0,1]),\, f_{4} ([0,1]) \right) > \rho_{\min} ^k \cdot 3\cdot \tilde\rho^{n}.
\end{equation*}
Next, by the choice of $\eta_1,\eta_2$ and $f_1,f_2,f_3,f_4$ it is clear that
$$\text{dist} \left( f_{1} ([0,1]),\, f_{3} ([0,1]) \right),\, \text{dist} \left( f_{1} ([0,1]),\, f_{4} ([0,1]) \right),\, \text{dist} \left( f_{2} ([0,1]),\, f_{3} ([0,1]) \right)\, \text{dist} \left( f_{2} ([0,1]),\, f_{4} ([0,1]) \right)$$
$$ \geq \frac{\diam (K)}{2}.$$

Now, fix any $f_\ell$ with $\ell\geq 5$. If 
$$f_1 ([0,1]) \cup  f_{2} ([0,1]) \cup f_\ell ([0,1])$$
is a disjoint union then we are done. Otherwise, suppose
$$f_1 ([0,1]) \cap f_\ell ([0,1]) \neq \emptyset.$$
Then,  since  $\text{diam} f_\ell ([0,1]) \leq (\tilde\rho)^N$ and by \eqref{Eq new assumption},
$$ \text{dist} \left( f_{\ell} ([0,1]),\, f_{3} ([0,1]) \right),\quad  \text{dist} \left( f_{\ell} ([0,1]),\, f_{4} ([0,1]) \right)\geq \frac{\diam (K)}{2}-2(\tilde\rho)^N >0.$$
So, since the unions $f_1 ([0,1]) \cup  f_{2} ([0,1])$ and $f_3 ([0,1]) \cup  f_{4} ([0,1])$ have already been shown to be disjoint, the union
$$f_3 ([0,1]) \cup  f_{4} ([0,1]) \cup f_\ell ([0,1])$$
is disjoint. The other cases are similar.  The proof of Claim \ref{Claim key} in Case 1 is complete.
$$ $$
\noindent{ \textbf{Proof of Claim \ref{Claim key}, Case 2:}} Assume now that  for all $n>N_0$,  for $\xi,\zeta \in \mathcal{A}^n$ as in Claim \ref{Claim UNI},
$$\text{dist} \left( f_{\xi} ([0,1]),\, f_{\zeta} ([0,1]) \right) \leq 3  (\tilde\rho)^n.$$
Let us find some  $\omega_1,\omega_2 \in \mathcal{A}^\mathbb{N}$ such that for all $k$
\begin{equation} \label{Eq omega1 and omega2}
\text{dist} \left( f_{\omega_1|_k} ([0,1]),\, f_{\omega_2|_k} ([0,1]) \right) > 3 (\tilde\rho)^k.
\end{equation}
Such  $\omega_1,\omega_2 $ must exist since $K_\Phi$ is infinite. 



Then for all large $n'>N_0$ and $\xi,\zeta \in \mathcal{A}^{n'}$  as in Claim \ref{Claim UNI}, for all large $n$ for all $x\in [0,1]$, we have
$$ \left|\frac{d}{dx}  \left(  \log \left( f_{\omega_1|_n} \circ f_{\xi} \right) ' - \log \left( f_{\omega_2|_n} \circ f_{\zeta} \right) ' \right) \left(x\right) \right|$$
$$ =  \left|\frac{d}{dx}  \left(  \log f_{\omega_1|_n}' \circ f_{\xi}   - \log  f_{\omega_2|_n} ' \circ f_{\zeta} \right) \left(x \right) -  \frac{d}{dx}  \left(  \log  f_{\zeta }  ' - \log  f_{\xi } ' \right)  \left(x\right)    \right|$$
$$\geq  \left| \frac{d}{dx}  \left(  \log  f_{\zeta }  ' - \log  f_{\xi } ' \right) \left(x\right) \right|   - \left|\frac{d}{dx}  \left(  \log f_{\omega_1|_n}' \circ f_{\xi}   - \log  f_{\omega_2|_n} ' \circ f_{\zeta} \right) \left(x \right) \right|  $$
$$\geq c - 2\tilde{C}\cdot \tilde{\rho}^{n'} = c -o(1).$$

Thus,   $\xi'' = \omega_1|_n *\xi$  and $\zeta'' =  \omega_2|_n *\zeta$ are words that can be made arbitrarily long, that satisfy that for some $c'>0$ and for all $x$
$$ \left|\frac{d}{dx}  \left(  \log f_{\xi''}  ' - \log   f_{\zeta''}  ' \right)  \left(x\right) \right|>c'.$$
Also,  by \eqref{Eq omega1 and omega2}, 
$$\text{dist} \left( f_{\xi'' } ([0,1]),\, f_{\zeta''} ([0,1]) \right) > 3 (\tilde\rho)^n > 3 (\tilde\rho)^{n+n'}.$$
Thus, we are back in Case 1 with these $\xi'',\zeta''$ (that we recall can be made arbitrarily long) and $c'$.

Thus, we have completely reduced Case 2 to Case 1, completing the proof of Claim \ref{Claim key}. \hfill{$\Box$}



\subsection{The Euclidean derivative cocycle and associated transfer operator} \label{Section pre}
Fix a $C^{2}$ IFS $\Phi$ as in Section \ref{Section induced IFS}, and let us retain the other  assumptions and notations from that Section. Let $\textbf{p}=(p_1,...,p_n)$ be a strictly positive probability vector on $\mathcal{A}$, and let $\mathbb{P}=\mathbf{p}^\mathbb{N}$ be the corresponding product measure on $\mathcal{A}^\mathbb{N}$. Let $\nu=\nu_\mathbf{p}$ be the  corresponding self-conformal measure.

Note that for every $P$ the induced IFS $\Phi^P$ has the same attractor $K$ as $\Phi$. That is, $K_{\Phi^P} = K_\Phi$. Furthermore, $\nu$ is also a self-conformal measure with respect $\Phi^P$ and the induced probability vector $\mathbf{p}^P$ on $\mathcal{A}^P$. 

Thus, by working with the induced IFS $\Phi^N$ as in Claim \ref{Claim key}, we may assume without the loss of generality that our (not conjugate to linear) original IFS $\Phi$ already admits  $f_1,f_2,f_3,f_4 \in \Phi$ that satisfy:
\begin{enumerate}
\item For every $k\geq 1$ there exists $i\in \lbrace 1,3\rbrace$ such that 
\begin{equation} \label{Add property 1}
f_i ([0,1]) \cup  f_{i+1} ([0,1]) \cup f_k ([0,1]).
\end{equation}
is a disjoint union. In particular, the union $f_i ([0,1]) \cup  f_{i+1} ([0,1])$ is disjoint for $i=1,3$.

\item There exists  $m',m>0$ such that: For every $x\in [0,1]$, for both $i=1,3$,
\begin{equation} \label{Add property 2}
m\leq \left|\frac{d}{dx}  \left(  \log f_{i } ' - \log f_{i+1} ' \right) \left(x\right) \right|\leq m',\text{ and } m-2\cdot \tilde{C}\cdot \sup_{f\in \Phi} ||f'||_\infty>0.
\end{equation}
\end{enumerate}
Note that for the second part of \eqref{Add property 2} we use Claim \ref{Claim key} Part (3) and that, with the notation $\tilde{\rho}$ as in Section \ref{Section induced IFS},
$$ \sup_{f\in \Phi^N} ||f'||_\infty \leq \tilde{\rho}^N.$$

Let  $G$ to be the free semigroup generated by the family $\{f_a: 1\leq a\leq n\}$, which acts on $[0,1]$ by composing the corresponding $f_a$'s. We define the derivative cocycle $c:G\times [0,1]\rightarrow \mathbb{R}$ via
\begin{equation} \label{The der cocycle}
c(I,x)=-\log \left| f'_I (x) \right|.
\end{equation}
Slightly adjusting our notation from Section \ref{Section induced IFS}, let 
$$\rho:= \sup_{f\in \Phi} ||f'||_\infty \in (0,1).$$ Also, we note that by uniform contraction
\begin{equation} \label{Eq C and C prime}
0<D:= \min \lbrace -\log |f' (x)| : f\in \Phi, x\in I \rbrace, \quad D':= \max \lbrace -\log |f' (x)| : f\in \Phi, x\in I \rbrace < \infty.
\end{equation}
A major part in our argument is played by the transfer operator:
 \begin{Definition} \label{Def transfer operator}
 For every $s \in \mathbb{C}$ such that $|\Re(s)|$ is small enough, let $P_{s}:C^1([0,1]) \rightarrow C^1([0,1])$ denote the transfer operator defined by, for $g \in C^1$ and $x \in [0,1]$,
$$P_{s} (g)(x) = \int e^{2\pi \cdot s\cdot  c(a,x)} g\circ f_a(x)\, d\mathbf{p}(a). $$
 \end{Definition}
Note that $\nu$ is the unique stationary measure corresponding to the measure $\sum_{a\in \mathcal{A}} p_a \cdot \delta_{\lbrace f_a \rbrace}$ on $G$. So, the detailed discussion about this operator as in  \cite[Section 11.5]{Benoist2016Quint}   applies in the setting we are considering.

\subsection{Disintegration of $\nu$ and the transfer operator} \label{Section dis} 
We can now construct the random measures  we discussed in Section \ref{Section method}. We begin by recalling the notion of a model (as in e.g. \cite{Shmerkin2016Solomyak}):  Let $I$ be a finite set of $C^2 (\mathbb{R})$ iterated function systems $\lbrace f_1^{(i)},...,f_{k_i} ^{(i)} \rbrace, i\in I$.  Let $\Omega = I^\mathbb{N}$. For $\omega \in \Omega$ and $n\in \mathbb{N}\cup \lbrace \infty \rbrace$ let
$$X_n ^{(\omega)} := \prod_{j=1} ^n \lbrace 1,...,k_{\omega_j} \rbrace.$$
Define a coding map $\Pi_\omega : X_\infty ^{(\omega)} \rightarrow \mathbb{R}$ via
$$\Pi_\omega (u) = \lim_{n\rightarrow \infty} f^{ (\omega_1)} _{u_1} \circ \circ \circ f^{ (\omega_n)} _{u_n} (0).$$
Next, for each $i\in I$ let $\mathbf{p}_i = (p_1 ^{(i)},...,p_{k_i} ^{(i)})$ be a probability vector with strictly positive entries. On each $X_\infty ^{(\omega)}$ we define the product measure
$$\eta^{(\omega)} = \prod_{n=1} ^\infty \mathbf{p}^{(\omega_n)}.$$
We can now define
$$\mu_\omega:= \Pi_\omega \left( \eta^{(\omega)} \right). $$
Letting $\sigma:\Omega \rightarrow \Omega$ be the left shift, we have the following dynamical self-conformality relation:
\begin{equation} \label{Eq stochastic self-similarity}
\mu_\omega = \sum_{u\in X_1 ^{(\omega)}} p_u ^{(\omega_1)} f_u ^{(\omega_1)} \mu_{\sigma(\omega)}.
\end{equation}
We also define
$$K_\omega := \supp(\mu_\omega) = \Pi_\omega \left( X_\infty ^{(\omega)} \right),$$
and note that here we have
$$K_\omega = \bigcup_{u\in  X_1 ^{(\omega)}}  f_u ^{(\omega_1)} K_{\sigma(\omega)}.$$

Next, for $\omega\in \Omega, N\in \mathbb{N}$, and $s\in \mathbb{C}$ we define an operator $P_{s,\omega,N}:C^1\left([0,1]\right)\rightarrow C^1\left([0,1]\right)$ by
\begin{equation} \label{Eq dis operator}
P_{s,\omega,N} \left( g \right) (x):= \sum_{I\in X_N ^{(\omega)}} \eta ^{(\omega)} (I) e^{2\pi s c(I,x)} g\circ f_I (x), \quad \text{ where } \eta ^{(\omega)} (I) := \prod_{n=1} ^{|I|} \textbf{p}^{(\omega_n)}_{  I_n }.
\end{equation}
Note that we only need to know $\omega|_{N}$ in order for $P_{s,\omega,N}$ to be well defined. 
Iterating \eqref{Eq stochastic self-similarity} we have the following equivariance relations, whose proof is left to the reader: First, for every $\omega\in \Omega$, $N\in \mathbb{N}$, and $g\in C^1\left([0,1]\right)$,
\begin{equation} \label{Eq equivariance}
\int g(x) \, d\mu_{\omega}(x) =\int P_{0,\omega,N} \left( g \right) (x)\, d\mu_{\sigma^N \omega} (x).
\end{equation}
Furthermore, for all integers $0\leq \tilde{n}\leq n, N,$ we have that
\begin{equation} \label{Eq equivariance 2}
P_{s,\omega, nN} \left( g \right) (x) =   P_{s, \sigma^{\tilde{n}N} \omega,  Nn-N\tilde{n}}  \left( P_{s,\omega, \tilde{n}N} \left( g \right)  \right) \left( x \right).
\end{equation}

Finally, let $\mathbb{Q}$ be a $\sigma$-invariant measure on $\Omega$. The triplet $\Sigma = (\Phi_{i\in I} ^{(i)}, (\mathbf{p}_i)_{i\in I}, \mathbb{Q})$ is called a model. We say the model is non-trivial if $\mu_\omega$ is non-atomic for $\mathbb{Q}$-a.e. $\omega$. We say it is Bernoulli if $\mathbb{Q}$ is a Bernoulli measure. We are now ready to state the main result of this Section:
\begin{theorem} \label{Theorem disint}
Let $\Phi$ be a  $C^2 ([0,1])$ IFS satisfying properties \eqref{Add property 1} and \eqref{Add property 2}, and let $\nu=\nu_{\mathbf{p}}$ be a self-conformal measure. Then there exists a non-trivial Bernoulli model $\Sigma=\Sigma(\Phi, \nu)$ such that:
\begin{enumerate}
\item (Disintegration of measure) $\nu = \int \mu_{(\omega)} d\mathbb{Q}(\omega)$.

\item (Disintegration of operator) For every $N\in \mathbb{N},s\in \mathbb{C}, f\in C^1([0,1]),$ and $x\in [0,1]$  we have 
\begin{equation} \label{Dis of transfer}
P_s ^{N} \left(f \right)(x) = \sum_{\omega \in \Omega^{N}} \mathbb{Q}\left([\omega]\right) P_{s,\omega, N} \left( f \right)(x).
\end{equation}

\item (Non-trivial  branching) For every $i\in I$ we have $k_i =2$ or $k_i=3$.

\item (Separation) For every $\omega \in \Omega$ the union
$$\bigcup_{ u\in X_1 ^{(\omega)}} f_u ^{(\omega_1)} ([0,1])$$
is disjoint.

\item (UNI in all parts)  There exist $m',m>0$ and $N_0\geq 0$ such that for all $N\geq N_0$, for every $\omega \in \Omega$ there exist $\alpha_1 ^N, \alpha_2 ^N \in X_N ^{(\omega)}$ such that
 
$m \leq \left|\frac{d}{dx}  \left(  \log f_{\alpha_1 ^N } ' - \log f_{\alpha_2 ^N } ' \right) \left(x\right) \right| \leq m', \quad \text{ for all } x\in [0,1]$.

\item (Federer property) For every $D>1$ there exists $C_D=C_D(\Sigma)>0$ such that:

For every $\omega\in \Omega$, for every $x\in \supp (\mu_\omega)$ and $r>0$,
$$\mu_{(\omega)}  \left( B(x, Dr) \right) \leq C_d \mu_{(\omega)}  \left( B(x, r) \right).$$
\end{enumerate}
\end{theorem}
Recall that our IFS $\Phi$ meets the conditions of Theorem \ref{Theorem disint} courtesy of Claim \ref{Claim key}. The proof of Theorem \ref{Theorem disint} is given in the next three subsections.  
\subsubsection{Construction of the model}
Recall that we are assuming conditions \eqref{Add property 1} and \eqref{Add property 2} hold for $\Phi$.  For $i=1,2,3,4$ we define the IFSs
$$\Phi_1  = \Phi_2= \lbrace f_1, f_2 \rbrace, \text{ and } \Phi_3=\Phi_4= \lbrace f_3, f_4 \rbrace.$$
For every $k\geq 5$ define the IFS 
$$\Phi_k  = \lbrace f_i,f_{i+1},f_k \rbrace$$ 
as in \eqref{Add property 1}.

 We now define
$$I=\lbrace  \Phi_i:\, i\in \mathcal{A} \rbrace.$$
Thus, for every $i\in I$ we associate the IFS $\Phi_i$. 

Note that  certain $f_i\in \Phi$ appear in multiple IFS's in $I$. Write, for $i\in \mathcal{A}$,
$$n_i = \left| \lbrace \Phi_j:\, f_i \in \Phi_j \rbrace \right| .$$
Recall that $\nu=\nu_{\mathbf{p}}$ where $\mathbf{p}\in \mathcal{P}(\mathcal{A})$. We now define a probability vector $\mathbf{q}$ on $I$ as follows:  
$${\mathbf{q}}_1 = {\mathbf{q}}_2= \frac{p_1}{n_1}+ \frac{p_{2}}{n_{2}},\quad  {\mathbf{q}}_3 = {\mathbf{q}}_4= \frac{p_3}{n_3}+ \frac{p_{4}}{n_{4}},\quad {\mathbf{q}}_k= \frac{p_i}{n_i}+ \frac{p_{i+1}}{n_{i+1}}+\frac{p_{k}}{n_{k}} \text{ for }k\geq 5.$$
Let $\mathbb{Q}=\mathbf{q}^\mathbb{N}$. This will be our Bernoulli selection measure on $\Omega^\mathbb{N}$.

Next, for every $i\in I$ we define the probability vector
$$\tilde{\mathbf{p}}_1 = \tilde{\mathbf{p}}_2  = \left( \frac{\frac{p_1}{n_1}}{\frac{p_1}{n_1}+ \frac{p_{2}}{n_{2}} }, \, \frac{\frac{p_{2}}{n_{2}}}{\frac{p_1}{n_1}+ \frac{p_{2}}{n_{2}} }\right), \quad \tilde{\mathbf{p}}_3 = \tilde{\mathbf{p}}_4  = \left( \frac{\frac{p_3}{n_3}}{\frac{p_3}{n_3}+ \frac{p_{4}}{n_{4}} }, \, \frac{\frac{p_{4}}{n_{4}}}{\frac{p_3}{n_3}+ \frac{p_{4}}{n_{4}} }\right), $$
$$ \tilde{\mathbf{p}}_k = \left( \frac{\frac{p_i}{n_i}}{\frac{p_i}{n_i}+ \frac{p_{i+1}}{n_{i+1}}+\frac{p_{k}}{n_{k}} }, \, \frac{\frac{p_{i+1}}{n_{i+1}}}{\frac{p_i}{n_i}+ \frac{p_{i+1}}{n_{i+1}}+\frac{p_{k}}{n_{k}} }, \, \frac{\frac{p_{k}}{n_{k}}}{\frac{p_i}{n_i}+ \frac{p_{i+1}}{n_{i+1}}+\frac{p_{k}}{n_{k}} }\right)  \text{ for } k\geq 5.$$
Our model is now fully defined.

\subsubsection{Proof of Parts (1)-(5) of Theorem \ref{Theorem disint}}

\noindent{ \textbf{Proof of Part (1)}} We first argue that
$$\sum_{i\in \mathcal{A}} f_i p_i \cdot  \int \mu_{(\omega)}\, d\mathbb{Q}(\omega) = \int \mu_{(\omega)}\, d\mathbb{Q}(\omega).$$
The proof is almost entirely the same as \cite[Section 2.2]{Algom2022Baker}, so we omit the details. Part (1) now follows since $\nu$ is the unique measure satisfying this identity.
$$ $$
\noindent{ \textbf{Proof of Parts (2) and (3)}} These are straightforward given our construction. Indeed, for part (2), when $N=1$ it suffices to note that for every $i\in \mathcal{A}$,
$$ \sum_{j:\, f_i\in \Phi_j} {\mathbf{q}}_j \cdot \mathbf{p} _j (i) = n_i \cdot \frac{p_i}{n_i} = p_i.$$ 
For general $N$ similar considerations apply.
$$ $$
\noindent{ \textbf{Proof of Part (4)}} This follows directly from our construction and from Claim \ref{Claim key} part (1).
$$ $$
\noindent{ \textbf{Proof of Part (5)}} Let $\omega \in \Omega$ and let $N_0=1$. Choose any $N>N_0$. By \eqref{Add property 2} and our construction, there exist  $m',m>0$ such that for   some $i,j \in \Phi_{\omega_N}$  with 
$$m\leq \left|\frac{d}{dx}  \left(  \log f_{i} ' - \log f_{j } ' \right) \left(x\right) \right|\leq m', \text{ for every } x\in [0,1].$$
Let $\xi \in X^{(\omega)} _{N-1}$, and put $\alpha_1 := \xi*i, \alpha_2:=\xi*j\in X^{(\omega)} _{N}$. Then  for all $x\in [0,1]$, 
$$\left|\frac{d}{dx}  \left(  \log f_{\alpha_1 } ' - \log f_{\alpha_2} ' \right) \left(x\right) \right|= \left|\frac{d}{dx}  \left(  \log \left( f_{\xi} \circ f_{i} \right) ' - \log \left(  f_{\xi} \circ f_{j} \right) ' \right) \left(x\right) \right|$$ 
$$ = \left|\frac{d}{dx}  \left(  \log  f_{i}'   - \log  f_{j} '  + \log \left(  f_{\xi}' \circ f_{i} \right) - \log \left(  f_{\xi}' \circ f_{j} \right)  \right)  \left(x\right) \right| $$
$$ \geq \left|\frac{d}{dx}  \left(  \log  f_{i}'   - \log  f_{j} ' \right)  \left( x \right) \right|  - \left| \frac{d}{dx}  \left(  \log \left(  f_{\xi}' \circ f_{i} \right) - \log \left(  f_{\xi}' \circ f_{j} \right)  \right)  \left(x\right) \right|.$$
By arguing similarly to \eqref{Eq tilde C},
$$\left| \frac{d}{dx}   \log \left(  f_{\xi}' \circ f_{i} \right)(x) \right|\leq \tilde{C}\cdot \left| f_{i}' (x) \right|\leq \tilde{C}\cdot \rho.$$
By the choice of $i,j$ we conclude that
$$\left|\frac{d}{dx}  \left(  \log f_{\alpha_1 } ' - \log f_{\alpha_2} ' \right) \left(x\right) \right|  \geq  m-2\cdot \tilde{C}\cdot \rho>0,$$
where the last inequality is due to the second part of \eqref{Add property 2}. \hfill{$\Box$}

\subsubsection{Proof of Part (6) of Theorem \ref{Theorem disint}}
This part of the proof is modelled after Naud's work in \cite[Section 6]{Naud2005exp}. Similarly to Naud,  we require the following Lemmas and definitions, that will also be used elsewhere in this note. However, unlike Naud \cite{Naud2005exp}, the Federer property we establish is for the \textit{random} measures in our model. Thus, we make sure that all our estimates are \textit{uniform in} $\omega$ (i.e. depend only on the model and not the measure under consideration).  
\begin{Definition} \label{Def cylinders}
Fix $\omega\in \Omega$ and $n\in \mathbb{N}$. The cylinder that corresponds to $u \in X_n ^{(\omega)}$ is the set
$$C_u := f^{ (\omega_1)} _{u_1} \circ \circ \circ f^{ (\omega_n)} _{u_n} ([0,1]) \subseteq [0,1].$$
\end{Definition}
From now on we fix the $\omega$ in question and suppress it in our notation. Also, note that by the definition of $\mu_\omega$, \eqref{Eq stochastic self-similarity}, and  Theorem \ref{Theorem disint} Part (4),
\begin{equation} \label{Eq measure of cylinder}
\mu_\omega (C_u) = \mu_\omega ( f^{ (\omega_1)} _{u_1} \circ \circ \circ f^{ (\omega_n)} _{u_n} ([0,1]) ) = \eta^{(\omega)} ([u_1,...,u_n]) = \prod_{i=1} ^n \mathbf{p}^{(\omega_i)} _{u_i}.
\end{equation}
For a cylinder set $C_\alpha$ let $|C_\alpha|$ denote its diameter. Recall that we are fixing some $\omega \in \Omega$, and considering cylinders with respect to $K_\omega$. 
\begin{Lemma} \label{Lemma 6.1 Naud}
There exist constants $C>0$ and $0<\delta_1,\delta_2<1$ uniform in $\omega$ such that for all cylinders $C_\alpha \subseteq C_\beta$
$$C^{-1} \delta_1 ^{|\alpha|-|\beta|} \leq \frac{|C_\alpha|}{|C_\beta|} \leq C\cdot \delta_2 ^{|\alpha|-|\beta|}.$$
\end{Lemma}
\begin{proof}
Write $\alpha=\beta\cdot u$ where $|u|=|\alpha|-|\beta|$. Then, omitting superscripts, for some $x_0$
$$|C_\alpha| = |f_\beta \circ f_u ([0,1])| = \left|f_\beta ' \circ f_u(x_0) \cdot f_u' (x_0)\right|$$
and for some $y_0$ we have
$$|C_\beta| = |f_\beta([0,1])| = \left|f_\beta ' (y_0) \right|.$$
Letting $L$ be as in \eqref{Eq bdd distortion}, recalling that $f_\beta$ is a composition of maps from $\Phi$,
$$L^{-1} \left( \min_{i\in I, f \in \Phi^{(i)}} \min_{x\in [0,1]} \left| f_i ' (x)\right| \right)^{|u|} \leq  L^{-1} \cdot \left|f_u' (x_0)\right| \leq  \frac{|C_\alpha|}{|C_\beta|} \leq L \cdot \left|f_u' (x_0)\right|\leq L\cdot \rho^{|u|},$$
as claimed.
\end{proof}

\begin{Lemma} \label{Lemma 6.2 Naud}
For every cylinder $C_\beta$ we have
$$C_\beta \cap K_\omega \subseteq \bigcup_{C_\alpha \subseteq C_\beta} C_\alpha,$$
where $|\alpha|=|\beta|+1$. Moreover, there is some $\lambda>0$ uniform in $\omega$ such that for any two distinct cylinders $C_{\alpha_i} \subseteq C_\beta$ with $i=1,2$ and $|\alpha_i|=|\beta|+1$ we have
$$\text{dist} \left( C_{\alpha_1},\, C_{\alpha_2} \right) \geq \lambda\cdot |C_\beta|.$$
\end{Lemma}
\begin{proof}
For the first assertion,  for every $n\in \mathbb{N}$ we have by definition
$$K_\omega = \bigcup_{u\in X_n ^{(\omega)}} f_u ^{(\omega)} K_{\sigma^n(\omega)}, \quad  K_\omega\subseteq K \subseteq [0,1] \text { for every } \omega, \text{ and } \quad K_{\sigma^n (\omega)}  = \bigcup_{i\in X^{(\sigma^n \omega)} _{1}} f^{ (\omega_{n+1})} _{i}  ( K_{\sigma^{n+1} (\omega)}).$$
By Theorem \ref{Theorem disint} part (4), 
$$C_\beta \cap K_\omega = f_{\beta} ^{(\omega)} ([0,1]) \cap  f_{\beta} ^{(\omega)} (K_{\sigma^n (\omega)}) =f_{\beta} ^{(\omega)} (K_{\sigma^n (\omega)}) .$$
So, 
$$C_\beta \cap K_\omega \subseteq   \bigcup_{i\in X^{(\sigma^n \omega)} _{1}} f_\beta ^{(\omega)} \circ f^{ (\omega_{n+1})} _{i} \circ ( K_{\sigma^{n+1} (\omega)}) \subseteq  \bigcup_{i\in X^{(\sigma^n \omega)} _{1}} f_\beta ^{(\omega)} \circ f^{ (\omega_{n+1})} _{i} ([0,1])=\bigcup_{C_\alpha \subseteq C_\beta} C_\alpha.$$

As for the second assertion, write $\alpha_1 = \beta\cdot i$ and   $\alpha_2 = \beta\cdot j$, where $i,j\in \Phi_{\omega_{|\beta|+1}}$. Then for any $x,y\in [0,1]$ there is some $z\in [0,1]$ such that
$$f_{\alpha_1} (x)- f_{\alpha_2 }(y) = f_\beta ( f_i(x)) - f_\beta (f_{j} (y)) = f_\beta ' (z)\cdot (f_i(x)-f_{j}(y)).$$
So, by bounded distortion \eqref{Eq bdd distortion},
$$ \left| f_{\alpha_1} (x)- f_{\alpha_2 }(y) \right| \geq L \cdot |C_\beta| \cdot  \text{dist} \left( C_{i},\, C_{j} \right) \geq L \cdot |C_\beta| \cdot \min_{s\in I, f,g\in \Phi_s}  \text{dist} \left( f([0,1]),\, g([0,1]) \right),$$
and by Theorem \ref{Theorem disint} part (4)
$$\min_{s\in I, f,g\in \Phi_s}  \text{dist} \left( f([0,1]),\, g([0,1]) \right)>0.$$
As required.
\end{proof}

\noindent{ \textbf{Proof of Theorem \ref{Theorem disint} Part (6)}} Fix $r>0$, $D>1$ and $x\in K_\omega$. In general, we aim to show that there exist cylinders $C_\alpha \subseteq C_\beta$ such that $|\alpha|-|\beta|$ depends only on $D$, and
$$C_\alpha \subseteq B(x,r) \text{ and } B(x, Dr)\cap K_\omega \subseteq C_\beta.$$
If this holds then by \eqref{Eq measure of cylinder}:
$$ \frac{\mu_\omega (B(x, Dr) }{\mu_\omega (B(x, Dr)} \leq \frac{\mu_\omega (C_\beta)}{\mu_\omega (C_\alpha)} \leq (\min_{j\in I} \min_{i\in \lbrace 1,...,k_{j} \rbrace} {p}^{(j)} _ {i} )^{ |\beta| -|\alpha|}.$$
Note that the latter bound depends on $D$, and in particular is uniform in $\omega$.

Set $J=B(x,r)$ and $J' = B(x, Dr)$. We first assume $J'\cap K_\omega \subsetneq C_i$ for some 1st generation cylinder of $K_\omega$. Set 
$$n = \min\lbrace j\geq 1:\, \exists C_\alpha \subset J',\, |\alpha|=j\rbrace.$$
Then $n\geq 2$. Let $C_\alpha \subset J'$ be such that $|\alpha|=n$. Let $C_\alpha \subset C_{\alpha'}$ with $|\alpha'|=n-1$. By definition of $n$ $C_{\alpha'}\not \subset J'$, so there are two options: 
\begin{enumerate}
\item If $J'\cap K_\omega \subset C_{\alpha'}$ then by Lemma \ref{Lemma 6.1 Naud} we have
$$C^{-1} \delta_1 \left|C_{\alpha'}\right|\leq \left| C_\alpha \right| \leq \left| J' \right|.$$

\item Otherwise, there exists a cylinder $C_{\beta'}$ such that $|\beta'|=n-1$ such that $C_{\alpha'}$ and $C_{\beta'}$ are consecutive and $J'\cap K_\omega \subseteq C_{\alpha'}\cup C_{\beta'}$. Indeed, $K_\omega \cap J'$ is covered by such cylinders of generation $n-1$ and none of them are included in $J'$. Consider now a bigger cylinder $C_{\gamma}$ such that $C_{\alpha'}\cup C_{\beta'}\subset C_{\gamma}$, and assume that
$$|\gamma| = \max \lbrace j\geq 0:\, \exists C_{\beta} \supset C_{\alpha'}\cup C_{\beta'},\, |\beta|=j\rbrace.$$
The maximality of $|\gamma|$ implies that $C_{\alpha'}\subset C_{\alpha''}$ and $C_{\beta'}\subset C_{\beta''}$ where $|\alpha''|=|\beta''|=|\gamma|+1$. Since the gap between $C_{\alpha'}$ and $C_{\beta'}$ is included in $J'$, by Lemma \ref{Lemma 6.2 Naud} we find that
$$|J'|\geq \text{dist} (C_{\alpha'}, C_{\beta'}) \geq \text{dist} (C_{\alpha''}, C_{\beta''}) \geq \lambda |C_{\gamma}|.$$
\end{enumerate}
We have just shown that there exists a cylinder $C_\beta$ such that $J'\cap K_\omega \subseteq C_\beta$ and $|C_\beta|\leq C'\cdot |J'|$, where $C'$ is independent of $r,x,D$.

Finally, there exists a decreasing sequence of cylinders $C_{\gamma_i}$ such that 
$$x\in C_{\gamma_i} \subsetneq ...\subsetneq C_{\gamma_1} \subsetneq C_{\gamma_0} \subsetneq C_\beta.$$
By Lemma \ref{Lemma 6.1 Naud} and the estimate on $|C_\beta|$ we have
$$|C_{\gamma_i}| \leq C\cdot C' \cdot \delta_2 ^{|\gamma_i|-|\beta|}\cdot Dr.$$
So, whenever $C\cdot C' \cdot \delta_2 ^{|\gamma_i|-|\beta|}\cdot D <1$ we have $C_{\gamma_i} \subseteq J$. This is the same as asking that 
$$|\gamma_i|-|\beta|> \frac{\log D CC'}{\log \delta_2}.$$ Hence there is a cylinder $C_\gamma = C_{\gamma_i}$ as required, such that $C_\alpha \subset J$ and $|\alpha|-|\beta|$ only depends on $D$.

Finally, we discuss the case when $J'\cap K_\omega $ is not included in a first generation cylinder. Note that $x\in C_i$ for some $i\in \lbrace 1,...,k_{\omega_1} \rbrace$. If $C_i \subseteq J'$ then $|C_i|\leq |J'|$ and following the same ideas we find $C_\alpha \subset J$ and $C_\alpha \subset C_i$ with $|\alpha|$ only depending on $D$.  So,
$$\frac{\mu_\omega (J')}{\mu_\omega (J)} \leq \frac{1}{\mu_\omega(C_\alpha)} \leq C_D ',$$
where $C_D '$ only depends on $D$ (in particular, does not depend on $\omega$). If $C_i \not \subset J'$ then by a similar gap argument to the one previously used,
$$|C_i| \leq \frac{\max_{j \in\lbrace 1,...,k_{\omega_1} \rbrace} |C_j|}{\min_{i\neq j \in \lbrace 1,...,k_{\omega_1} \rbrace} \text{dist} (C_i, C_j)} |J'|.$$
Noting that the latter constant can be bounded above uniformly in terms of the model, the proof is concluded in the same way.

\subsection{Spectral gap and reduction to an $L^2$ contraction estimate} \label{Section reduction}
We equip $C^1([0,1])$ with the norm
\begin{equation} \label{Eq B-Q norm}
||\varphi||_{C^1} = ||\varphi||_\infty + ||\varphi'||_{\infty}.
\end{equation}
Following Dolgopyat \cite[Section 6]{Dolgopyat1998rapid} and Naud \cite[the discussion prior to Lemma 5.2]{Naud2005exp}, for every $b \neq 0$ we  define yet another norm on $C^1([0,1])$ via
\begin{equation} \label{Eq norm dol theta}
||\varphi||_{(b)} = ||\varphi||_\infty + \frac{||\varphi'||_{\infty}}{|b|}.
\end{equation}
These two norms on $C^1([0,1])$ are clearly equivalent.

Recall the notations and definitions from Section \ref{Section pre}. The main goal of this (entire) Section is to prove the following spectral gap-type estimate:
\begin{theorem} \label{Theorem spectral gap}
Let $\Phi$ be a  $C^2 ([0,1])$ IFS satisfying properties \eqref{Add property 1} and \eqref{Add property 2}, and let $\nu=\nu_{\mathbf{p}}$ be a self-conformal measure. Then there exist $C,\gamma, \epsilon,R>0$ and some $0<\alpha<1$ such that for all $|b|>R$, $a\in \mathbb{R}$ with $|a|<\epsilon$, and $n\in \mathbb{N}$
\begin{equation} \label{Eq spectral gap}
||P_{a+ib} ^{n} ||_{C^1} \leq C\cdot |b|^{1+\gamma}\cdot  \alpha^n.
\end{equation}
\end{theorem}
Recall that by Claim \ref{Claim key} every not conjugate to linear $C^2$ IFS admits an induced IFS satisfying the conditions of Theorem \ref{Theorem spectral gap}.

We will  work with the model  constructed in Theorem \ref{Theorem disint}. Also, recall the definition of the operators $P_{s,\omega,k}$ from \eqref{Eq dis operator}.  We first reduce Theorem \ref{Theorem spectral gap} to the following statement about the $L^2$-contraction of \textit{all} parts of the transfer operator. 
\begin{Proposition} \label{Prop. 5.3 Naud} Assume the conditions of Theorem \ref{Theorem spectral gap}, and let $\Sigma$ be the model from Theorem \ref{Theorem disint}. Then there is some $N\in \mathbb{N}$ and $0<\alpha<1$ such that for $s=a+ib$ with $|a|$ small enough and $|b|$ large enough:

For every $\omega \in \Omega$,
$$\int \left| P_{s,  \omega , nN} W \right|^2 \, d\mu_{\sigma^{nN} \omega} \leq \alpha^n$$
for all $W\in C^1([0,1])$ with $||W||_{(b)} \leq 1$. 
\end{Proposition}
This is a randomized analogue of \cite[Proposition 5.3]{Naud2005exp} in Naud's work.

In the reminder of this Section, we reduce Theorem \ref{Theorem spectral gap} to  Proposition \ref{Prop. 5.3 Naud}. Our argument is roughly based on Naud's corresponding argument \cite[Section 5]{Naud2005exp}, with some significant variations due to our  model construction. To this end, we require the following two Lemmas. For $f\in C^1 ([0,1])$ let $||f||_{L([0,1]}$ denote the Lipschitz constant of $f$. 
\begin{Lemma} \label{Lemma 5.2 Naud} \cite[Proof of Theorem 2.1.1 part (ii)]{bishop2013fractal} There exists a constant $C_2>0$  such that  for all $f\in C^1 ([0,1])$,
 for every $\omega \in \Omega, x\in [0,1]$ and $\kappa \in \mathcal{P}([0,1])$,
$$\left|  P_{0,  \omega,  n} \left(f \right) (x) - \int P_{0,  \omega,  n} \left(f \right)(y) d\kappa(y) \right|  \leq C_2 \rho^n || f||_{L([0,1])}.$$

\end{Lemma}
\begin{proof}
The  operator
$$\mu \mapsto \int P_{0,  \omega,  n} \left(f \right)(y) d\mu(y)$$
contracts the space $\mathcal{P} \left( [0,1] \right)$ with the dual Lipschitz metric by a factor of $\rho^n$. Indeed, the IFS $\lbrace f^{(\omega)} _\beta \rbrace_{\beta\in X_n ^{(\omega)}}$ satisfies that $\max_{\beta\in X_n ^{(\omega)}} || \left( f^{(\omega)} _\beta \right)'||_\infty \leq \rho^n$. The Lemma now follows by noting that $P_{0,  \omega,  n} \left(f \right) (x) = \int P_{0,  \omega,  n} \left(f \right)(y)\,d\delta_x(y)$, taking $C_2$ to be the diameter of $\mathcal{P} \left( [0,1] \right)$. 
\end{proof}

Recall that $D'$ is as in \eqref{Eq C and C prime}. 
\begin{Lemma} (A-priori bounds) \label{Lemma iterations}
There exists $C_1>0$  such that for all $|a|$ small enough and $|b|$ large enough, for all $f\in C^1 ([0,1])$, writing $s=a+ib$,
$$|| \left( P_s ^n f\right)' ||_\infty \leq C_1 |b| \cdot ||P_a ^n f||_\infty+\rho ^n || P_a ^n \left( |f'| \right)||_\infty, \text{ and }$$
$$|| \left( P_{s, \omega, n} f\right)' ||_\infty \leq C_1 |b| \cdot ||P_{a, \omega, n}  f||_\infty+\rho ^n || P_{a, \omega, n} \left( |f'| \right)||_\infty, \, \text{ for all }\omega.$$
In particular, there exists a constant $C_6>0$ such that for all $n\geq 0$ and $s=a+ib$
$$|| P_s ^{n} ||_{(b)} \leq C_6 \cdot e^{nD' |a|}.$$
\end{Lemma}
\begin{proof}
By  Definition \ref{Def transfer operator} $\left( P_s ^n f\right)'(x)$ equals
$$\sum_{|I|=n} e^{(a+ib)\cdot 2 \pi \cdot  c(I,x)}  (a+ib) 2 \pi\cdot \left( \log f_I ' \right)' (x)  \mathbf{p}(I) f\circ f_I (x) + e^{(a+ib)\cdot 2 \pi \cdot  c(I,x)} \mathbf{p}(I)  \left(  f \circ f_I \right)' (x).$$
So, the first and second assertions follows by  noting that, as in e.g. \cite[proof of Claim 2.12]{algom2021decay}
$$\sup_{x\in [0,1], I\in \mathcal{A}^*} \left| \frac{d}{dx} \left( \log f_I ' \right) (x) \right| = C_1 <\infty.$$
For the last one,  if $||f||_{(b)}\leq 1$ then $||f||_\infty \leq 1$ and so $||f'||_\infty  \leq |b|$. Thus, by the first  assertion, and since  $||P_s ^{n} f||_\infty \leq e^{n D' |a|}$ as $||f||_\infty \leq 1$, we have
$$|| P_s ^{n} f ||_{(b)} = ||P_s ^{n} f||_\infty + \frac{|| \left( P_s ^n f\right)' ||_\infty}{|b|} \leq e^{n D' |a|} +C_1 e^{n D' |a|} +e^{n D' |a|} = C_6 e^{n D' |a|}. $$
\end{proof}

\noindent{ \textbf{Proof that Proposition \ref{Prop. 5.3 Naud} implies Theorem \ref{Theorem spectral gap}}} Fix $s=a+ib$ and $N>0$ as in Proposition \ref{Prop. 5.3 Naud}. Set
$$n=2[\frac{C}{N} \log |b|], \quad \tilde{n}=[\frac{C}{N} \log |b|],\quad \text{ with } C>0 \text{ to be chosen later.}$$
For all $s\in \mathbb{C}$ and $f\in C^1([0,1])$ with $||f||_{(b)}\leq 1$ and all $x\in [0,1]$ we have, by \eqref{Dis of transfer},
\begin{equation} \label{Eq convex combination 2}
\left| P_s ^{nN} \left(f \right) \right|(x) \leq  \sum_{\omega \in \Omega^{nN}} \mathbb{Q}\left([\omega]\right) \left| P_{s,\omega, nN} \left( f \right) \right|(x).
\end{equation}
Now, fix $\omega \in \Omega$. Then, by \eqref{Eq equivariance 2},
$$\left| P_{s,\omega, nN} \left( f \right) \right|(x) =  \left| P_{s, \sigma^{\tilde{n}N} \omega,  nN-\tilde{n}N}  \left( P_{s,\omega, \tilde{n}N} \left( f \right)  \right) \left( x \right) \right| \leq P_{a, \sigma^{\tilde{n}N} \omega,  nN-\tilde{n}N} \left( \left| P_{s,\omega, \tilde{n}N} \left( f \right) \right| \right)(x).$$
Set $m=(n-\tilde{n})N$. Note that
$$ P_{a, \sigma^{\tilde{n}N} \omega,  m} \left( \left| P_{s,\omega, \tilde{n}N} \left( f \right) \right| \right)(x) = \sum_{I \in X^{(\sigma^{\tilde{n}N} \omega)} _m} e^{2\pi a c(I,x)} \sqrt{ \eta ^{(\sigma^{\tilde{n}N} \omega) }(I)}  \cdot  \left( \sqrt{ \eta ^{(\sigma^{\tilde{n}N} \omega) }(I)}  \left| P_{s,\omega, \tilde{n}N} \left( f \right) \right|\circ f_I (x) \right).$$
Then by Cauchy-Schwartz,
$$ \left( P_{a, \sigma^{\tilde{n}N} \omega,  m} \left( \left| P_{s,\omega, \tilde{n}N} \left( f \right) \right| \right)(x) \right)^2 \leq e^{2m |a|\cdot D'} \cdot P_{0, \sigma^{\tilde{n}N} \omega,  m} \left( \left| P_{s,\omega, \tilde{n}N} \left( f \right)  \right|^2 \right) (x).$$

We wish to invoke Lemma \ref{Lemma 5.2 Naud}, so we must estimate $|| \left| P_{s,\omega, \tilde{n}N} \left( f \right)  \right|^2 ||_{L([0,1])}$: By Lemma \ref{Lemma iterations}, using that $||f||_\infty \leq 1$ and $||f'||_\infty \leq |b|$,
$$|| \left| P_{s,\omega, \tilde{n}N} \left( f \right)  \right|^2 ||_{L([0,1])} \leq  2 \cdot ||  P_{s,\omega, \tilde{n}N} \left( f \right) ||_\infty \cdot  || \left( P_{s,\omega, \tilde{n}N} \left( f \right) \right)' ||_\infty $$
$$\leq 2e^{\tilde{n}N |a|D'} \cdot \left( C_1 |b|\cdot  || P_{a,\omega, \tilde{n}N} f||_\infty + \rho^{\tilde{n}N} || P_{a,\omega, \tilde{n}N} \left( |f'|\right)||_\infty \right)$$
$$\leq  2e^{\tilde{n}N aD'} \cdot \left( C_1 |b| e^{\tilde{n}N |a|D'} + \rho^{\tilde{n}N}\cdot |b|\cdot e^{\tilde{n}N |a|D'} \right)$$
$$\leq C_3 |b| e^{\tilde{n}N 2|a|D'}$$

Via Lemma \ref{Lemma 5.2 Naud} applied with $\kappa = \mu_{ \sigma^{nN} (\omega)}$, \eqref{Eq equivariance},  and the previous calculation we have
$$\Vert P_{s,\omega, nN} \left( f \right)  \Vert ^2 _\infty \leq e^{2m |a|\cdot D'} \left( \int P_{0, \sigma^{\tilde{n}N} \omega,  m} \left( \left| P_{s,\omega, \tilde{n}N} \left( f \right)  \right|^2 \right) \,d\mu_{ \sigma^{nN} (\omega)} + C_2 \rho ^m || \left| P_{s,\omega, \tilde{n}N} \left( f \right)  \right|^2 ||_{L([0,1])} \right)$$
$$\leq |b|^{C |a| 2 D'} \left( \int \left| P_{s,\omega, \tilde{n}N} \left( f \right)  \right|^2  \,d\mu_{ \sigma^{\tilde{n}N} (\omega)} +  C_3C_2 \rho^m |b| e^{\tilde{n}N 2|a|D'} \right)$$
$$\leq  |b|^{C |a| 2 D'} \left( \int \left| P_{s,\omega, \tilde{n}N} \left( f \right)  \right|^2  \,d\mu_{ \sigma^{\tilde{n}N} (\omega)} +  C_3C_2 \rho^m |b|^{1+2|a|D'C} \right).$$
Applying Proposition \ref{Prop. 5.3 Naud}, 
$$||  P_{s,\omega, nN} \left( f \right) ||_\infty ^2 \leq |b|^{C |a| 2 D'} \left( \frac{C_4}{|b|^{ \frac{C}{N}|\log \alpha| }}+  \frac{C_5}{|b|^{ C|\log \rho| -1-2|a|D'C} } \right),$$
where $C_4,C_5$ are positive constants. Recalling that $|a|$ is already assumed to be small, we can now choose $C$ with $C |\log \rho|>1+2|a|D'C$ and $|a|$ even closer to $0$ so that if $|a|$ is small enough and $|b|$ is large enough,
$$||  P_{s,\omega, nN} \left( f \right) ||_\infty \leq \frac{1}{|b|^\beta}$$
for some $\beta>0$. Using that the same bound works for every term in the convex combination \eqref{Eq convex combination 2}, we obtain
$$||  P_s ^{nN} f ||_\infty  \leq \frac{1}{|b|^\beta}.$$

Next, by \eqref{Dis of transfer}
\begin{equation}
\left( P_s ^{nN} \left(f \right) \right)' (x) = \sum_{\omega \in \Omega^{nN}} \mathbb{Q}\left([\omega]\right) \left( P_{s,\omega, nN} \left( f \right) \right)' (x).
\end{equation}
So, since $||f'||_\infty \leq |b|$ (as $||f||_{(b)}\leq 1$) we have, as in Lemma \ref{Lemma iterations}, for some $c_1>0$
$$ \left| \left( P_s ^{nN} \left(f \right) \right)' (x)  \right| \leq  \sum_{\omega \in \Omega^{nN}} \mathbb{Q}\left([\omega]\right) \left( |b|\cdot c_1\cdot \left| P_{s,\omega, nN} \left( f \right)(x)\right|+ \rho^{nN} P_{a, \omega, {nN}} \left( \left| f' \right| \right)(x)  \right) $$
$$ \leq  c_1 \cdot |b| \sum_{\omega \in \Omega^{nN}}  \mathbb{Q}\left([\omega]\right)  \left| P_{s,\omega, nN} \left( f \right)(x)\right| + \rho^{nN} \cdot  |b| \cdot e^{nN D' |a|}.$$
Thus, 
$$ \frac{1}{|b|} || \left( P_s ^{nN} f \right)' ||_\infty \leq c_1 \sum_{\omega \in \Omega^{nN}}  \mathbb{Q}\left([\omega]\right)  || P_{s,\omega, nN} \left( f \right) ||_\infty  + \rho ^{nN} e^{nN D' |a|}$$
$$\leq c_1 \sum_{\omega \in \Omega^{nN}}  \mathbb{Q}\left([\omega]\right)  || P_{s,\omega, nN} \left( f \right) ||_\infty  + \frac{1}{\rho^{2N}}\cdot  |b|^{2C( D' |a|-  |\log \rho|) }.$$
So, using similar ideas, for $|a|$ small and $|b|$ large and possibly a large $C$,
$$|| P_s ^{Nn} ||_{(b)} \leq \frac{1}{|b|^{\beta'}},$$
for some $\beta'>0$ and $n=[\frac{C}{N}\log |b|]$.

Finally, given $m\in \mathbb{N}$ we write 
$$m=d\cdot N [\frac{C}{N}\log |b|]+r \text { for the corresponding } d,0\leq r\leq N\cdot [\frac{C}{N}\log |b|].$$
  Let $\beta''>0$ be such that $C D' |a| - \beta' < -\beta''$ for all small enough $|a|$. Applying Lemma \ref{Lemma iterations},
$$|| P_s ^m ||_{(b)} \leq || P_s ^r ||_{(b)} \cdot || P_s ^{dN\cdot [\frac{C}{N}\log |b|]} ||_{(b)}  \leq  C_6 \cdot e^{C (\log |b|)\cdot D' |a|}  \cdot \left( \frac{1}{|b|^{\beta'}} \right)^d$$ 
$$ \leq  C_6 \cdot \left( \frac{ |b|^{ CD'|a|}}{|b|^{\beta'}} \right)^d \leq C_6 \cdot \left( \frac{1}{|b|^{\beta''}} \right)^d \leq C_6 \cdot |b|^{\beta''} \rho_{\beta''} ^m,$$
for some $\rho_{\beta''}\in (0,1)$.  This is true for all $\beta''>0$ with $\beta''$ small, and since $$||\cdot ||_{C^1} \leq |b|\cdot  ||\cdot ||_{(b)}$$ Theorem 2.3 follows. \hfill{$\Box$}

\subsection{The key Lemma} \label{Section key Lemma}
Given $A>0$ let
$$\mathcal{C}_A = \lbrace f\in C^1([0,1]): \, f>0,\text{ and } \left|f'(x)\right|\leq A\cdot f(x), \, \forall x\in [0,1] \rbrace.$$
Note: If $f\in \mathcal{C}_A$ and $u,v \in [0,1]$ then
$$e^{-A|u-v|} \leq \frac{f(u)}{f(v)}\leq e^{A|u-v|}.$$
We will prove the following variant of Naud's result \cite[Lemma 5.4]{Naud2005exp}, that is the key to the proof of Proposition \ref{Prop. 5.3 Naud}:
\begin{Lemma} \label{Lemma 5.4 Naud}
There exist $N>0, A>1$ and $0< \alpha <1$ such that for all $s=a+ib$ with $|a|$ small and $|b|$ large:

For every $\omega\in \Omega$ there exist a finite set of bounded operators $\lbrace N_s ^J \rbrace_{J\in \mathcal{E}_{s,\omega}}$ on $C^1([0,1])$ such that:
\begin{enumerate}
\item The cone $\mathcal{C}_{A|b|}$ is stable under $N_s ^J$ for all $J\in \mathcal{E}_{s,\omega}$. 

\item For all $H\in \mathcal{C}_{A|b|}$ and $J\in \mathcal{E}_{s,\omega}$,
$$\int | N_s ^J H|^2 d \mu_{\sigma^N \omega} \leq \alpha\cdot \int P_{0,\omega,N} \left( \left| H^2 \right| \right) d \mu_{\sigma^N \omega}.$$

\item Let $H\in \mathcal{C}_{A|b|}$ and $f\in C^1([0,1])$ be such that $|f|\leq H$ and $|f'|\leq A|b|H$. Then for every $\omega \in \Omega$ there exists $J\in \mathcal{E}_{s,\omega}$ such that
$$|P_{s,\omega, N} f| \leq N_s ^J H \quad \text{ and } |(P_{s,\omega, N} f)'|\leq A|b|N_s ^J H.$$
\end{enumerate} 
\end{Lemma}
As is customary in the literature (e.g. \cite{Naud2005exp, Stoyanov2011spectra})  we will refer to the operators constructed in Lemma \ref{Lemma 5.4 Naud} as \textit{Dolgopyat operators}. 
$$ $$

\noindent{ \textbf{Proof that Lemma \ref{Lemma 5.4 Naud} implies Proposition \ref{Prop. 5.3 Naud}}} First, observe that for every $\omega\in \Omega, s\in \mathbb{C}$ and $N,n\in \mathbb{N}$ we have, by \eqref{Eq equivariance 2}
$$P_{s,\omega, nN} = P_{s,\sigma^{nN} \omega, N} \circ ...\circ P_{s,\sigma^N \omega, N} \circ P_{s,\omega, N}.$$
Let $f\in C^1([0,1])$ with $f\neq 0$ and $||f||_{(b)} \leq 1$. Put $H= ||f||_{(b)}$. Then $|f|\leq H$ and $|f'|\leq |b|\cdot ||f||_{(b)} \leq A|b|H$.  By induction, for all $n\geq 1$  and every $\omega$ there are $J_{\omega} \in \mathcal{E}_{s,\omega}, J_{\sigma^N \omega} \in \mathcal{E}_{s,\sigma^N \omega},...,J_{\sigma^{nN} \omega}\in \mathcal{E}_{s,\sigma^{Nn} \omega}$ such that
$$\left| P_{s,\omega,nN} f \right| \leq N_s ^{J_{\sigma^{nN} \left( \omega \right)}} N_s ^{J_{\sigma^{(n-1)N} \left(  \omega  \right) }}....N_s ^{J_{\omega}} H   \text{ and }  \left|\left(P_{s,\omega,nN}\right)'\right|\leq A|b|N_s ^{J_{\sigma^{nN} \left( \omega \right)}} N_s ^{J_{\sigma^{(n-1)N} \left(  \omega  \right) }}....N_s ^{J_{\omega}} H. $$
So, by Lemma \ref{Lemma 5.4 Naud} and our equivariance relations,
\begin{eqnarray*} \int \left|P_{s,\omega,nN} f\right|^2 d\mu_{\sigma^{nN} \omega}  &=&  \int |N_s ^{J_{n,\omega}} N_2 ^{J_{n-1,\omega}}....N_s ^{J_{1,\omega}} H |^2 d\mu_{\sigma^{nN}\omega }  \\
&\leq & \alpha\cdot \int  P_{0, \sigma^{N(n-1)} \omega, N}  \left( | N_2 ^{J_{n-1,\omega}}....N_s ^{J_{1,\omega}} H |^2 \right) d\mu_{\sigma^{nN}\omega}\\
&= & \alpha\cdot \int    \left( | N_2 ^{J_{n-1,\omega}}....N_s ^{J_{1,\omega}} H |^2 \right) d\mu_{\sigma^{(n-1)N} \omega}\\
&\leq & \alpha^n \cdot \int P_{0,  \omega, N}   \left(  |H |^2 \right) d\mu_{\sigma^{N} \omega}\\
&= & \alpha^n \cdot \int  |H |^2  d\mu_{ \omega}\\
&\leq  & \alpha^n.\\
\end{eqnarray*}
$\hfill{\Box}$

\subsection{The triple intersections property}
The following Proposition allows for the construction, for every $\omega$, of a special partition of $[0,1]$ that has the triple intersections property on $K_\omega$: This means that whenever a cell intersects $K_\omega$, two other nearby cells must also intersect $K_\omega$. It is based upon \cite[Proposition 5.6]{Naud2005exp}, but as usual there are significant variation due to our model setting. Thus, while the partitions themselves depend on $\omega$, certain metric features of them, e.g. the size of the cells and the distance of the endpoints from $K_\omega$, are  uniform across our model $\Sigma$. 
\begin{Proposition} \label{Prop. 5.6 Naud}
There exist constants $A_1 ', A_1, A_2>0$ such that for all $\omega$ and every $\epsilon>0$ small enough:

There exists a finite collection of closed intervals $(V_i)_{1\leq i \leq q}$ ordered along $[0,1]$ such that:
\begin{enumerate}
\item $[0,1] = \bigcup_{i=1} ^q V_i$, and $i\neq j \Rightarrow \text{Int } V_i \cap \text{Int } V_j = \emptyset$. 

\item For all $1\leq i \leq q$ we have $\epsilon A_1 ' \leq |V_i| \leq \epsilon A_1$. 

\item For all $1\leq j \leq q$ such that $V_j\cap K_\omega \neq \emptyset$, either  $V_{j-1}\cap K_\omega \neq \emptyset$ and $V_{j+1}\cap K_\omega \neq \emptyset$, or $V_{j-2}\cap K_\omega \neq \emptyset$ and $V_{j-1}\cap K_\omega \neq \emptyset$, or $V_{j+1}\cap K_\omega \neq \emptyset$ and $V_{j+2}\cap K_\omega \neq \emptyset$.

\item For all $1\leq i \leq q$ such that $V_i \cap K_\omega \neq \emptyset$, we have 
$$\text{dist}(\partial V_i,\, K_\omega)\geq A_2 |V_i|$$
\end{enumerate}
\end{Proposition}
It is critical to our argument is that the constants $A_1',A_1, A_2$ may be chosen uniformly across the model.

The proof of Proposition \ref{Prop. 5.6 Naud}, that we discuss now, is roughly modelled after Naud's arguments in \cite[Section 7]{Naud2005exp}. Fix $\omega$ and recall Definition \ref{Def cylinders} (cylinders of $K_\omega$). We require the following Lemmas:
\begin{Lemma} \label{Lemma 7.1 Nuad}
There exists a constant $B_1$ such that for all $x\in K_\omega$ and all $r>0$, there exists a cylinder $C_\alpha$ such that
$$C_\alpha \subseteq B(x,r) \text{ and } B_1 \cdot r \leq |C_\alpha|.$$
\end{Lemma}
\begin{proof}
Consider a cylinder $C_\alpha$ with $x\in C_\alpha\subseteq B(x,r)$, and assume $|\alpha|$ is minimal. Since $x\in K_\omega$ such a cylinder exists. Let $C_\alpha \subseteq C_\beta$ with $|\beta|=|\alpha|-1$. By minimality of $|\alpha|$ we cannot have $C_\beta$ included in $B(x,r)$, and thus $|C_\beta|\geq r$. Via Lemma \ref{Lemma 6.1 Naud} we obtain 
$$C^{-1} \delta_1 r \leq C^{-1} \delta_1 |C_\beta| \leq |C_\alpha|$$
as claimed.  
\end{proof}

\begin{Lemma} \label{Lemma 7.2 Naud}
Let $C_\beta$ be a cylinder. Then there a finite set of at least $3$ words $A_\beta$ such that
$$C_\beta \cap K_\omega \subseteq \bigcup_{\gamma \in A_\beta} C_\gamma,$$
where $C_\gamma \subseteq C_\beta$ and $|\gamma| = |\beta|+2$.
\end{Lemma}
\begin{proof}
By definition, there is some $\beta \in X_n ^{(\omega)}$ such that
$$C_\beta := f^{ (\omega_1)} _{\beta_1} \circ \circ \circ f^{ (\omega_n)} _{\beta_n} ([0,1]).$$
Furthermore, by Theorem \ref{Theorem disint} Part (4)
$$C_\beta \cap K_\omega = f^{ (\omega_1)} _{\beta_1} \circ \circ \circ f^{ (\omega_n)} _{\beta_n} (K_{\sigma^n (\omega)}).$$
Now, by the definition of the model,
$$K_{\sigma^n (\omega)}  = \bigcup_{\zeta\in X^{(\sigma^n \omega)} _{2}} f^{ (\omega_{n+1})} _{\zeta_{1}} \circ  f^{ (\omega_{n+2})} _{\zeta_{2}} ( K_{\sigma^{n+2} (\omega)}), $$
and $4\leq \left|X^{(\sigma^n \omega)} _{2} \right|\leq 9$ by Theorem \ref{Theorem disint} Part (3). Note that we are also using Theorem \ref{Theorem disint} Part (4) to see that $ X^{(\sigma^n \omega)} _{2}$ does not have exact overlaps (maps with different coding are not equal). Thus,
$$C_\beta \cap K_\omega \subseteq \bigcup_{\gamma\in X^{(\sigma^n \omega)} _{2}} f^{ (\omega_1)} _{\beta_1} \circ \circ \circ f^{ (\omega_n)} _{\beta_n} \circ f^{ (\omega_{n+1})} _{\zeta_{1}} \circ  f^{ (\omega_{n+2})} _{\zeta_{2}} ( [0,1]),$$
which is a union of $4\leq k \leq 9$ cylinders that are contained in $C_\beta$. As required.
\end{proof}

\begin{Remark}
By Lemma \ref{Lemma 6.1 Naud}, there are uniform constants $B_2,B_3>0$ independent of $\beta, A_\beta$, such that  for all $\gamma \in A_\beta$ we have
$$B_2 |C_\beta|\leq |C_\gamma| \leq B_3 |C_\beta|.$$
Applying Lemma \ref{Lemma 6.2 Naud}, or Theorem \ref{Theorem disint} Part (4) directly, there is a uniform constant $B_4>0$ independent of $\beta, A_\beta$,  such that for all $\gamma_1\neq \gamma_2 \in A_\beta$ we have
$$\text{dist} \left( C_{\gamma_1},\, C_{\gamma_2} \right) \geq B_4 \left| C_\beta \right|.$$
\end{Remark}

\noindent{ \textbf{Proof of Proposition \ref{Prop. 5.6 Naud}}} Let $0<\epsilon \ll 1$. Set $p:=[\frac{1}{\epsilon}]$. First, we divide $[0,1]$ into $p$ closed intervals $J_i$ with disjoint interiors such that
$$[0,1] = \bigcup_{i=1} ^p J_i \text{ and } \epsilon \leq \left| J_i \right| \leq 2 \epsilon.$$
For every $1\leq i \leq p$ write $J_i = [x_i, x_{i+1}]$. We may  assume $x_1 =0$ and $x_{p+1}=1$.

We first deal with part (4): For $2\leq i \leq p$, if $B(x_i, \frac{\epsilon}{8})\cap K_\omega = \emptyset$ we set $\tilde{x}_i :=x_i$. Otherwise, let $x_i ' \in B(x_i, \frac{\epsilon}{8})\cap K_\omega$. Applying Lemma \ref{Lemma 7.1 Nuad}, we find a cylinder $C_\alpha \subseteq B(x_i ' , \frac{\epsilon}{8})$ with $|C_\alpha| \geq B_1 \frac{\epsilon}{8}$. By Lemma \ref{Lemma 7.2 Naud}, there are consecutive cylinders $C_{\gamma_1},C_{\gamma_2} \subset C_\alpha$ such that
$$\text{dist} \left( C_{\gamma_1},\, C_{\gamma_2} \right) \geq B_4 \left| C_\alpha \right|.$$ 
Set 
$$\tilde{x}_i := \frac{1}{2}\left( \max C_{\gamma_1} + \min C_{\gamma_2}\right).$$
In both cases 
$$\left| \tilde{x}_i - x_i \right| \leq \frac{\epsilon}{4} \text{ and } \text{dist} \left( \tilde{x}_i , K_\omega \right) \geq \min \left( \frac{\epsilon}{8},\, B_1 B_4 \frac{\epsilon}{16}\right).$$
As for the boundary points $x_1=0,x_{p+1}=1$, by \eqref{Assumption on endpoints}
$$\text{dist} \left( K_\omega,\, \partial [0,1] \right) \geq \text{dist} \left( K,\, \partial [0,1] \right) > 0.$$
So, upon taking $\epsilon \ll 1$, we may put  $\tilde{x}_1 =0$ and $\tilde{x}_{p+1} =1$. 

For all $1\leq i \leq p$ set $\tilde{J}_i:=[\tilde{x}_i ,\tilde{x}_{i+1}]$. Then, writing $B_5 = \frac{2}{5} \min \left( \frac{B_1 B_4}{16}, \frac{1}{8} \right)$ we have
$$\frac{\epsilon}{2}\leq \left|  \tilde{J}_i \right| \leq \frac{5}{2}\epsilon, \,   \text{dist} \left( \partial \tilde{J}_i  , K_\omega \right) \geq B_5 \left|  \tilde{J}_i \right|, \text{ and } [0,1] \subseteq \bigcup_{i=1} ^p \tilde{J}_i.$$
Furthermore, the intervals $\tilde{J}_i$ still have disjoint interiors. Thus, this collection of intervals satisfies parts (1),(2), and (4) of the Proposition. Let us call these intervals $J_i$.

We now deal with property (3): Fix $J_i$ such that $K_\omega\cap J_i \neq \emptyset$. Let $x\in K_\omega \cap J_i$, so that $B(x, B_5 \frac{\epsilon}{2})\subset J_i$. By Lemma \ref{Lemma 7.1 Nuad} and Lemma \ref{Lemma 7.2 Naud} there are $3$ consecutive cylinders $C_{\gamma_1}, C_{\gamma_2}, C_{\gamma_3}$ such that
$$C_{\gamma_1} \cup  C_{\gamma_2} \cup  C_{\gamma_3} \subseteq B(x, B_5 \frac{\epsilon}{2})\subset J_i.$$
In addition, we have for $i=1,2$ 
$$\text{dist} \left( C_{\gamma_1} , C_{\gamma_2} \right) \geq B_4 \cdot B_1 \cdot B_5 \frac{\epsilon}{2}.$$
We set
$$y_i := \frac{1}{2}\left( \max C_{\gamma_1} + \min C_{\gamma_2}\right) \text{ and } z_i:= \frac{1}{2}\left( \max C_{\gamma_2} + \min C_{\gamma_3}\right),$$
and write
$$J_i ^1 = [x_i,\, y_i],\, J_i ^2 = [y_i,\, z_i],\, J_i ^3 = [z_i,\, x_{i+1}].$$
Then for all $j=1,2,3$ we have  $J_i ^{j} \cap K \neq \emptyset$, and 
$$B_1 B_2 B_5 \frac{\epsilon}{2}\leq \left| J_i ^j \right| \leq \frac{5}{2}\epsilon \text{ and } \text{dist} \left( \partial \tilde{J}_i ^j , K_\omega \right) \geq \min \left( B_5 B_4 B_1 \frac{\epsilon}{4},\, B_5 \frac{\epsilon}{2}\right).$$

Finally, the set of intervals
$$\lbrace J_i:\, J_i \cap K_\omega = \emptyset \rbrace \bigcup_{j=1} ^3 \lbrace J_i ^j:\, J_i \cap K_\omega \neq \emptyset \rbrace$$
now satisfy all the properties in Proposition \ref{Prop. 5.6 Naud}. \hfill{$\Box$}

\subsection{Proof of Lemma \ref{Lemma 5.4 Naud}}
Fix $\omega$ and let $s=a+ib$. We begin by constructing the Dolgopyat operators as in Lemma \ref{Lemma 5.4 Naud}.  Let $N\in \mathbb{N}$ be sufficiently large in the sense of Theorem \ref{Theorem disint} part (5), and in other ways that will be specified soon, and let $\alpha_1 ^N, \alpha_2 ^N \in X_N ^{(\omega)}$ be the length $N$ words satisfying the conclusion of  Theorem \ref{Theorem disint} Part (5). Let us fix $\epsilon' ,\frac{1}{|b|}>0$  uniformly in $\omega$, that are small enough (to be determined later). Let $(V_i)_{1\leq i\leq q}$ be a triadic partition  as in Proposition \ref{Prop. 5.6 Naud} of $K_{\sigma^N \omega}$ of modulus $\epsilon=\frac{\epsilon'}{|b|}$. For all $(i,j)\in \lbrace 1,2\rbrace \times \lbrace 1,...,q\rbrace$ set
$$Z_j ^i = f_{\alpha_i ^N} \left( V_j \right).$$

By Proposition \ref{Prop. 5.6 Naud} part (4)
$$\text{dist} \left( K_{\sigma^N \omega} \cap V_j ,\, \partial V_j \right) \geq A_2 A_1 ' \frac{\epsilon'}{|b|} \text{ whenever } K_{\sigma^N \omega} \cap V_j \neq \emptyset.$$
Then there exists a cut off function $\chi_j \in C^1 ([0,1])$ such that $0\leq \chi_j \leq 1$ on $[0,1]$, $\chi_j =1$ on $\text{conv} \left( K_{\sigma^N \omega} \cap V_j \right)$, and $\chi_j =0$  outside of $\text{Int}(V_j)$. Then there exists $A_3>0$ that depends only on the previous (uniform in $\omega$) constants such that
$$||\chi_j '||_\infty \leq A_3 \frac{|b|}{\epsilon'}.$$
Define
$$J_{s,\omega} = \lbrace (i,j)\in \lbrace 1,2\rbrace \times \lbrace 1,...,q\rbrace : \, V_j \cap K_{\sigma^N \omega} \neq \emptyset \rbrace.$$
Fix $0< \theta <1$ to be determined later. Let $\emptyset \neq J \subseteq J_{s,\omega}$. Define a function $\chi_J \in C^1 ([0,1])$ by
$$\chi_J (x) = 1- \theta \cdot \chi_j \circ f_{\alpha_i ^N} ^{-1} (x)  \text{ if } (i,j)\in J \text{ and } x\in Z_i ^j, \, \text{ and } \chi_J (x) = 1 \text{ otherwise. }$$
Note that $\chi_J$ is well defined by the separation property in Theorem \ref{Theorem disint} Part (4), and since the $V_j$'s intersect potentially only at their endpoints by Proposition \ref{Prop. 5.6 Naud}, where all the $\chi_j$ vanish.

We can now define the Dolgopyat operators $N_s ^J$ on $C^1 ([0,1])$ by:
$$N_s ^J \left( f \right)(x):= P_{a,\omega, N} \left( \chi_J \cdot f \right).$$
We proceed to prove the three assertions of Lemma \ref{Lemma 5.4 Naud}.
\subsubsection{Part 1: construction of an invariant cone}
We follow the same notations of the construction carried out in the previous section, and prove Lemma \ref{Lemma 5.4 Naud} Part (1):
\begin{Lemma} \label{Lemma cone}
There exist $A>1$, $N\in \mathbb{N}$ and $0<\theta <1$ such that for $s=a+ib$ with $|a|$ sufficiently small and $|b|$ sufficiently large, for every $\omega$,  
\begin{enumerate}
\item The cone $\mathcal{C}_{A|b|}$ is stable under  every $N_s ^J$.

\item If $f\in C^1([0,1])$ and $H\in \mathcal{C}_{A|b|}$ satisfy
$$ |f|\leq H \text{ and } \left| f' \right| \leq A|b|H$$
then
$$\left| \left( P_{s,\omega,N} \left( f \right) \right)' (x) \right| \leq  A|b| N_s ^J (H) (x).$$

\item If $H\in \mathcal{C}_{A|b|}$ then $P_{a,\omega,N} \left(H^2 \right) \in \mathcal{C}_{ \frac{3}{4}A|b|}$.
\end{enumerate}
\end{Lemma}
Our proof  is roughly based on \cite[proof of equation (4)]{Naud2005exp}:
\begin{proof}
Fix $H\in \mathcal{C}_{A|b|}$ where $A$ is yet to be determined. Then for all $x\in [0,1]$,
$$\left| N_s ^J \left( H \right)' (x) \right| = \left| P_{a,\omega,N} \left( \chi_J \cdot H \right)'(x) \right|$$
$$ \leq \sum_{I\in X_N ^{(\omega)}} e^{a\cdot 2 \pi \cdot  c(I,x)} \left| a 2 \pi\cdot \left( \log f_I ' \right)' (x) \right|  \eta^{(\omega)}(I) \left(H\cdot \chi_J \right)\circ f_I (x) + e^{a\cdot 2 \pi \cdot  c(I,x)} \eta^{(\omega)}(I)  \left| \left(  \left(H\cdot \chi_J \right) \circ f_I \right)' (x) \right|.$$
Now, for every $I\in X_N ^{(\omega)}$, by separation (Theorem \ref{Theorem disint} Part (4)) we have
$$\left| \left(\chi_J \circ f_{I} \right)' \right| \leq \theta A_3 \frac{|b|}{\epsilon'}.$$
Also, we can find a constant $\tilde{C}$ uniform in $N$, and $a$ such that if $|a|$ is small enough then
$$\left| a 2 \pi\cdot \left( \log f_{I} ' \right)' (x) \right|  \leq \tilde{C}.$$
Therefore,
$$\left| N_s ^J \left( H \right)' (x) \right| \leq \sum_{I\in X_N ^{(\omega)}} e^{a\cdot 2 \pi \cdot  c(I,x)}  \eta^{(\omega)}(I)) \left(H\cdot \chi_J \right)\circ f_I (x) \cdot \tilde{C}  + e^{a\cdot 2 \pi \cdot  c(I,x)} \eta^{(\omega)}(I)  \left(H\cdot \chi_J \right)\circ f_I  (x) \cdot  A \cdot |b| \rho^N$$
$$+ e^{a\cdot 2 \pi \cdot  c(\alpha_1 ^N,x)} \eta^{(\omega)}(I)  H \circ f_{I}  (x) \cdot \theta A_3 \frac{|b|}{\epsilon'}.$$
Using that $H = \frac{ \left(\chi_J H\right) }{\chi_J}\leq \frac{1}{1-\theta}\chi_J H$ we obtain
$$\left| N_s ^J \left( H \right)' (x) \right| \leq \left( \frac{\tilde{C}}{|b|}+A_3 \frac{\theta}{(1-\theta)\epsilon'}+A \rho^{N} \right) |b| N_s ^J \left( H \right) (x) \leq A|b|N_s ^J \left( H \right) (x),$$
assuming $|b|$ is large enough, $|a|$ is small enough, and 
$$\theta \leq \min \left( \frac{1}{2},\, \epsilon' \frac{A-1}{4 A_3} \right) \text{ and } \rho^N \leq \frac{A-1}{2A}.$$

Note that the above calculation works for any $A>1$. Now, if $f\in C^1([0,1])$ and $H\in \mathcal{C}_{A|b|}$ satisfy
$$ |f|\leq H \text{ and } \left| f' \right| \leq A|b|H$$
then
$$\left| P_{s,\omega, N} \left( f \right)' (x) \right|  \leq \sum_{I\in X_N ^{(\omega)}} e^{a\cdot 2 \pi \cdot  c(I,x)} \left| (a+ib) 2 \pi\cdot \left( \log f_I ' \right)' (x) \right|  \eta^{(\omega)}(I) H\circ f_I (x)$$
$$ + e^{a\cdot 2 \pi \cdot  c(I,x)} \eta^{(\omega)}(I) \left| \left( f\circ f_I \right)' (x) \right|$$
$$ \leq \sum_{I\in \mathcal{A}^N} e^{a\cdot 2 \pi \cdot  c(I,x)}   \eta(I) H\circ f_I (x)\cdot \hat{C}\cdot |b| + e^{a\cdot 2 \pi \cdot  c(I,x)} \eta(I) H\circ f_I (x) \cdot A|b|\rho^N,$$
where $\hat{C}$ is independent of $N$ and $|b|$ is large enough. Note that $\chi_J ^{-1} \leq 2$ when  $\theta \leq \frac{1}{2}$ and so in this case $H\leq 2\chi_J H$. So, under this assumption
$$\left| P_{s,\omega, N} \left( f \right)' (x) \right|  \leq A|b| N_s ^J \left( H \right) (x)$$
as long as $A \geq 4\hat{C}$ and $\rho^N \leq \frac{1}{4}$.

We thus fix $A\geq \max \left( 2, 4\hat{C} \right)$ and take $N$ large enough so that $\rho^N \leq \min \left( \frac{A-1}{2A}, \frac{1}{4} \right)$ and fix $\theta \leq \min \left( \frac{1}{2}, \epsilon' \frac{A-1}{4A_3} \right)$.  A similar argument shows that this choice of parameters also yields Part  (3).

\end{proof}

\subsubsection{Part 2: $L^2$ contraction of the cones}
We proceed to prove that the operators $N_s ^J$ contract these cones in the $L^2$ norms. We require  the following Definition:
\begin{Definition}
A subset $J\subseteq J_{s,\omega}$ is called dense if for every $1\leq j \leq q$ such that $V_j \cap K_{\sigma^N \omega} \neq \emptyset$ there exists $1\leq j' \leq q$ such that:
$$\exists i \text{ such that } (i,j')\in J \text{ and } |j' - j|\leq 2.$$
For a dense subset $J$ we write
$$W_J = \lbrace x\in K_{\sigma^N \omega}:\, \exists(i,j)\in J,\, x\in V_j\rbrace.$$
\end{Definition}
We need the following key Lemma, a variant of \cite[Lemma 5.7]{Naud2005exp}:
\begin{Lemma} \label{Lemma 5.7 Nuad} Let $J$ be dense and fix $H\in C_{A|b|}$. Then there exists $\tilde{\epsilon}>0$ independent of $H,|b|,\omega, N$ and $J$ such that
$$\int_{W_J} H \, d \mu_{\sigma^N \omega} \geq \tilde{\epsilon}\int_{K_{\sigma^N \omega}} H d  \mu_{\sigma^N \omega}.$$
\end{Lemma}
\begin{proof}
Let
$$G:=\lbrace 1\leq i \leq q: V_i\cap K_{\sigma^N \omega}\neq \emptyset\rbrace$$
and note that $K_{\sigma^N \omega} \subseteq \bigcup_{i\in G} V_i$. For every $i\in G$, by density of $J$, there exists some index $j(i)$ such that $\left(i',\, j(i)\right)\in J$ for some $i'\in \lbrace 1,2\rbrace$ such that $\left| j(i)-i \right| \leq 2$. We  thus get a function $j:G\rightarrow G$ such that for every $i\in G$ the set $ j^{-1} \left( \lbrace i \rbrace \right)$ contains at most $5$ elements.

Let $r=3A_1 \frac{\epsilon'}{|b|}$. For every $i\in G$ fix some $u_i \in V_i \cap K_{\sigma^N \omega}$.  By Proposition \ref{Prop. 5.6 Naud}, 
$$V_{j(i)},\, V_i \subseteq B(u_i,\, r).$$
Moreover, by Proposition \ref{Prop. 5.6 Naud} part (4), for $r' = \frac{1}{2}A_2 A_1' \frac{\epsilon'}{|b|}$ and some carefully chosen $v_i \in K_{\sigma^N \omega}\cap V_{j(i)}$ we have
$$\text{dist} \left( K_{\sigma^N \omega} \cap V_{j(i)},\, \partial V_{j(i)} \right) = \text{dist} \left( v_i,\, \partial V_{j(i)} \right)  \text{ and so } V_{j(i)} \supset B(v_i,\, r').$$
Note that
$$B(v_i, r') \subset B(u_i, r) \subset B(v_i, 2r).$$
So, by the Federer property of $\mu_{\sigma^N \omega}$ proved  in Theorem \ref{Theorem disint},
$$\mu_{\sigma^N \omega}\left( B(u_i, r)  \right) \leq \mu_{\sigma^N \omega} \left( B(v_i, 2r)  \right) \leq C_{ \frac{2r}{r'}} \mu_{\sigma^N \omega} \left( B(v_i, r')  \right) \leq C_{ \frac{2r}{r'}} \mu_{\sigma^N \omega} \left( V_{j(i)} \right),$$
where we note that $C_{ \frac{2r}{r'}}$ does not depend on ${\sigma^N \omega}$.

Now, let $H\in C_{A|b|}$.  Then, as long as $|b|$ is large enough,
\begin{eqnarray*}
\int_{K_{\sigma^N \omega}} H d\mu_{\sigma^N \omega} &=&\sum_{i\in G} \int_{V_i} H d\mu_{\sigma^N \omega}\\
&\leq & \sum_{i\in G} \int_{B(u_i, r)} H d\mu_{\sigma^N \omega}\\
&\leq & \sum_{i\in G} \left( \max_{B(u_i,r)} H \right) \mu_{\sigma^N \omega} \left( B(u_i,r) \right) \\
&\leq & C'\sum_{i\in G} e^{2A|b|r} \left( \min_{V_{j(i)}} H\right) \mu_{\sigma^N \omega}\left( V_{j(i)} \right) \\
&\leq & C'e^{C''}\sum_{i\in G} \int_{V_{j(i)}} H d\mu_{\sigma^N \omega} \\
&\leq & C'e^{C''}\sum_{j: \exists i, (i,j)\in J} \int_{V_{j(i)}} H d\mu_{\sigma^N \omega} \\
&\leq &5 C'e^{C''}\int_{W_J} Hd\mu_{\sigma^N \omega}. \\
\end{eqnarray*}
Since the constant $C',C''$ are uniform (in particular, in $|b|$ and $\omega$), the proof is complete.
\end{proof}

\begin{Definition}
We define
$$\mathcal{E}{s,\omega} := \lbrace J\subseteq J_{s,\omega}:\, J \text{ is dense } \rbrace.$$
\end{Definition}
We can now prove the required contraction property in Lemma \ref{Lemma 5.4 Naud} Part (2). 
\begin{Proposition} \label{Prop 5.9 Naud} There exists $0<\alpha<1$ uniform in $\omega$ such that for all $s=a+ib$ with $|a|$ small and $|b|$ large, for all $H\in C_{A|b|}$ and all $J\in \mathcal{E}{s,\omega}$ we have
$$\int_{K_{\sigma^N \omega}} \left| N_s ^J \left( H \right) \right|^2  d\mu_{\sigma^N \omega} \leq \alpha \int_{K_{\sigma^N \omega}} P_{0,\omega,N} \left( H^2 \right) \,d \mu_{\sigma^N \omega}.$$

\end{Proposition}
This is a randomized analogue of \cite[Proposition 5.9]{Naud2005exp}.
\begin{proof}
Let $H\in C_{A|b|}$. For every $x\in [0,1]$, by the Cauchy-Schwartz inequality and the definition of $N_s ^J$,
$$\left( N_s ^J \left( H \right) \right)^2 (x)$$
is bounded above by the product of
\begin{equation} \label{Eq first prod}
\sum_{I\in X_N ^{(\omega)}} e^{a\cdot 2 \pi \cdot  c(I,x)} \eta ^{(\omega)} (I) \chi_J ^2 \circ f_{I}(x) 
\end{equation}
and
$$P_{a,\omega,N} \left( H^2 \right)(x).$$
For all $x\in W_J$ there exists $i\in \lbrace 1,2\rbrace$ such that
$$\chi_J \circ f_{\alpha_i ^N} (x) = 1-\theta.$$
Recall that for every $I\in X_N ^{(\omega)}$ we have by \eqref{Eq C and C prime}
$$\left| a\cdot c(I,x) \right| \leq |a|\cdot D' \cdot N.$$
Therefore, recalling the notations from Section \ref{Section dis}, if $x\in W_J$ then the sum in \eqref{Eq first prod} is bounded by 
$$e^{|a|\cdot D' \cdot N}-\theta\cdot e^{N\left( \min_{i\in I, 1\leq j \leq k_i} \log \mathbf{p}_j ^{(i)} +|a|\cdot D' \right)}.$$

Since 
$$\int_{K_{\sigma^N \omega}} \left( N_s ^J \left( H \right) \right)^2  d \mu_{\sigma^N \omega} = \int_{W_J} \left( N_s ^J \left( H \right) \right)^2 d\mu_{\sigma^N \omega} + \int_{K_{\sigma^N \omega}\setminus W_J} \left( N_s ^J \left( H \right) \right)^2 d\mu_{\sigma^N \omega}$$
the previous discussion and the fact that $e^{|a|\cdot D' \cdot N}$ always bounds \eqref{Eq first prod}  show that
$$\int_{K_{\sigma^N \omega}} \left( N_s ^J \left( H \right) \right)^2 d\mu_{\sigma^N \omega} \leq  \left(e^{|a|\cdot D' \cdot N}-\theta\cdot e^{N\left( \min_{i\in I} \log \mathbf{p}_j ^{(i)} +|a|\cdot D' \right)} \right) \int_{W_J} P_{a,\omega,N} \left( H^2 \right) d \mu_{\sigma^N \omega}$$
$$ + e^{|a|\cdot D' \cdot N}\cdot \int_{K_{\sigma^N \omega}\setminus W_J} P_{a,\omega,N} \left( H^2 \right) d \mu_{\sigma^N \omega}$$
$$ = e^{|a|\cdot D' \cdot N}\cdot \int_{K_{\sigma^N \omega}} P_{a,\omega,N} \left( H^2 \right) d \mu_{\sigma^N \omega} - \theta\cdot e^{N\left( \min_{i\in I} \log \mathbf{p}_j ^{(i)} +|a|\cdot D' \right)}  \int_{W_J} P_{a,\omega,N} \left( H^2 \right) d \mu_{\sigma^N \omega}.$$

Now, by Lemma \ref{Lemma cone} $P_{a,\omega,N} \left( H^2 \right)\in C_{\frac{3}{4}A|b|}$ since $H\in C_{a|b|}$.   Therefore, applying Lemma \ref{Lemma 5.7 Nuad} to this function we obtain
$$\int_{K\sigma^N \omega} \left( N_s ^J \left( H \right) \right)^2 d \mu_{\sigma^N \omega} \leq \left( e^{|a|\cdot D' \cdot N}- \tilde{\epsilon}\cdot \theta\cdot e^{N\left( \min_{i\in I} \log \mathbf{p}_j ^{(i)} +|a|\cdot D' \right)} \right) \int_{K_{\sigma^N \omega}} P_{a,\omega,N} \left( H^2 \right) d\mu_{\sigma^N \omega}.$$
Recall that
$$P_{a,\omega,N} \left( H^2 \right) \leq e^{N\cdot D' \cdot |a|} P_{0,\omega,N}\left( H^2 \right).$$
So, if $|a|$ is small enough there is some $0\leq \alpha <1$ such that
$$\left( e^{|a|\cdot D' \cdot N}- \tilde{\epsilon}\cdot \theta\cdot e^{N\left(  \min_{i\in I} \log \mathbf{p}_j ^{(i)} +|a|\cdot D' \right)} \right) \cdot e^{N\cdot D' \cdot |a|} \leq \alpha <1.$$
Therefore
$$\int_{K_{\sigma^N \omega}} \left( N_s ^J \left( H \right) \right)^2 d \mu_{\sigma^N \omega} \leq \alpha \cdot \int_{K_{\sigma^N \omega}} P_{0,\omega,N} \left( H^2 \right) d\mu_{\sigma^N \omega},$$
as claimed.

\end{proof}

\subsubsection{Part 3: Domination of the Dolgopyat operators}
We now turn to Part (3) of Lemma \ref{Lemma 5.4 Naud}. First we need the following key Lemma:
\begin{Lemma} \label{Lemma 5.10 Naud} Let $H\in C_{A|b|}, f\in C^1 \left( [0,1] \right)$ be such that
$$\left| f \right| \leq H, \text{ and } \left| f' \right| \leq A|b|H.$$
For every $j=1,2$ define functions $\Theta_j:[0,1]\rightarrow \mathbb{R}_+$ via
$$\Theta_1 (x):= \frac{ \left| e^{(a+ib)\cdot 2 \pi \cdot  c(\alpha_1 ^N,x)} \eta^{(\omega)}(\alpha_1 ^N) f \circ f_{\alpha_1 ^N} (x)+e^{(a+ib)\cdot 2 \pi \cdot  c(\alpha_2 ^N,x)} \eta^{(\omega)}(\alpha_2 ^N) f\circ f_{\alpha_2 ^N} (x)     \right| }{ (1-2\theta)e^{a\cdot 2 \pi \cdot  c(\alpha_1 ^N,x)} \eta^{(\omega)}(\alpha_1 ^N) H \circ f_{\alpha_1 ^N} (x)+e^{a\cdot 2 \pi \cdot  c(\alpha_2 ^N,x)} \eta^{(\omega)}(\alpha_2 ^N) H\circ f_{\alpha_2 ^N} (x) }$$
$$\Theta_2 (x):= \frac{ \left| e^{(a+ib)\cdot 2 \pi \cdot  c(\alpha_1 ^N,x)} \eta^{(\omega)}(\alpha_1 ^N) f \circ f_{\alpha_1 ^N} (x)+e^{(a+ib)\cdot 2 \pi \cdot  c(\alpha_2 ^N,x)} \eta^{(\omega)}(\alpha_2 ^N) f\circ f_{\alpha_2 ^N} (x)     \right| }{ e^{a\cdot 2 \pi \cdot  c(\alpha_1 ^N,x)} \eta^{(\omega)}(\alpha_1 ^N) H \circ f_{\alpha_1 ^N} (x)+(1-2\theta)e^{a\cdot 2 \pi \cdot  c(\alpha_2 ^N,x)} \eta^{(\omega)}(\alpha_2 ^N) H\circ f_{\alpha_2 ^N} (x) }.$$
Then for $\theta$ and $\epsilon'$ small enough, and for all $|a|$ small enough and $|b|$ large enough, for all $j$ such that $V_j\cap K_{\sigma^N \omega} \neq \emptyset$ there exist $j'$ such that $|j'-j|\leq 2$ and $V_{j'}\cap K_{\sigma^N \omega} \neq \emptyset$ and some $i\in \lbrace 1,2\rbrace$ such that: For every $x\in V_{j'}$,
$$\Theta_i (x) \leq 1.$$
\end{Lemma}
This is a model version of \cite[Lemma 5.10]{Naud2005exp}. For the proof, we require the following basic Lemmas from \cite{Naud2005exp}:
\begin{Lemma} \cite[Lemma 5.11]{Naud2005exp} \label{Lemma 5.11 Naud}
 Let $Z\subseteq I$ be an interval with $|Z|\leq \frac{c}{|b|}$. Let $H,f$ be as in Lemma \ref{Lemma 5.10 Naud}. Then for all $c$ small enough, either $\left| f(u) \right| \leq \frac{3}{4}H(u)$ for all $u\in Z$, or $ \left| f(u) \right| \geq \frac{1}{4}H(u)$ for all $u \in Z$.
\end{Lemma}
In the following arguments, for $0\neq z\in \mathbb{C}$ let $\arg (z)\in (-\pi,\, \pi]$ be the unique real number such that $|z|e^{i\arg(z)}=z$.
\begin{Lemma} \cite[Lemma 5.12]{Naud2005exp} \label{Lemma 5.12 Naud} Let $0\neq z_1,z_2 \in \mathbb{C}$ be such that
$$\frac{|z_1|}{|z_2|} \leq M \text{ and } 0< \epsilon\leq \left| \arg (z_1) - \arg(z_2) \right| \leq 2\pi -\epsilon.$$
Then there exists some $0<\delta(M,\epsilon)<1$ such that
$$\left| z_1 + z_2 \right| \leq (1-\delta)|z_1| +|z_2|.$$

\end{Lemma}
$$ $$
\noindent{ \textbf{Proof of Lemma \ref{Lemma 5.10 Naud}}} First, we choose $\epsilon'$ small enough so that the conclusion of Lemma \ref{Lemma 5.11 Naud} is true for all $Z=Z_j ^i=f_{\alpha_i ^N} \left(V_j\right)$, and one checks that this does not change $A$ and $N$ (see the end of the proof of Lemma \ref{Lemma cone}, and recall that $f_{\alpha_i ^N}$ are contractions). It is clear that $|Z_i ^j| \leq |V_j|$. We add the assumption that
$$0<\theta \leq \frac{1}{8} \text{ so that } 1-2\theta \geq \frac{3}{4}.$$
Let $V_j, V_{j+1}, V_{j+2}$ be a triad of intervals such that each of them intersects $K_{\sigma^N \omega}$. Write
$$\widehat{V_j} = V_j \cup V_{j+1} \cup V_{j+2}.$$

If for some $(i,j') \in \lbrace 1,2 \rbrace \times\lbrace j,j+1,j+2\rbrace$  we have $\left| f(u) \right| \leq \frac{3}{4}H(u)$ for all $u\in Z_{j'} ^i$ then $\Theta_i (u)\leq 1$ for all $u\in Z_{j'} ^i$, and we are done.

Otherwise, by Lemma \ref{Lemma 5.11 Naud} we have for all $(i,j')\in \lbrace 1,2 \rbrace \times\lbrace j,j+1,j+2\rbrace$ and every $u\in Z_{j'} ^i$,
$$\left| f(u) \right| \geq \frac{1}{4}H(u).$$
We aim to make use of Lemma \ref{Lemma 5.12 Naud}. For every $x\in \widehat{V_j}$ set
$$z_1(x):=e^{(a+ib)\cdot 2 \pi \cdot  c(\alpha_1 ^N,x)} \eta^{(\omega)}(\alpha_1 ^N) f \circ f_{\alpha_1 ^N} (x),\quad z_2(x):=e^{(a+ib)\cdot 2 \pi \cdot  c(\alpha_2 ^N,x)} \eta^{(\omega)}(\alpha_2 ^N) f \circ f_{\alpha_2 ^N} (x).$$
Let 
$$M=4e^{2N\left(-  \min_{i\in I, 1\leq j \leq k_i} \log \mathbf{p}_j ^{(i)} +|a|\cdot D' \right) } e^{2A\epsilon' A_1}.$$
We claim that for  $j'\in \lbrace j,j+1,j+2\rbrace$, 
$$ \left| \frac{z_1(x)}{z_2(x)} \right| \leq M \text{ for all } x\in V_{j'},\, \text{ or } \left| \frac{z_2(x)}{z_1(x)} \right| \leq M \text{ for all } x\in V_{j'}.$$
Indeed, 
$$\frac{1}{4}e^{-2N\left( -\min_{i\in I, 1\leq j\leq k_i} \log \mathbf{p}_j ^{(i)}  +|a|\cdot D' \right) } \frac{H\circ \alpha_1 ^N(x)}{H\circ \alpha_2 ^N(x)} \leq \left| \frac{z_1(x)}{z_2(x)} \right| \leq  4e^{2N\left(- \min_{i\in I, 1\leq j\leq k_i} \log \mathbf{p}_j ^{(i)}  +|a|\cdot D' \right) } \frac{H\circ \alpha_1 ^N(x)}{H\circ \alpha_2 ^N(x)}. $$
If for some $x_0 \in V_{j'}$ we have $\frac{H\circ \alpha_1 ^N(x_0)}{H\circ \alpha_2 ^N(x_0)}\leq 1$ then for all $x\in V_{j'}$ we obtain
$$ \frac{H\circ \alpha_1 ^N(x)}{H\circ \alpha_2 ^N(x)} \leq \frac{e^{A A_1 \epsilon'} H\circ \alpha_1 ^N(x_0)  }{ e^{-A A_1 \epsilon'} H\circ \alpha_2 ^N(x_0) }\leq e^{2AA_1 \epsilon'},$$
and so $\left| \frac{z_1(x)}{z_2(x)} \right| \leq M$. If $\frac{H\circ \alpha_1 ^N(x)}{H\circ \alpha_2 ^N(x)}\geq 1$ for all $x \in V_{j'}$ then 
$$\left| \frac{z_2(x)}{z_1(x)} \right|  \leq 4e^{2N\left( \min_{i\in I} \log \mathbf{p}_j ^{(i)}  +|a|\cdot D' \right)} \leq M.$$

We next control the relative variations in the arguments of $z_1,z_2$. Since
$$\left| z_i (x) \right| \geq e^{-N\left( \min_{i\in I, 1\leq j\leq k_i} \log \mathbf{p}_j ^{(i)}  +a\cdot D \right) } \frac{1}{4} H\circ f_{\alpha_i ^N }(x)>0,\quad \text{ for all } x\in \widehat{V_j} \text{ and } i=1,2, $$
there exist two $C^1$ functions $L_i :\widehat{V_j}\rightarrow \mathbb{C}$ such that 
$$L_i ' (x) = \frac{z_i ' (x)}{z_i (x)} \text{ and } e^{L_i (x)} = z_i (x) \text{ for all } x\in \widehat{V_j}.$$
For one possible construction see \cite[Proof of Lemma 5.10]{Naud2005exp}. For all $ x\in \widehat{V_j}$ set
$$\Phi(x):= \Im \left(L_1 (x)) - \Im(L_2(x)\right).$$
Then for $ x\in \widehat{V_j}$
$$\Phi'(x) = \Im \left( \frac{z_1 ' (x)}{z_1(x)}- \frac{z_2'(x)}{z_2(x)} \right)$$
$$= b \frac{d}{dx} \left( \log f_{\alpha_1 ^N} ' - \log f_{\alpha_2 ^N} ' \right)(x)+ \Im \left( \frac{ \left( f\circ f_{\alpha_1 ^N} \right)' (x) }{f\circ f_{\alpha_1 ^N} (x)}  - \frac{\left( f\circ f_{\alpha_2 ^N} \right)' (x) }{f\circ f_{\alpha_2 ^N} (x)}  \right).$$
Now, by our assumptions on $f,H$
$$\left| \frac{ \left( f\circ f_{\alpha_1 ^N} \right)' (x) }{f\circ f_{\alpha_1 ^N} (x)}  - \frac{\left( f\circ f_{\alpha_2 ^N} \right)' (x) }{f\circ f_{\alpha_2 ^N} (x)}   \right|\leq 8A|b| \rho^N,$$
and so by Theorem \ref{Theorem disint} Part (5) for all $ x\in \widehat{V_j}$
$$m-8A\rho^N \leq \frac{\left| \Phi' (x) \right|}{|b|}\leq m'+8A\rho^N.$$
Let $x\in V_j, x' \in V_{j+2}$. By the mean value Theorem
$$\left( m-8A\rho^N  \right)A_1 ' \epsilon' \leq \left| \Phi(x)-\Phi(x')\right| \leq \left( m'+8A\rho^N \right)3A_1 \epsilon'.$$
So, if $N$ is large enough then there are constants $B_1,B_2>0$ such that, independently of $x,x'$ and $|b|$,
$$B_1 \epsilon' \leq \left| \Phi(x)-\Phi(x')\right| \leq B_2 \epsilon'.$$
We now pick $\epsilon'$ so that 
$$(B_2 + \frac{B_1}{2})\epsilon' \leq \pi,$$
and put $\epsilon=B_1 \frac{\epsilon'}{4}$. 

Suppose, towards a contradiction, that there are $x\in V_j$ and $x'\in V_{j+2}$ such that both
$$\Phi(x),\Phi(x')\in \bigcup_{k\in \mathbb{Z}} [ 2 k \pi-\epsilon, 2k \pi +\epsilon].$$ 
Since $\left| \Phi(x)-\Phi(x')\right| \leq B_2 \epsilon'$ we cannot have
$$\Phi(x) \in [ 2 k_1 \pi-\epsilon, 2k_1 \pi +\epsilon] \text{ and } \Phi(x') \in [ 2 k_2 \pi-\epsilon, 2k_2 \pi +\epsilon]$$
with $k_1\neq k_2$. Indeed, it will imply
$$2\pi - B_1 \frac{\epsilon'}{2}\leq 2\pi - 2\epsilon \leq \left| \Phi(x)-\Phi(x')\right|  \leq B_2 \epsilon',$$
a contradiction. But then 
$$ B_1 \epsilon'\leq \left| \Phi(x)-\Phi(x')\right|  \leq 2\epsilon = B_1 \epsilon'/2,$$
which is also a contradiction.

We conclude that there exists $j'\in \lbrace j,j+2\rbrace$ such that for all $x\in V_{j'}$,
$$\text{dist} \left( \Phi(x),\,2\pi \mathbb{Z} \right)>\epsilon.$$
Since $e^{i\Phi(x)} = e^{i \left( \arg(z_1)-\arg(z_2)\right)}$, the conditions of Lemma \ref{Lemma 5.12 Naud} are met. Thus, either for every $x\in V_{j'}$
$$\left| z_1(x)+z_2(x) \right| \leq \left(1-\delta(M,\epsilon)\right) \left| z_1(x) \right|+\left|z_2(x)\right|,$$
or for every $x\in V_{j'}$
$$\left| z_1(x)+z_2(x) \right| \leq \left(1-\delta(M,\epsilon) \right) \left| z_2(x) \right|+\left|z_1(x)\right|,$$
depending on whether
$$ \left| \frac{z_1(x)}{z_2(x)} \right| \leq M \text{ for all } x\in V_{j'},\, \text{ or } \left| \frac{z_2(x)}{z_1(x)} \right| \leq M \text{ for all } x\in V_{j'}.$$
Choosing $0<\theta<\frac{1}{2}\delta(M,\epsilon)$, we have $\Theta_i(x)\leq 1$ for all $x\in V_{j'}$ and some $i\in \lbrace 1,2\rbrace$. $\hfill{\Box}$
$$ $$

\noindent{ \textbf{Proof of Lemma \ref{Lemma 5.4 Naud} Part (3)} Fix constants $N,A,\epsilon',\theta$ so that parts (1) and (2) of Lemma \ref{Lemma 5.4 Naud}, and so that Lemma \ref{Lemma 5.10 Naud}, all hold true. Let $f\in C^1([0,1],H\in C_{A|b|}$ be such that
$$\left| f \right| \leq H,\text{ and } \left| f' \right| \leq A|b|H.$$
We aim to show that there exists a dense subset $J\in \mathcal{E}_{s,\omega}$ such that
$$|P_{s,\omega, N} f| \leq N_s ^J H \quad \text{ and } |(P_{s,\omega, N} f)'|\leq A|b|N_s ^J H.$$
Since the latter statement holds true for all $J\neq \emptyset$ by Lemma \ref{Lemma cone}, we focus on the first one.

Let
$$J:=\lbrace (i,j):\, \Theta_i(x)\leq 1 \text{ for all } x\in V_j\rbrace.$$
By Lemma \ref{Lemma 5.10 Naud} $J$ is dense. Recall that by Theorem \ref{Theorem disint} Part (4), for all $I\neq J \in X_N ^{(\omega)}$,
\begin{equation} \label{Eq separated UNI}
f_{I} ([0,1]) \cap f_{J} ([0,1]) = \emptyset.
\end{equation}

Let $x\in [0,1]$. If $x \notin \text{ Int } V_j$ for all  $(i,j)\in J$, then by the definition of $\chi_J$ and by \eqref{Eq separated UNI}, 
$$ \chi_J \circ f_{I} (x) =1 \text{ for all } I\in X_N ^{(\omega)}.$$
Therefore,
\begin{eqnarray*}
\left| P_{s,\omega,N} \left( f \right) (x) \right| &\leq & \sum_{I\in X_N ^{(\omega)}} e^{a\cdot 2 \pi \cdot  c(I,x)} \eta^{(\omega)}(I) H\circ f_I (x)\\
&=& \sum_{I\in X_N ^{(\omega)}} e^{a\cdot 2 \pi \cdot  c(I,x)} \eta^{(\omega)}(I) (H\cdot \chi_J) \circ f_I (x)\\
&=& N_s ^J \left( H \right) (x).
\end{eqnarray*}

Now, suppose $x \in \text{ Int }V_j$ for some  $(i,j)\in J$.

\textbf{Case 1} If $(1,j)\in J$ but $(2,j)\notin J$ then by \eqref{Eq separated UNI} $\chi_J \circ f_{I} (x) =1$ for all $I\neq \alpha_1 ^N$: Indeed, \eqref{Eq separated UNI} implies  that $f_{I}(x) \notin f_{\alpha_1 ^N}(V_j)$ since $f_{I}(x) \notin f_{\alpha_1 ^N}([0,1])$. Since $\Theta_1(x)\leq 1$ we find that
\begin{eqnarray*}
\left| P_{s,\omega,N} \left( f \right) (x) \right| &\leq & \sum_{I\neq \alpha_i ^N,  I\in   X_N ^{(\omega)}} e^{a\cdot 2 \pi \cdot  c(I,x)} \eta^{(\omega)}(I) H\circ f_I (x) + \\
&&(1-2\theta)e^{a\cdot 2 \pi \cdot  c(\alpha_1 ^N,x)} \eta^{(\omega)}(\alpha_1 ^N) H \circ f_{\alpha_1 ^N} (x)+e^{a\cdot 2 \pi \cdot  c(\alpha_2 ^N,x)} \eta^{(\omega)}(\alpha_2 ^N) H  \circ f_{\alpha_2 ^N} (x)\\
&\leq & \sum_{I\neq \alpha_i ^N, I\in   X_N ^{(\omega)}} e^{a\cdot 2 \pi \cdot  c(I,x)} \eta^{(\omega)}(I) \left( H\cdot \chi_J \right) \circ f_I (x) + \\
&&e^{a\cdot 2 \pi \cdot  c(\alpha_1 ^N,x)} \eta^{(\omega)}(\alpha_1 ^N) (H\cdot \chi_J)  \circ f_{\alpha_1 ^N} (x)+e^{a\cdot 2 \pi \cdot  c(\alpha_2 ^N,x)} \eta^{(\omega)}(\alpha_2 ^N) (H\cdot \chi_J)   \circ f_{\alpha_2 ^N} (x)\\
&=& N_s ^J \left( H \right) (x).
\end{eqnarray*}
The case $(2,j)\in J$ but $(1,j)\notin J$ is similar.

\textbf{Case 2} Suppose now both $(1,j)\in J$ and $(2,j)\in J$. Then both $\Theta_1(x),\Theta_2 (x) \leq 1$. Taking half of the sum of these inequalities, we deduce that
$$ \left| e^{(a+ib)\cdot 2 \pi \cdot  c(\alpha_1 ^N,x)} \eta^{(\omega)}(\alpha_1 ^N) f \circ f_{\alpha_1 ^N} (x)+e^{(a+ib)\cdot 2 \pi \cdot  c(\alpha_2 ^N,x)}\eta^{(\omega)}(\alpha_2 ^N) f\circ f_{\alpha_2 ^N} (x)     \right| $$
$$\leq \left(1-\theta \right)e^{a\cdot 2 \pi \cdot  c(\alpha_1 ^N,x)} \eta^{(\omega)}(\alpha_1 ^N) H \circ f_{\alpha_1 ^N} (x)+\left(1-\theta \right)e^{a\cdot 2 \pi \cdot  c(\alpha_2 ^N,x)} \eta^{(\omega)}(\alpha_2 ^N) H\circ f_{\alpha_2 ^N} (x) $$
$$\leq e^{a\cdot 2 \pi \cdot  c(\alpha_1 ^N,x)} \eta^{(\omega)}(\alpha_1 ^N) (H\cdot \chi_J)  \circ f_{\alpha_1 ^N} (x)+e^{a\cdot 2 \pi \cdot  c(\alpha_2 ^N,x)}\eta^{(\omega)}(\alpha_2 ^N) (H\cdot \chi_J)   \circ f_{\alpha_2 ^N} (x).$$
Since for every $I\in X_N ^{(\omega)}$ with $I\neq \alpha_i ^N$ we have $\chi_J \circ f_{I} (x) =1$, it follows that, similarly to case 1,
$$\left| P_{s,\omega,N} \left( f \right) (x) \right|\leq N_s ^J \left( H \right) (x).$$
The proof is complete. $\hfill{\Box}$

\section{From spectral gap to a renewal theorem with an exponential error term} \label{Section renewal}
We keep our notations and assumptions from Section \ref{Section non-linearity implies s-gap}. In particular, we are working with the induced IFS from Claim \ref{Claim key}, so that Theorem \ref{Theorem spectral gap} holds as stated. However, from this point forward we will no longer require the model constructed in Theorem \ref{Theorem disint}. Recall also  that $\Phi$ is $C^2$, and that $\nu_\mathbf{p}$ is a self-conformal measure such that $\mathbf{p}$ is a strictly positive probability vector. Recall that $\mathbb{P}=\mathbf{p}^\mathbb{N}$ is our Bernoulli measure on $\mathcal{A}^\mathbb{N}$.

 First, we record the following consequence of Theorem \ref{Theorem spectral gap}; More precisely, only the third part follows from Theorem \ref{Theorem spectral gap}, the rest are already well known in our setting. For every $t>0$ let 
$$\mathbb{C}_t = \lbrace z\in \mathbb{C}:\, |\mathcal{R}z|<t\rbrace.$$ 
\begin{theorem} \label{Theorem U}
Let $\epsilon,\gamma>0$ be as in Theorem \ref{Theorem spectral gap}. 
\begin{enumerate}
\item \cite[Lemma 11.17]{Benoist2016Quint} The transfer operator $P_{s+i\theta}$ is a bounded operator on $C^1([0,1])$ for every $|s|<\epsilon$ and $\theta$, that depends analytically on $s+i\theta$.

\item \cite[Proposition 4.1]{Li2018decay}  There exists an analytic operator $U(s+i\theta )$  for $s+i\theta \in \mathbb{C}_\epsilon$ on $C^1([0,1])$ such that, for $s+i\theta \in \mathbb{C}_\epsilon$,
$$(I-P_{-(s+i\theta)})^{-1} = \frac{1}{\chi(s+i\theta)}N_0+U(s+i\theta ),\, \text{ where } N_0(f)=\int f \, d \mathbb{P}.$$
\item \cite[Proposition 4.26]{li2018fourier} There exists some $C>0$ such that for all $s+i\theta \in \mathbb{C}_\epsilon$,
$$||U(s+i\theta )||\leq C(1+\left| \theta \right| )^{1+\gamma}, \, \text{ where the operator norm is defined via \eqref{Eq B-Q norm}.} $$
\end{enumerate}
\end{theorem}
Note that in \cite[Section 4]{Li2018decay} the transfer operator has an additional minus in the exponent, which explains the left hand side of the equation in Part (2) in our setting.  Part (3) was treated by Li in \cite{li2018fourier} in a related setting to ours (cocycles that arise from  actions of algebraic groups on certain projective spaces). It follows from the analyticity of $U$ when $|\theta|<R$ for the $R>0$ as in Theorem \ref{Theorem spectral gap}, and otherwise from a direct application of Theorem \ref{Theorem spectral gap} via the  identity
$$(I-P_{s+i\theta})^{-1} = \sum_{n=0} ^\infty P_{s+i\theta} ^n.$$
We refer to the discussion in  \cite[Proposition 4.26]{li2018fourier} for more details.

We now define a renewal operator as follows: For a non-negative bounded function $f$ on $[0,1] \times \mathbb{R}$ and $(z,t)\in [0,1]\times \mathbb{R}$, define
$$Rf(z,\, t):= \sum_{n=0} ^\infty \int f( \eta.z,\,  \sigma(\eta,\, z) -t) d\mathbf{p}^n (\eta).$$
Since $f$ is positive, this sum is well defined. For every $(z,x)\in [0,1]\times \mathbb{R}$ let $f_z :\mathbb{R} \rightarrow \mathbb{R}$  be the function
$$f_z (x):=f(z,x).$$
We also define, for $(z,\theta) \in [0,1]\times \mathbb{C}$ the Fourier transform
$$\hat{f}(z,\theta) := \int e^{i \theta u} f(z,u)\,du=\hat{f_z}(\theta).$$
The following Proposition is the main result of this Section:
\begin{Proposition} \label{Proposition renewal}
Let $\epsilon$ be as in Theorem \ref{Theorem spectral gap}, and let $C\subseteq \mathbb{R}$ be a compact set. Fix a non-negative bounded and continuous function $f(y,x)$ on $[0,1] \times \mathbb{R}$ such that $f_y \in C_C ^{3} (\mathbb{R})$ for every $y\in [0,1]$, and $\hat{f}_z (\theta) \in C^1([0,1])$ for every fixed $\theta\in \mathbb{R}$. Assume
$$\sup_{y}  \left( |f_y ^{(3)}|_{L^1} + |f_y |_{L^1} \right)<\infty.$$
Then  for every $z\in [0,1]$ and $t\in \mathbb{R}$
$$Rf(z, t) = \frac{1}{\chi} \int_{[0,1]} \int_{-t} ^\infty f(y,u)\,du d\nu(y) + e^{-\epsilon\cdot |t|}O\left(e^{\epsilon | \supp f_z|} \left(  |( f_z)^{(3)}|_{L^1} +  |f_z|_{L^1} \right) \right)$$
where $|\text{supp}(f_z)|$ is the supremum of the absolute value of the elements of $\supp(f)$.
\end{Proposition}
We remark that Proposition \ref{Proposition renewal} is an IFS-type analogue of the renewal Theorem of Li \cite[Theorem 1.1 and Proposition 4.27]{li2018fourier}. Similarly to Li, we derive it from our spectral gap result Theorem \ref{Theorem spectral gap} via Theorem \ref{Theorem U}.
\begin{proof}
We begin by arguing, similarly to \cite[Lemma 4.6]{Li2018decay} and \cite[Proposition 4.14]{boyer2016rate}, that for every $(z,t)\in [0,1]\times \mathbb{R}$,
\begin{equation} \label{Eq Lemma 4.6 Li}
Rf(z, t) = \frac{1}{\chi} \int_{[0,1]}  \int_{-t} ^\infty f(y,u)\,du d\nu(y) + \lim_{s\rightarrow 0^+} \frac{1}{2\pi} \int e^{-it \theta}  U(s-i\theta)\hat{f}(z,\theta)\,d\theta,
\end{equation}

Indeed, for $s\geq 0$ write
$$B_s f(z,t)= \int e^{-s \sigma(i,\, z)} f(i.z,\, \sigma(i,z)+t)d\mathbf{p}(i),$$
and put $B:=B_0$. Note that 
$$Rf(z,-t) = \sum_{n\geq 0} B^n f(z,t).$$
By the monotone convergence Theorem, since $f\geq 0$ and $\sigma(\eta, z)\geq 0$ we have
$$\lim_{s\rightarrow 0^+} \sum_{n\geq 0 } \int e^{-s \sigma(\eta,\, z)} f(\eta.z,\, \sigma(\eta,\,z)+t) d\mathbf{p}^n (\eta) = \sum_{n\geq 0 } \int  f(\eta.z,\, \sigma(\eta,\,z)+t) d\mathbf{p}^n (\eta).$$
Thus,
\begin{equation} \label{Eq (4.3) Li}
\sum_{n\geq 0} B^n (f)(z,t) = \lim_{s\rightarrow 0^+} \sum_{n\geq 0 } B_s ^n (f) (z,t).
\end{equation}

Recall that $f_y \in C_C ^{3} (\mathbb{R}) \subset L^1 (\mathbb{R})$ for every $y\in [0,1]$. So, using the inverse Fourier transform, we have
$$\sum_{n\geq 0 } B_s ^n (f) (z,t) = \sum_{n\geq 0 } \int e^{-s \sigma(\eta,\, z)} f(\eta.z,\, \sigma(\eta,\,z)+t) d\mathbf{p}^n (\eta)$$
\begin{equation} \label{Eq (4.4) Li}
= \sum_{n\geq 0 } \int e^{-s \sigma(\eta,\, z)} \frac{1}{2\pi} \int_\mathbb{R} e^{ i\theta( \sigma(\eta,z)+t)} \hat{f}(\eta.z,\, \theta) d\theta d\mathbf{p}^n (\eta).
\end{equation}
We now argue that the sum in  \eqref{Eq (4.4) Li} is absolutely convergent: Since $f_y$ is compactly supported for every $y\in [0,1]$, $\hat{f_y}\in C^\omega (\mathbb{C})$. Note that for every $y\in [0,1]$, $k\in \mathbb{N}\cup \lbrace 0 \rbrace$, $\delta\geq 0$, and $\theta\in \mathbb{R}$,
$$\left| \hat{f_y}(\pm i\delta+\theta) \right| \leq e^{\delta | \supp f_y|} \frac{1}{|\theta|^k} |(f_y)^{(k)}|_{L^1}$$
So, plugging in $k=0$ and $k=3$, 
\begin{equation} \label{Eq 4.60}
\left|  \hat{f_y}(\pm i\delta+\theta) \right|  \leq e^{\delta | \supp f_y|} \frac{2}{|\theta|^{3}+1} \left( |(f_y)^{(3)}|_{L^1} + |f_y|_{L^1} \right)
\end{equation}

Thus, as long as $s>0$, putting $C=\int_\mathbb{R} \frac{2}{|\theta|^{3}+1} d\theta$,
$$ \sum_{n\geq 0 } \int e^{-s \sigma(\eta,\, z)} \int_\mathbb{R} | \hat{f}(\eta.z,\, \theta)| d\theta d\mathbf{p}^n (\eta) \leq  C\cdot e^{\delta | \supp f_y|} \cdot  \sup_{y}  \left( |f_y ^{(3)}|_{L^1} + |f_y |_{L^1} \right) \cdot  \sum_{n\geq 0 } \int e^{-s \sigma(\eta,\, z)}\cdot d\mathbf{p}^n (\eta) $$
$$\leq  C\cdot e^{\delta | \supp f_y|} \cdot  \sup_{y}  \left( |f_y ^{(3)}|_{L^1} + |f_y|_{L^1} \right) \cdot  \sum_{n\geq 0 }  e^{-s D\cdot n}<\infty.$$
It follows that the sum in  \eqref{Eq (4.4) Li} is absolutely convergent. Thus, we can use Fubini's Theorem to change the order of integration. Since for every $\theta \in \mathbb{R}$ $\hat{f}(\cdot,\theta) \in C^1 ([0,1])$, we can apply Theorem \ref{Theorem U} to obtain
\begin{eqnarray*}
\sum_{n\geq 0 } B_s ^n (f) (z,t) &=& \frac{1}{2\pi} \int_\mathbb{R} \sum_{n\geq0} \int e^{(-s+ i\theta)\sigma(\eta,z)} \hat{f}(\eta.z,\, \theta)d\mathbf{p}^n (\eta) e^{it\theta} d\theta \\
 &=& \frac{1}{2\pi} \int_\mathbb{R} \sum_{n\geq0} P_{-s+ i\theta} ^n  \hat{f}(z,\, \theta) e^{it\theta} d\theta \\ 
 &=& \frac{1}{2\pi} \int_\mathbb{R} (1-P_{-s+ i\theta})^{-1}  \hat{f}(z,\, \theta) e^{it\theta} d\theta \\
  &=& \frac{1}{2\pi} \int_\mathbb{R} (1-P_{-(s- i\theta)})^{-1}  \hat{f}(z,\, \theta) e^{it\theta} d\theta \\
   &=& \frac{1}{2\pi} \int_\mathbb{R} \left( \frac{1}{\chi(s- i\theta)}N_0+U(s- i\theta ) \right)  \hat{f}(z,\, \theta) e^{it\theta} d\theta. \\
\end{eqnarray*}
Since $\frac{1}{s-i\theta} = \int_0 ^\infty e^{-(s-i\theta)u} du$ for $s>0$, and since  $\hat{f_y} \in L^1 (\mathbb{R})$ for all $y$ by \eqref{Eq 4.60}, we have
\begin{eqnarray*}
\frac{1}{2\pi} \int_\mathbb{R} \frac{N_0}{\chi(s-i\theta)}  \hat{f}(z,\, \theta) e^{it\theta} d\theta &=& \frac{1}{2\pi} \frac{1}{\chi} \int_{[0,1]} \int_\mathbb{R} \frac{\hat{f}(y,\, \theta)}{s-i\theta} e^{it\theta} d\theta d\nu(y)\\
&=& \frac{1}{\chi} \int_{[0,1]} \int_0 ^\infty f(y,\,u+t) e^{-su} du d\nu(y).\\
\end{eqnarray*}
When $s\rightarrow 0^+$, since $f$ is integrable with respect to $\nu\times du$, this converges to $$\frac{1}{\chi} \int \int_0 ^\infty f(y,\,u+t)  du d\nu(y)$$  by monotone convergence. Taking tally of our computations, \eqref{Eq Lemma 4.6 Li} is proved.

Furthermore, using the bound on the norm of $U$ from Theorem \ref{Theorem U} and   \eqref{Eq 4.60},  by dominated convergence
$$\lim_{s\rightarrow 0^+} \frac{1}{2\pi} \int e^{-it \theta}  U(s-i\theta)\hat{f}(z,\theta)\,d\theta =  \frac{1}{2\pi} \int e^{-it \theta}  U(-i\theta)\hat{f}(z,\theta)\,d\theta.$$
Define
$$T_z(\theta) = U(-i\theta)\hat{f}(z,\theta).$$
Then $T_z$ is a tempered distribution with an analytic continuation to $|\Im \theta |<\epsilon$ by Theorem \ref{Theorem spectral gap} and Theorem \ref{Theorem U}, such that for any $|b|<\epsilon$ we have $T_z(\cdot + ib)\in L^1$. Furthermore, by Theorem \ref{Theorem U} and \eqref{Eq 4.60}, for any $|b|<\epsilon$, $\sup_{\eta <|b|} ||T(\cdot + i\eta)||_{L^1}<\infty$ (see the computation below).  Applying \cite[Lemma 4.28]{li2018fourier} we have, for all $0<\delta<\epsilon$
$$\left| \frac{1}{2\pi} \int e^{-it \theta}  U(-i\theta)\hat{f}(z,\theta)\,d\theta \right| = \left| \check{T}(t) \right| \leq e^{-\delta |t|} \max \left| T(\pm i \delta+\theta) \right|_{L^1 (\theta)}.$$
Finally, by \eqref{Eq 4.60} and Theorem \ref{Theorem U},
$$\max \left|  U(-i\theta\pm\delta)\hat{f}(z, \theta\pm i\delta) \right|_{L^1 (\theta)} \leq  \int \left||U(\pm \delta-i\theta)\right||\cdot  \frac{e^{\delta | \supp f_z|}}{|\theta|^{3}+1}  \left( |(f_z)^{(3)}|_{L^1} + |f_z|_{L^1} \right) \,d\theta $$
$$ \leq C\cdot  e^{\delta | \supp f_z|} \left( |(f_z)^{(3)}|_{L^1} + |f_z|_{L^1} \right) \int \frac{1}{|\theta|^{3}+1}\cdot (1+|\theta|)^{1+\gamma}\, d\theta = O\left(e^{\delta | \supp f_z|} \left( |(f_z)^{(3)}|_{L^1} + |f_z|_{L^1} \right) \right), $$
where in the last equality we used that $3-(1+\gamma)>1$. Via \eqref{Eq Lemma 4.6 Li} and the preceding paragraph, the proof is complete.

\end{proof}

\section{From the renewal Theorem to equidistribution} \label{Section equi}
We keep our notations and assumptions from Section \ref{Section non-linearity implies s-gap}. In particular, we are working with the induced IFS from Claim \ref{Claim key} so that Theorem \ref{Theorem spectral gap}, and therefore Proposition \ref{Proposition renewal}, hold as stated. Recall also  that $\Phi$ is $C^2$, and that $\nu_\mathbf{p}$ is a self-conformal measure such that $\mathbf{p}$ is a strictly positive probability vector. Recall that $\mathbb{P}=\mathbf{p}^\mathbb{N}$ is our Bernoulli measure on $\mathcal{A}^\mathbb{N}$.

Recall the definition of the semigroup $G$ from Section \ref{Section pre}.  In this Section we discuss an effective equidistribution result for a certain random walk, driven by a symbolic version the derivative cocycle $\tilde{c}$: It is defined as $\tilde c:G\times \mathcal{A}^\mathbb{N} \rightarrow \mathbb{R}$ 
\begin{equation} \label{The symbolic der cocycle}
\tilde c(I,\omega)=-\log f'_I(x_\omega).
\end{equation}

We can now define a random walk as follows: Recalling that $\sigma: \mathcal{A}^\mathbb{N} \rightarrow \mathcal{A}^\mathbb{N}$ is the left shift, for every $n\in \mathbb{N}$ we define a function  on $\mathcal{A}^\mathbb{N}$ via
\begin{equation} \label{Eq for Sn}
S_n(\omega)=-\log f'_{\omega|_n}(x_{\sigma^n(\omega)}).
\end{equation}
Let $X_1: \mathcal{A}^\mathbb{N}\rightarrow \mathbb{R}$ be the random variable
\begin{equation} \label{Eq for X1}
X_1(\omega):= \tilde c(\omega_1,\sigma(\omega))=-\log f'_{\omega_1}(x_{\sigma(\omega)}).
\end{equation}
For every integer $n>1$ we define 
$$X_n(\omega) =  - \log f_{\omega_n} ' \left( x_{\sigma^n (x_\omega)} \right)  = X_1 \circ \sigma^{n-1}.$$
Let $\kappa$ be the law of the random variable $X_1$. Then for every $n$, $X_n \sim \kappa$.  By uniform contraction  there exists $D,D'\in \mathbb{R}$ as in \eqref{Eq C and C prime},  
so $\kappa \in \mathcal{P}([D,D'])$. In particular, the support of $\kappa$ is bounded away from $0$. It is easy to see that for every $n \in \mathbb{N}$ and $\omega\in A^\mathbb{N}$ we have
$$S_n(\omega) = \sum_{i=1} ^n X_i (\omega).$$
Thus, in this sense $S_n$ is a random walk.

Next, we define a function  on $\mathcal{A}^\mathbb{N}$ that resembles a stopping time: For $k>0$ let
$$ \tau_k(\omega):=\min\{n: S_n(\omega) \geq k\}.$$
We emphasize that $k$  is allowed to take non-integer values. We also recall that $\chi$ is the corresponding Lyapunov exponent.  Recalling \eqref{Eq C and C prime}, it is clear that for every $k>0$ and $\omega\in A^\mathbb{N}$ we have
$$-\log |f_{\omega|_{\tau_k(\omega)}}'(x_{\sigma^{\tau_k(\omega)}(\omega)})|= S_{\tau_k(\omega)} (\omega) \in [k, k+D'].$$
We also define a local $C^3$ norm on $[-1-D, D'+1]$ by, for $g\in C^3(\mathbb{R})$,
\begin{equation}\label{Def local norm}
||g||_{C^3} = \max_{i=0,1,2,3} \max_{x\in [-1-D, D'+1]} \left|g^{(i)}(x)\right|.
\end{equation}
The following Theorem is an effective equidistribution result for the random variables $S_{\tau_k}$:
\begin{theorem} \label{Theorem equi} Let $\epsilon>0$  be as in Proposition \ref{Proposition renewal}. Then for every $k>D'+1$, $g\in C^{3}  (\mathbb{R})$, and $\omega \in \mathcal{A}^\mathbb{N}$,
$$\mathbb{E}\left( g \left( S_{\tau_k(\eta)} (\eta)-k \right) \big|\, \sigma^{\tau_k(\eta)} \eta  =\omega \right) = \frac{1}{\chi} \int_{D} ^{D'} \int_{-y} ^0 g(x+y)\,dxd\kappa(y)+e^{-\frac{\epsilon k}{2}} O\left(  ||g||_{C^3} \right).$$
\end{theorem}
Theorem \ref{Theorem equi} is  a version of results of Li and Sahlsten \cite[Propositions 2.1 and  2.2]{li2019trigonometric}, that has  better (exponential) error terms due to our non-linear setting. Indeed, Li-Sahlsten work with self-similar IFS's, where $S_{\tau_k}$ necessarily has slower equidistribution rates (no faster than polynomial is possible). This also makes our treatment  more complicated since the random walk $S_n$ is not IID as in their work. Another point of difference is that our renewal operator from Section \ref{Section renewal}, that is used critically for the proof of Theorem \ref{Theorem equi} via Proposition \ref{Proposition renewal}, is defined for functions on $[0,1]\times \mathbb{R}$, rather than just on $\mathbb{R}$. Nonetheless, we will follow along the scheme of proof as in \cite[Section 4]{li2019trigonometric},  and deduce Theorem \ref{Theorem equi} from Proposition \ref{Proposition renewal}.
\subsection{Regularity properties of renewal measures}
Let $\psi$ be the even smooth bump function given by
$$\psi (x) :=C_0\cdot \exp\left( \frac{-1}{1-x^2} \right) \text{ if  } x\in (-1,1), \text{ and } 0 \text{  otherwise.}$$
Here $C_0$ is chosen so that $\psi$ is also a probability density function. Next, for $\delta\neq0$ write 
$$\psi_\delta (x):= \frac{1}{\delta^2} \psi(\frac{x}{\delta^2}).$$
The following Proposition is an analogue of \cite[Proposition 4.5]{li2019trigonometric}:
\begin{Proposition} \label{Proposition 4.5}  There exists some $C_1=C_1 (\psi)>0$ such that, for our $\epsilon>0$ from Theorem \ref{Theorem spectral gap}:

For all $k>0$, $z\in [0,1]$, $\delta\leq 1$, and $b,a$ such that $b-a\geq 2\delta$,
$$R\left( 1_{[a,b]} \right)(z, k) \leq 3(b-a) \left(\frac{1}{\chi}+ e^{\epsilon\left(-k+b-a+2\right)} C_\psi O\left(\frac{1}{\delta^{6}} + 1 \right) \right),$$
where the function $1_{[a,b]}(z,x)= 1_{[a,b]}(x)$ is constant in $z$ (the $[0,1]$ coordinate).

\end{Proposition}
Note that our error term is upgraded compared with \cite[Proposition 4.5]{li2019trigonometric}, due to the exponential error term in Proposition \ref{Proposition renewal}.
\begin{proof}
If $x\in [a,b]$ then $[x-b,x-a]$ contains either $[0,\delta]$ or $[-\delta,0]$. Therefore, as $\psi_\delta$ is even
$$\psi_\delta * 1_{[a,b]}(x) = \int_a ^b \psi_\delta (x-v)dv\geq  \int_0 ^{\delta} \psi_\delta (v)dv = \frac{1}{2}.$$
So, as  functions on $\mathbb{R}$,
$$1_{[a,b]}\leq 3 \psi_\delta * 1_{[a,b]}.$$
We proceed to bound $R(\psi_\delta * 1_{[a,b]})$, where we note that $\psi_\delta * 1_{[a,b]}$ is a compactly supported function in $C^3 (\mathbb{R})$. By Proposition \ref{Proposition renewal}, since our function is constant in the $[0,1]$-coordinate,
$$R(\psi_\delta * 1_{[a,b]})(z,k) = \frac{1}{\chi}\int_{-k} ^\infty \psi_\delta * 1_{[a,b]}(x)dx  +e^{-\epsilon\cdot |k|} O\left(e^{\epsilon | \supp \psi_\delta * 1_{[a,b]} |} \left( |(\psi_\delta * 1_{[a,b]})^{(3)}|_{L^1} + |\psi_\delta * 1_{[a,b]}|_{L^1} \right) \right).$$
Up to global multiplicative constants, the first and third terms are less than $\int \psi_\delta * 1_{[a,b]}(x)dx =(b-a)$. For the second term, recall that 
$$(\psi_\delta * 1_{[a,b]})^{(3)}=\psi_\delta ^{(3)} * 1_{[a,b]}$$
and so
$$||(\psi_\delta * 1_{[a,b]})^{3)}||_{L^1}\leq ||\psi_\delta ^{(3)}||_{L^1} \cdot  ||1_{[a,b]}||_{L^1}\leq C_\psi \cdot \frac{1}{\delta^{2\cdot 3}}\cdot (b-a).$$
Finally, $\psi_\delta * 1_{[a,b]}$ is supported on $[a,b]+[-\delta^2,\delta^2]\subseteq [a-1,b+1]$. This yields the desired bound.
\end{proof}
We will also require the following Lemma from \cite{li2019trigonometric}, that follows directly from Proposition \ref{Proposition 4.5} in our case:
\begin{Lemma} \label{Lemma 4.6} \cite[Lemma 4.6]{li2019trigonometric}
There is some $C>0$ such that for all $s$, $z\in [0,1]$, and $k\in \mathbb{R}$ we have
$$R(1_{[0,s]})(z, k)\leq C\cdot \max\lbrace 1,s\rbrace,$$
where the function $1_{[0,s]}(y,x)= 1_{[0,s]}(x)$ is constant in the $[0,1]$ coordinate.
\end{Lemma}

\subsection{Residue process} Let $f$ be a non-negative bounded Borel function on $[D,D']\times \mathbb{R}$. For $k\in \mathbb{R}$ and $z\in [0,1]$ define the residue operator by
$$Ef(z, k) := \sum_{n\geq 0} \int \int f(\sigma(i, \eta.z),\, \sigma(\eta,z)-k) d\mathbf{p}^n (\eta)d\mathbf{p}(i).$$
Recall that for every $y\in [0,1]$ we let $f_y :\mathbb{R} \rightarrow \mathbb{R}$ be the function 
$$f_y (x):=f(y,x).$$ 
The following Proposition is an analogue of \cite[Proposition 4.7]{li2019trigonometric}:
\begin{Proposition} \label{Proposition 4.7} (Residue process) Let  $C\subseteq \mathbb{R}$ be a compact set. Fix a non-negative bounded, compactly supported, and continuous function $f(y,x)$ on $[D,D'] \times \mathbb{R}$ such that $f_y \in C_C ^{3} (\mathbb{R})$ for every $y\in [D,D']$, and $\hat{f}_z (\theta) \in C^1([D,D'])$ for every $\theta \in \mathbb{R}$. Assume
$$\sup_{y}  \left( |f_y ^{(3)}|_{L^1} + |f_y |_{L^1} \right)<\infty.$$
 Then for every $k>0$ and $z\in [0,1]$,
$$Ef(z, k) = \frac{1}{\chi} \int_{-k} ^\infty \int_{\mathcal{A}} \int_{[0,1]} f( \sigma(i,y), x) d\nu (y) d\mathbf{p}(i)dx+e^{-\epsilon\cdot k}  O\left(e^{\epsilon | \supp f|} \left( \sup_{y,y'\in [D',D]} |( f_y)^{(3)}|_{L^1} +  |f_{y'}|_{L^1} \right) \right)$$
\begin{proof}
For  $z \in [0,1]$ and $u\in \mathbb{R}$ define a non-negative bounded Borel function
$$Qf(z,u) = \int f( \sigma(i,z), u)d\mathbf{p}(i).$$
Then  
$$Ef(z, k) = \sum_{n\geq 0 } Qf(\eta.z, \sigma(\eta, z)-t) d\mathbf{p}^n (\eta) = R(Qf)(z,k).$$
Since $\Phi$ is a $C^2$ IFS, the map $z\mapsto \sigma(i,z)\in C^1([0,1],\, [D',D])$ for all $i\in \mathcal{A}$, and therefore $z\mapsto \hat{f}_{\sigma(i,z)} (\theta) \in C^1([0,1])$ for all $i\in \mathcal{A}$ and $\theta \in \mathbb{R}$. So,  as our assumptions on $f$  imply that  $Qf(z,k)$  meets its conditions,  the result is now a direct application of Proposition \ref{Proposition renewal}.
\end{proof}
\end{Proposition}

\subsection{Residue process with cut-off}
Define a cutoff operator $E_C$ on real non-negative bounded Borel functions on $[D,D']\times \mathbb{R}$  via, for $(z,k)\in [0,1]\times \mathbb{R}$,
$$E_C f(z, k) = \sum_{n\geq 0} \int_{\sigma(\eta, z)<k\leq \sigma(i, \eta.z)+\sigma(\eta,z)} f( \sigma(i, \eta.z),\, \sigma(\eta,z) - k) d\mathbf{p}(i)d\mathbf{p}^n (\eta).$$
We have the following analogue of \cite[Lemma 4.9]{li2019trigonometric}, which follows here from  Lemma \ref{Lemma 4.6} (or Proposition \ref{Proposition 4.5}):
\begin{Lemma}  \label{Lemma 4.9}
There exists $C_2>0$ such that for all $k\in \mathbb{R}$ and $z\in [0,1]$ we have
$$E_C (\mathbf{1})(z, k)  \leq C_2$$
\end{Lemma}
\begin{proof}
Recalling \eqref{Eq C and C prime}, we have
\begin{eqnarray*}
E_C (\mathbf{1})(z, k) &=& \sum_{n\geq 0} \mathbf{p}\times \mathbf{p}^n \left( \lbrace (i,\eta): \sigma(\eta,z)-k\in [-\sigma(i,\eta.z),0) \rbrace \right)\\
&\leq & \sum_{n\geq 0} \mathbf{p}\times \mathbf{p}^n \left( \lbrace (i,\eta): \sigma(\eta,z)-k\in [-D',0] \rbrace \right)\\
&=& R\left( 1_{[-D',0]} \right)(z,k)\\
&\leq & C\cdot \max \lbrace 1, D' \rbrace,
\end{eqnarray*}
where in the last inequality we applied Lemma \ref{Lemma 4.6} (or Proposition \ref{Proposition 4.5}).
\end{proof}
By Lemma \ref{Lemma 4.9} the operator $E_C$ is well defined for bounded Borel functions. For a function $f(y,x)$ on $\mathbb{R}^2$ we denote, for all $x$,
$$f_x (y) = f(y,x).$$
 We have the following analogue of \cite[Proposition 4.10]{li2019trigonometric}:

\begin{Proposition} \label{Prop 4.10} Let $K\subseteq \mathbb{R}^2$ be a compact set, and fix a  bounded and continuous function $f(y,x)$ on $[D',D] \times \mathbb{R}$ such that $K$ contains $\supp(f)$, $f_y \in C_{P_2 K} ^{3} (\mathbb{R})$ for every $y\in [D',D]$, and $f_x \in C^1([0,1])$ for every $x\in \mathbb{R}$. Assume
$$\sup_{y}  \left( |f_y ^{(3)}|_{L^1} + |f_y |_{L^1} \right)<\infty.$$
Then for $|K|<k$ and every $z\in [0,1]$ we have
$$E_C f(z,k)= \int_{[0,1]} \int_\mathcal{A} \int_{-\sigma(i,y)} ^0 f(\sigma(i,y),x)dx d\mathbf{p}(i)d\nu(y) + e^{-\frac{\epsilon k}{2}} O_{|K|}(\left( \sup_{y,y'\in [D',D]} ||f_y ^{(3)} ||_{L^1} +  || f _{y'}||_{L^1}  \right)$$
\end{Proposition}
\begin{Remark}
By a standard decomposition into real and imaginary parts, and then each one into positive and negative parts, it suffices to prove Proposition \ref{Prop 4.10} under the assumption that $f$ is non-negative (with the same parameters assumed in the original Proposition). 
\end{Remark}
The following Lemma relates the operators $E_C$ and $E$:
\begin{Lemma} \cite[Lemma 4.12]{li2019trigonometric} Under the assumptions of Proposition \ref{Prop 4.10}, let
$$f_o (y,x):=1_{-y\leq x <0} f(y,x).$$
Then
$$E_C f(z,k)= Ef_o (z,k).$$
\end{Lemma}
Fix $1>\delta>0$ small enough so that $|K|+\delta \leq k$. 
We use $\psi_\delta$ to regularize these functions; Let
$$f_\delta(y,x) := \int f_o (y,x-x_1)\psi_\delta(x_1) dx_1 = \psi_\delta*f_o(y,x)$$
The following Lemma is an upgraded version of  \cite[Lemma 4.13]{li2019trigonometric}, as it has an exponential error term:
\begin{Lemma} \label{Lemma 4.13} Under the assumptions of Proposition \ref{Prop 4.10},
$$E(f_\delta)(z,k) = \int_{[0,1]} \int_{\mathcal{A}} \int_{-\sigma(i,y)} ^0 f(\sigma(i,y),x)dx d\mathbf{p}(i) d\nu(y) +e^{-\epsilon\cdot k}  O\left(e^{\epsilon | \supp f_\delta|} \left( \sup_{y, y'\in [D',D]} |( f_{\delta,y})^{(3)}|_{L^1} + |f_{\delta,y'}|_{L^1} \right) \right)$$
\end{Lemma}
\begin{proof}
We first wish to apply Proposition \ref{Proposition 4.7} to the function $f_\delta$. To this end, notice that by our assumptions for every $y\in [D',D]$ the function $f_y \in C^3 _{P_2 K} (\mathbb{R})$, and satisfies the required integrability conditions. Also, since $f_x \in C^1([0,1])$ for every $x\in \mathbb{R}$, it is clear that for every $\theta \in \mathbb{R}$, the function
$$z\mapsto \left(\hat{f_o}\right)_z (\theta) = \int e^{i \theta u} f_o(z,u)\,du =  \int_{-z} ^0 e^{i \theta u} f(z,u)\,du$$
belongs to $C^1([D,D'])$. 

So, applying Proposition \ref{Proposition 4.7},
$$E(f_\delta)(z,k) =\frac{1}{\chi} \int_{-k} ^\infty \int_\mathcal{A} \int_{[0,1]} f_\delta( \sigma(i,y), x) d \nu(y)  d\mathbf{p}(i)dx +e^{-\epsilon\cdot k}  O\left(e^{\epsilon | \supp f_\delta|} \left( \sup_{y,y'} |( f_{\delta,y})^{(3)}|_{L^1} +  |f_{\delta,y'}|_{L^1} \right) \right)$$
Also, for every $y \in P_1 \supp(f)$
$$\int_{-k} ^\infty f_\delta (y,x)dx =\int_{-k} ^\infty \int_{-y} ^0 f(y,x_1)\psi_\delta (x-x_1)dx_1  dx =  \int_{-y} ^0 f(y,x_1) \int_{-k} ^\infty \psi_\delta (x-x_1)dx dx_1. $$
Now, since $k-\delta \geq |K|$ and $y\in P_1 \supp( f)$ then 
$$-k-x_1 \leq -k+y\leq -\delta.$$ 
It follows that for every $x_1 \in P_2 (\supp f)$ we have
$$1 \geq \int_{-k} ^\infty \psi_\delta (x-x_1)dx \geq \int_{-\delta} ^\infty \psi_\delta =1 .  $$
So,
$$\int_{-k} ^\infty \int_\mathcal{A} \int_{[0,1]} f_\delta( \sigma(i,y), x) d \nu(y)  d\mathbf{p}(i)dx = \int_{[0,1]} \int_{\mathcal{A}} \int_{-\sigma(i,y)} ^0 f(\sigma(i,y),x)dx d\mathbf{p}(i) d\nu(y)$$
which implies the Lemma.
\end{proof}


The following Lemma is based on \cite[Lemma 4.15]{li2019trigonometric}. It is upgraded due to the fact that our bump function vanishes outside of $[-1,1]$.
\begin{Lemma} \label{Lemma 4.15}
Let $\varphi$ be a $C^1 (\mathbb{R})$ function with $||\varphi'||_\infty<\infty$ and $||\varphi||_\infty \leq 1$. Let 
$$\varphi_o (u) = 1_{[a,b]}(u) \varphi(u).$$ 
Then $|\psi_\delta * \varphi_o(u) - \varphi_o (u)|$ is bounded by:
\begin{itemize}
\item $||\varphi'||_\infty\cdot \delta$ if $u\in [a+\delta,\, b-\delta]$.
\item $2$ if $u\in [a-\delta,\,  a+\delta]$ or $u\in [b-\delta,\, b+\delta]$. 
\item $\psi_\delta *1_{[a,b]} (u)$ if $u<a-\delta$ or $u>b+\delta$.
\end{itemize}
\end{Lemma}
\begin{proof}
If $u\in [a+\delta,b-\delta]$ then, since $\psi_\delta$ is supported on $[-\delta^2, \delta^2]$
$$\left|\psi_\delta * \varphi_o(u) - \varphi_o (u)\right| = \left|\int \psi_\delta (t)(\varphi_o(u-t)-\varphi_o(u))dt\right|$$
$$\leq \int_{-\delta} ^{\delta} \psi_\delta(t) \left|\varphi_o(u-t)-\varphi_o(u)\right|dt.$$
If $|t|\leq \delta$ then $u-t \in [a,b]$. Since $|\varphi_o ' (u)|\leq ||\varphi'||_\infty$ for $u\in [a,b]$ we obtain
$$\int_{-\delta} ^{\delta} \psi_\delta(t)\left|\varphi_o(u-t)-\varphi_o(u)\right|dt \leq \int_{-\delta} ^{\delta} \psi_\delta(t)\cdot \left|t \right| \cdot ||\varphi'||_\infty dt \leq \delta ||\varphi'||_\infty.$$

In the second case, we use the trivial bound
$$|\psi_\delta * \varphi_o(u) - \varphi_o (u)|\leq 2.$$

In the third case, if $u\in (-\infty, a-\delta)\cup(b+\delta,\infty)$ then $\varphi_o (u)=0$, and so
$$|\psi_\delta*\varphi_o|\leq |\psi_\delta*1_{[a,b]}|.$$
This gives the Lemma.
\end{proof}

\noindent{\textbf{Proof of Proposition \ref{Prop 4.10}}} By Lemma \ref{Lemma 4.13}, we only need to estimate $E(|f_\delta - f_o|)(z,k)$.

Since $f_o(y,x)=1_{-y\leq x<0}(x) f(y,x)$, applying Lemma \ref{Lemma 4.15}, $|f_\delta- f_o|(x)$ is bounded by the sum of the following three terms (we sometimes omit the $y$ variable in the following computation):
\begin{itemize}
\item $\sup_{y} || f' _y||_\infty \cdot \delta$ if $x\in [-y+\delta,-\delta]$.
\item $2$ if $-y-\delta\leq x\leq -y+\delta$ or $-\delta \leq  x \leq  \delta$
\item $\psi_\delta*1_{[-y,0]} (x)$ if $x< -y-\delta$ or $x>\delta$.  
\end{itemize}
By the definition of $|K|$, the first term is smaller than
$$\sup_{y} | f' _y|_\infty \cdot \delta\cdot 1_{[-|K|+\delta,-\delta]}.$$
The third term is equal to
$$1_{[-\infty, -y-\delta]\cup[\delta,\infty]}\psi_\delta*1_{[-y,0]}(x) = 1_{[-\infty, -y-\delta]\cup[\delta,\infty]}(x) \int_{-y} ^0 \psi_\delta (x-x_1)dx_1$$
$$=1_{[-\infty, -y-\delta]\cup[\delta,\infty]}(x) \int_x ^{x+y} \psi_\delta (x_1) dx_1.$$

This gives us 
$$E\left( \left|f_\delta - f_o\right| \right) (z, k) = \sum_{n\geq 0} \int \int |f_\delta - f_o|(\sigma(i, \eta.z),\, \sigma(\eta,\,z)-k) d\mathbf{p}^n (\eta)d\mathbf{p}(i)$$
$$\leq \sum_{n\geq 0} \int ( \sup_{y} || f' _y||_\infty \cdot \delta\cdot 1_{[-|K|,-\delta]}(\sigma(\eta,z)-k) $$
$$+2\cdot  1_{[-\sigma(i, \eta.z)-\delta,-\sigma(i,\eta.z)+\delta] \cup [-\delta,\delta]} (\sigma(\eta,z)-k)$$
$$+ 1_{[-\infty, -\sigma(i, \eta.z)-\delta]\cup[\delta,\infty]} (\sigma(\eta,z)-k) \int_{\sigma(\eta,z)-k} ^{\sigma(\eta,z)-k+\sigma(i,\eta.z)} \psi_\delta (x_1)  dx_1)  d\mathbf{p}^n (\eta) d\mathbf{p}(i).$$

By Lemma \ref{Lemma 4.6} the first term is dominated by 
$$\sup_{y} | f' _y|_\infty \cdot \delta|K|.$$
For the second term, since $\sigma$ is a cocycle,
$$\int  1_{[-\sigma(i, \eta.z)-\delta,-\sigma(i,\eta.z)+\delta]} (\sigma(\eta,z)-k) d\mathbf{p}^n (\eta) \mathbf{p}(i) = \int 1_{[-\delta,\delta]} (\sigma(\eta,z) - k)d\mathbf{p}^{n+1}(\eta).$$
So, the second term  is smaller than $ 4R\left(1_{[-\delta,\delta]}\right)(z,k)$. By proposition \ref{Proposition 4.5}, this is dominated by
$$3\delta\left(\frac{1}{\chi}+e^{-\epsilon\cdot k} C_\psi O\left(\frac{1}{\delta^{6}} + 1 \right) \right)$$
The third term, by a change of the order of integration (as in \cite[Proof of Proposition 4.10]{li2019trigonometric}), is smaller than
$$\int_{[-\infty, -\delta] \cup \delta,\infty]} \psi_\delta (x_1) E_C (1)(z, x_1+k)dx_1.$$
By Lemma \ref{Lemma 4.9}, this is smaller than
$$  C_2 \int_{[-\infty, -\delta] \cup \delta,\infty]} \psi_{\delta} (x_1) dx_1.$$
This latter term is $0$ since $\psi_\delta$ is supported on $[-\delta^2, \delta^2]$.

$$ $$
\noindent{\textbf{Conclusion of proof}} Assuming $0<\delta\leq 1$ is small enough, we have shown that
$$E_C f(z,k)= \int_{[0,1]} \int_\mathcal{A} \int_{-\sigma(i,y)} ^0 f(\sigma(i,y),x)dx d\mathbf{p}(i)d\nu(y)$$
with the sum following error terms: From Lemma \ref{Lemma 4.13} we have
$$e^{-\epsilon\cdot k}  O\left(e^{\epsilon | \supp f_\delta|} \left( \sup_{y,y'} |( f_{\delta,y})^{(3)}|_{L^1} +  |f_{\delta,y'}|_{L^1} \right) \right)$$
and from the previous argument above,
$$\sup_{y} | f' _y|_\infty \cdot \delta|K|+6\delta\left(\frac{1}{\chi}+e^{-\epsilon\cdot k} C_\psi O\left(\frac{1}{\delta^{6}} + 1 \right) \right).$$

Note that, for every $y\in [D,D']$,
$$|( f_{\delta,y})^{(3)}|_{L^1} = |( \psi_\delta*f_o(y,x))^{(3)}|_{L^1} =  |\psi_\delta*(f_o(y,x))^{(3)}|_{L^1} \leq  |\psi_\delta|_{L^1}\cdot |(f_o(y,x))^{(3)}|_{L^1} \leq |( f_{y})^{(3)}|_{L^1}.  $$
Indeed, the second equality holds since $\frac{d}{dx}f_o (y,x)$ exists except for  at most two points, the third inequality is Young's inequality, and the last one is trivial. A similar calculation shows that
$$|f_{\delta,y}|_{L^1} \leq |f_{y}|_{L^1}.$$
Next, put 
$$\delta = e^{-\frac{\epsilon k}{12}}$$
so that
$$e^{-\epsilon k} \cdot \delta^{-6} = e^{-\frac{\epsilon k}{2}}$$
and the error becomes
$$ e^{-\frac{\epsilon k}{2}} O_{|K|}\left( \sup_{y,y'\in [D',D]} ||f_y ||_{L^1} + ||f_{y'} ^{(3)} ||_{L^1} \right).$$
This completes the proof.

 \subsection{Proof of Theorem \ref{Theorem equi}}
First, let use relate our previous discussion to the symbolic setting outlined prior to Theorem \ref{Theorem equi}: Recall that  for $x\in [0,1]$ and $k\in \mathbb{R}$ we defined a cutoff operator
$$E_C f(x, k) = \sum_{n\geq 0} \int_{c(\eta, x)<k\leq c(i, \eta.x)+c(\eta,x)} f( c(i, \eta.x),\, c(\eta,x) - k) d\mathbf{p}(i)d\mathbf{p}^n (\eta).$$ 
We can define a symbolic analogue via, for $\omega \in \mathcal{A}^\mathbb{N}$,
$$E_C f(\omega, k) = \sum_{n\geq 0} \int_{\tilde c (\eta, \omega)<k\leq \tilde c (i, \eta.\omega)+\tilde c (\eta,\omega)} f( \tilde c (i, \eta.\omega),\, \tilde c (\eta,\omega) - k) d\mathbf{p}(i)d\mathbf{p}^n (\eta).$$ 
Then, by the definition of our cocycles $c$ \eqref{The der cocycle} and $\tilde c$ \eqref{The symbolic der cocycle} 
$$E_C f(\omega, k) = E_C f(x_\omega, k).$$
Thus, the conclusion of Proposition \ref{Prop 4.10} applies to $E_C (\omega,k)$ with the same error terms.

 Now, let $g\in C^{3}  (\mathbb{R})$. Let $\rho$ be a smooth cutoff function such that $\rho|_{[0, D']}=1$, and such that it becomes $0$ outside of $[-1, D' +1]$. Let 
 $$f(y,x):=g(y+x)\rho(y) \rho(x+y).$$
 Then $f(y,x)=g(x+y)$ when $y,x+y\in [0, D']$. By definition,
 $$\mathbb{E}\left( g \left( S_{\tau_k(\eta)} (\eta)-k \right) \big|\, \sigma^{\tau_k(\eta)} \eta  =\omega \right) = E_C f(\omega,k) = E_C f(x_\omega, k).$$
The function $f$ satisfies the conditions of Proposition \ref{Prop 4.10}, and so we can also apply  this proposition to $E_C f(x_\omega, k)$. The conclusion of Theorem \ref{Theorem equi} follows, with an error term of
$$e^{-\frac{\epsilon k}{2}} O_{|K|}(\left( \sup_{y,y'\in [D,D']} ||f_y ^{(3)} ||_{L^1} +  || f _{y'}||_{L^1}  \right).$$
Since $\supp(\rho)\subseteq [-1,D'+1]$, for every $y\in [D,D']$ the function $f_y$ is supported on $[-1-D, 1+D']$. So, we can take $K=[D,D']\times [-1-D, 1+D']$, and for all $y\in [D,D']$
$$|| f _y||_{L^1} = O_\rho\left( \max_{x\in [-1-D, D'+1]} \left|g(x)\right| \right).$$
Similarly, for all $y'\in [D,D']$
$$||f_{y'} ^{(3)} ||_{L^1} = O_\rho\left( \max_{i=0,1,2,3} \max_{x\in [-1-D, D'+1]} \left|g^{(i)}(x)\right| \right).$$
Combining the previous three displayed equation, the Theorem is proved.

\section{From equidistribution to Fourier decay} \label{Section Fourier decay}
We keep our assumptions and notations from Section \ref{Section non-linearity implies s-gap}.  In particular, we are working with the induced IFS from Claim \ref{Claim key} so that Theorem \ref{Theorem spectral gap} holds. Therefore, Proposition \ref{Proposition renewal} holds, and so we have Theorem \ref{Theorem equi} at our disposal. This Theorem will be the key to our arguments in this Section.  Recall also  that $\Phi$ is $C^2$, and that $\nu_\mathbf{p}$ is a self-conformal measure such that $\mathbf{p}$ is a strictly positive probability vector. Recall that $\mathbb{P}=\mathbf{p}^\mathbb{N}$ is our Bernoulli measure on $\mathcal{A}^\mathbb{N}$.

In this section we show that:
$$\text{ There exists }  \alpha>0 \text{ such that } \left|\mathcal{F}_q (\nu)\right| \leq O\left( \frac{1}{ \left| q \right| ^\alpha} \right), \text{ as } |q|\rightarrow \infty.$$
We will prove this under the additional (minor) assumption that $\Phi$ is orientation preserving; The general case is similar, but notationally heavier.

Let $|q|$ be large and let $k=k(q)\approx \log |q|$ to be chosen later. Let $\epsilon>0$ be as in Theorem \ref{Theorem equi}.  We define a stopping time $\beta_k :\mathcal{A}^\mathbb{N} \rightarrow \mathbb{N}$ by
$$\beta_k (\omega)= \min \lbrace m:\, \left| f_{\omega|_m} ' (x_0) \right| < e^{-k-\frac{\epsilon k}{8}} \rbrace, \, \text{ where } x_0 \in I \text{ is a prefixed point}.$$
We require the following Theorem from \cite{algom2020decay}, that relates the random variables $S_{\tau_k}$, $\tau_k$ (defined in Section \ref{Section equi}), and $\beta_k$. For any $t\in \mathbb{R}$ let $M_t:\mathbb{R}\rightarrow \mathbb{R}$ denote the scaling map $M_t (x)=t\cdot x$. 
\begin{theorem}   (Linearization) \label{Theorem linearization}
For any $\beta \in (0,1)$,
$$\left| \mathcal{F}_q (\nu) \right|^2 \leq \int \left| \mathcal{F}_q (M_{e^{-S_{\tau_k (\omega)}}} \circ f_{\omega|_{\tau_k{(\omega)}+1} ^{\beta_k(\omega)}}  \nu) \right|^2 d\mathbb{P}(\omega)+O\left(|q| e^{-(k+\frac{k\epsilon}{8})-\beta\cdot \frac{k\epsilon}{8}} \right).$$
Also, there exists a global constant $C'>1$ such that for every $\omega \in \mathcal{A}^\mathbb{N}$
\begin{equation} \label{Eq. uniform push}
\left| f_{\omega|_{\tau_k{(\omega)}+1} ^{\beta_k(\omega)}} ' (x) \right| = \Theta_{C'} \left(e^{-\frac{\epsilon k}{8}} \right).
\end{equation}
\end{theorem}
\begin{proof}
This is a combination of \cite[Lemma 4.3, Lemma 4.4, and Claim 4.5]{algom2020decay}.
\end{proof}
Theorem \ref{Theorem linearization} is the only place in the proof where we use our additional assumption that $\Phi$ is orientation preserving. See \cite[Corollary 4.6 and Remark 4.7]{algom2020decay} on how to remove this assumption.

Next, for every $k>0$ we define a measurable partition $\mathcal{P}_k$ of $\mathcal{A}^\mathbb{N}$ via the relation
$$\omega \sim_{\mathcal{P}_k} \eta \iff \sigma^{\tau_{k} (\omega)} \omega = \sigma^{\tau_{k} (\eta)} \eta.$$
Writing $\mathbb{E}_{\mathcal{P}_k (\xi)} (\cdot )$ for the expectation with respect to the  conditional measure of $\mathbb{P}$ on a cell corresponding to a $\mathbb{P}$-typical $\xi$, it follows from Theorem \ref{Theorem linearization} and the law of total probability that:  For any $\beta \in (0,1)$,
\begin{equation} \label{Eq gk}
\left| \mathcal{F}_q (\nu) \right|^2 \leq \int \mathbb{E}_{\mathcal{P}_k (\xi)} \left( \left| \mathcal{F}_q (M_{e^{-S_{\tau_k (\omega)}}} \circ f_{\xi|_{\tau_k{(\xi)}+1} ^{\beta_k(\xi)}}  \nu) \right|^2 \right) d\mathbb{P}(\xi)+O\left(|q| e^{-(k+\frac{k\epsilon}{8})-\beta\cdot \frac{k\epsilon}{8}} \right).
\end{equation}
And, for every $\xi \in \mathcal{A}^\mathbb{N}$,
$$ \left| f_{\xi|_{\tau_k{(\xi)}+1} ^{\beta_k(\xi)}} ' (x) \right| = \Theta_{C'} \left(e^{-\frac{\epsilon k}{8}} \right).$$

Now, for every fixed $\xi \in \mathcal{A}^\mathbb{N}$ and $k>0$, define the $C^\omega$ function
$$g_{k,\xi} (t)=\left| \mathcal{F}_q \left( M_{e^{(-t-k)}} \circ f_{\xi|_{\tau_k{(\xi)}+1} ^{\beta_k(\xi)}}  \nu \right) \right|^2.$$
Then by the definition of the local $C^3$ norm on $[-1-D, D'+1]$ as in \eqref{Def local norm}, assuming $|q|\cdot e^{-k}\rightarrow \infty$ as $q\rightarrow \infty$,
$$||g_{k,\xi}||_{C^3} \leq O\left( \left( |q|\cdot e^{-k} \right)^3 \right).$$
We emphasize that the bound above holds uniformly across $\xi$.  Applying Theorem \ref{Theorem equi} we obtain for a $\mathbb{P}$-typical $\xi$
\begin{eqnarray*}
\mathbb{E}_{\mathcal{P}_k (\xi)} \left( \left| \mathcal{F}_q (M_{e^{-S_{\tau_k (\omega)}}} \circ f_{\xi|_{\tau_k{(\xi)}+1} ^{\beta_k(\xi)}}  \nu) \right|^2 \right) &=&  \mathbb{E}_{\mathcal{P}_k (\xi)} \left( g_{k, \xi} \left(S_{\tau_k(\omega)}-k \right) \right) \\
&=& \frac{1}{\chi} \int_{D} ^{D'} \int_{-y} ^0 g_{k,\xi}(x+y)\,dxd\kappa(y)+e^{-\frac{\epsilon k}{2}} O\left( \left( |q|\cdot e^{-k} \right)^3 \right) \\
&=&\frac{1}{\chi} \int_{D} ^{D'} \int_{-y} ^0 g_{k,\xi}(x+y)\,dxd\kappa(y)+e^{-\frac{\epsilon k}{2}}\left( |q|\cdot e^{-k} \right)^3 O\left( 1\right). \\
\end{eqnarray*}

Plugging the above equality into \eqref{Eq gk}, using the definition of $g_{k,\xi}$ and that if $f\geq 0$ then 
$$\frac{1}{\chi} \int_{D} ^{D'} \int_{-y} ^0 f(x+y) \,dxd\kappa(y) \leq \frac{1}{\chi} \int_{0} ^{D'}  f(t) \,dt, $$
we obtain:
\begin{eqnarray*}
\left| \mathcal{F}_q (\nu) \right|^2 &\leq& \int \mathbb{E}_{\mathcal{P}_k (\xi)} \left( \left| \mathcal{F}_q (M_{e^{-S_{\tau_k (\omega)}}} \circ f_{\xi|_{\tau_k{(\xi)}+1} ^{\beta_k(\xi)}}  \nu) \right|^2 \right) d\mathbb{P}(\xi)+O\left(|q| e^{-(k+\frac{k\epsilon}{8})-\beta\cdot \frac{k\epsilon}{8}} \right) \\
&= & \frac{1}{\chi} \int \int_{D} ^{D'} \int_{-y} ^0 g_{k,\xi}(x+y)\,dxd\kappa(y) d\mathbb{P}(\xi)+O\left(|q| e^{-(k+\frac{k\epsilon}{8})-\beta\cdot \frac{k\epsilon}{8}} \right)+e^{-\frac{\epsilon k}{2}}\left( |q|\cdot e^{-k} \right)^3 O\left( 1\right)\\
&\leq  & \frac{1}{\chi}  \int \int_0 ^{D'} \left| \mathcal{F}_q \left( M_{e^{(-t-k)}} \circ f_{\xi|_{\tau_k{(\xi)}+1} ^{\beta_k(\xi)}}  \nu \right) \right|^2 dt d\mathbb{P}(\xi)+O\left(|q| e^{-(k+\frac{k\epsilon}{8})-\beta\cdot \frac{k\epsilon}{8}} \right)\\
&&+e^{-\frac{\epsilon k}{2}}\left( |q|\cdot e^{-k} \right)^3 O\left( 1\right).
\end{eqnarray*}

Finally, we use the following Lemma (originally due to Hochman \cite{Hochman2020Host}) to deal with the oscillatory integral above:
\begin{Lemma} \cite[Lemma 2.6]{algom2020decay} (Oscillatory integral) For every $\xi \in  \mathcal{A}^\mathbb{N}$, $k>0$, and $r>0$
$$\int_0 ^{D'} \left| \mathcal{F}_q \left( M_{e^{(-t-k)}} \circ f_{\xi|_{\tau_k{(\xi)}+1} ^{\beta_k(\xi)}}  \nu \right) \right|^2 dx  =O\left( \frac{1}{r|q|e^{-(k+\frac{\epsilon k}{8})}}+\sup_{y} \nu(B_r (y)) \right)$$
\end{Lemma}
Note that we use \eqref{Eq. uniform push} to get uniformity in the first term on the right hand side.
$$ $$

\noindent{\textbf{Conclusion of proof}} By the argument above we can bound $\left| \mathcal{F}_q (\nu) \right|^2$ by the sum of the following terms. Every term is bounded with implicit dependence on $\mathbf{p}$ and the underlying IFS. For simplicity, we ignore global multiplicative constants so we omit the big-$O$ notation:

Linearization: For any prefixed $\beta\in (0,1)$,
$$|q| e^{-(k+\frac{k\epsilon}{8})-\beta\cdot \frac{k\epsilon}{8}};$$
Equidistribution: 
$$e^{-\frac{\epsilon k}{2}}\left( |q|\cdot e^{-k} \right)^3 ;$$
Oscillatory integral: For every $r>0$
$$\frac{1}{r|q|e^{-(k+\frac{\epsilon k}{8})}}+\sup_{y} \nu(B_r (y)).$$

\noindent{\textbf{Choice of parameters}} For $|q|$ we choose $k=k(q)$ that satisfies 
$$|q|=e^{k+\frac{k\epsilon}{7}}.$$
We also choose $r=e^{-\frac{k \epsilon}{100}}$ and $\beta = \frac{1}{2}$.  Then we get:

Linearization:
$$|q| e^{-(k+\frac{k\epsilon}{8})-\beta\cdot \frac{k\epsilon}{8}} = e^{\frac{k \epsilon}{7} -\frac{k\epsilon}{8}- \frac{k\epsilon}{16}},\, \text{ this decay exponentially fast in }k.$$
Equidistribution: 
$$e^{-\frac{\epsilon k}{2}}\left( |q|\cdot e^{-k} \right)^3 =  e^{-\frac{\epsilon k}{2} +\frac{3k\epsilon}{7}},\, \text{ this decay exponentially fast in }k.$$
Oscillatory integral: There is some $d=d(\nu)>0$ such that
$$\frac{1}{r|q|e^{-(k+\frac{\epsilon k}{8})}} + \sup_{y} \nu(B_r (y)) \leq  \frac{1}{e^{-\frac{k\epsilon}{100}+ \frac{k \epsilon}{7} -\frac{k\epsilon}{8}}}+e^{-\frac{d \epsilon k}{100}},\, \text{ this decay exponentially fast in }k.$$
Here we made use of \cite[Proposition 2.2]{Feng2009Lau}, where it is shown that there is some $C>0$ such that for every $r>0$ small enough $\sup_{y} \nu(B_r (y))\leq Cr^d$. 

Finally, summing these error terms, we see that for some $\alpha>0$ we have $\left| \mathcal{F}_q (\nu) \right| = O\left(e^{-k\alpha}\right)$. Since as $|q|\rightarrow \infty$ we have $k\geq C_0 \cdot \log |q|$ for some uniform $C_0>0$, our claim follows.   \hfill{$\Box$}

\section{On the proof of Corollary \ref{Main Corollary}}
In this Section we prove Corollary \ref{Main Corollary}: All self-conformal measures with respect to a $C^\omega (\mathbb{R})$ IFS that contains a non-affine map have polynomial Fourier decay:  First, we show that any given $C^\omega (\mathbb{R})$ IFS is either not $C^2$ conjugate to linear, or it is $C^\omega$ conjugate to a self-similar IFS (Claim \ref{Claim JLMS}). By this dichotomy,   Corollary \ref{Main Corollary} follows from Theorem \ref{Main Theorem} and from  a Fourier decay result about smooth images of self-similar measures from a paper of of Algom et al. \cite{Algom2023Wu}. 
\subsection{On conjugate to linear real analytic IFSs}
Recall that $\Phi$, a $C^2(\mathbb{R})$ IFS, is  called linear if the following holds:
$$f'' (x) =0 \, \text{ for all } x\in K_\Phi \text{ and } f\in \Phi.$$
In particular, if $\Phi$ is $C^\omega (\mathbb{R})$ and linear then it is self-similar. Recall that an IFS $\Psi$ is called $C^2$ conjugate to $\Phi$ if there is a $C^2$ diffeomorphsim $h$ between neighbourhoods of the corresponding attractors such that
$$\Psi = h\circ \Phi \circ h^{-1} := \lbrace h\circ g \circ h^{-1} \rbrace_{g\in \Phi}.$$

The geometric properties of linear non self-similar   $C^r (\mathbb{R})$ smooth  IFSs when $r\neq \omega$ are not well understand. In fact, even showing the existence of such IFSs  is a highly non-trivial question (no such example is known). However, in the analytic category our understanding is much better:
\begin{Claim} \label{Claim JLMS} Let $\Phi$ be a $C^\omega([0,1])$ IFS. Then there is a dichotomy:
\begin{enumerate}
\item $\Phi$ is not $C^2$ conjugate to linear, or

\item $\Phi$ is conjugate to a linear $C^\omega (\mathbb{R})$ IFS via  an analytic map $g$, that is a  diffeomorphism on $[0,1]$. In particular, $\Phi$ is $C^\omega$ conjugate to a self-similar IFS.
\end{enumerate}
\end{Claim} 
Several variants of Claim \ref{Claim JLMS} exist in the literature under various assumptions (see e.g. \cite{BedfordFisher1997Ratio}). We also note that a $C^2$ version of the Claim holds if $K_\Phi$ is an interval (via a closely related argument). 

First, we require the following Proposition, a special case of the Poincar\'{e}-Siegel Theorem \cite[Theorem 2.8.2]{Katok1995Hass}:
\begin{Proposition} \label{Prop Poincare} \cite[Proposition 2.1.3]{Katok1995Hass} Let $g\in C^\omega ([0,1])$ be a contracting map. Then there exists some non-trivial interval $J\subseteq [0,1]$ and a diffeomorphism $h\in C^\omega ([0,1], J)$ such that $h\circ g\circ h^{-1}$ is affine.
\end{Proposition}
We also require the following Lemma:
\begin{Lemma} \label{Lemma lin fede}
Let $\Phi$ be a $C^2(\mathbb{R})$ IFS that is $C^2$ conjugate to a linear IFS. If there exists $g\in \Phi$ such that $g'' =0$ on its attractor $K_\Phi$, then $\Phi$ is already linear; That is, for every $f\in \Phi$, $f'' =0$ on $K_\Phi$.
\end{Lemma}
\begin{proof}
Let $h\in C^2 (\mathbb{R})$ be such that the IFS $h\circ \Phi \circ h^{-1}$ is linear. Let $g\in \Phi$ be such that $g'' =0$ on $K:=K_\Phi$. We first show that this implies that $h''$ vanishes on $K$: By our assumption, for all  $z\in h(K)$ we have
$$\left( h\circ g \circ h^{-1} \right)''(z)=0.$$
Writing $h^{-1} (z)=y\in K$, we compute:
$$\log \left( h\circ g \circ h^{-1} \right)' (z) = \log h'\circ g(y)+\log g'(y)-\log h'(y).$$
Combining the two previous displayed equations, it follow that
$$0= \frac{\left( h\circ g \circ h^{-1} \right)''(z)}{\left( h\circ g \circ h^{-1} \right)' (z)} = \frac{d}{dz}  \log\left( h\circ g \circ h^{-1} \right)' (z) = \frac{h''\circ g(y)\cdot g'(y)}{h'\circ g(y)}+\frac{g''(y)}{g'(y)}-\frac{h''(y)}{h'(y)}.$$
By our assumption $g''(y)=0$ since $y\in K$. We conclude that 
\begin{equation} \label{Eq key equation}
\frac{h''\circ g(y)\cdot g'(y)}{h'\circ g(y)} = \frac{h''(y)}{h'(y)}.
\end{equation}
Finally, let $x\in K$. Since   $g''(y)=0$ for all $y\in K$ and  for every $k\in \mathbb{N}$ we have for the $k$-fold composition $ g^{\circ k }(x) \in K$, then for all $k\in \mathbb{N}$, $\left( g^{\circ k } \right)''(x)=0$. A similar argument shows that for every $k$ the IFS $h\circ \Phi^k \circ h^{-1}$ is linear. It follows that in \eqref{Eq key equation} we can substitute $g^{\circ k}$ for $g$. As the IFS $\Phi$ is uniformly contracting, this shows that the LHS of \eqref{Eq key equation} can be made arbitrarily small, but the RHS remains fixed. This is only possible if $h''(y)=0$.  We conclude that $h''$ vanishes on $K$.

Finally, let $f\in \Phi$. Then for all  $z\in h(K)$ we have, for $y=h^{-1}(z) \in K$
$$0= \frac{d}{dx}  \log\left( h\circ f \circ h^{-1} \right)' (z) = \frac{h''\circ f(y)\cdot f'(y)}{h'\circ f(y)}+\frac{f''(y)}{f'(y)}-\frac{h''(y)}{h'(y)}.$$
Since $y\in K$ then $f(y)\in K$ and so by the previous paragraph $h''\circ f(y)=0$ and $h''(y)=0$.  As $|f'|>0$ on $[0,1]$ by assumption, this is only possible if $f''(y)=0$. It follows that $f''$ vanishes on $K_\Phi$, as claimed.
\end{proof}
$$ $$
\noindent{ \textbf{Proof of Claim \ref{Claim JLMS}}} Suppose $\Phi \in C^\omega (\mathbb{R})$ is  $C^2$ conjugate to linear. Let $g\in \Phi$ be any map. By Proposition \ref{Prop Poincare}, there is some non-trivial interval $J\subseteq [0,1]$ and a map $h\in C^\omega ([0,1], J)$ such that $h\circ g\circ h^{-1}$ is affine. Then the IFS $h\circ \Phi \circ h^{-1}$ is $C^\omega$ and contains an affine map. Since $\Phi$ is $C^2$ conjugate to linear, so is $h\circ \Phi \circ h^{-1}$. Since this IFS contains an affine map, by Lemma \ref{Lemma lin fede} it is already linear. So, it is linear and analytic, hence it must be self-similar. Thus, the second alternative of Claim \ref{Claim JLMS} holds true. \hfill{$\Box$}

\subsection{Proof of Corollary \ref{Main Corollary}}
We now prove Corollary \ref{Main Corollary}. Let $\Psi$ be a $C^\omega (\mathbb{R})$ IFS, and assume $\Psi$ contains a non-affine map. Recall that we are always assuming $K_\Psi$ is infinite.  By Claim \ref{Claim JLMS} there are two cases to consider:

The first alternative is that $\Psi$ is not $C^2$ conjugate to linear. Then, by Theorem \ref{Main Theorem}, every self-conformal measure $\nu$ admits some $\alpha>0$ such that
\begin{equation*}
\left| \mathcal{F}_q \left( \nu \right) \right| = O\left(\frac{1}{ |q|^\alpha} \right).
\end{equation*}

The second alternative is that $\Psi$ is $C^\omega$ conjugate to a self-similar IFS $\Phi$. Let $g$ denote the conjugating map. We have the following easy Lemma:
\begin{Lemma} \label{Lemma g not affine}
The analytic map $g$ is not affine.
\end{Lemma} 
\begin{proof}
Suppose towards a contradiction that $g$ is affine. Recall that $g$ is a conjugating map between $\Psi$, a $C^\omega$ IFS, and a self-similar IFS. So, both IFS's in question are in fact self-similar. However, our standing assumption is  that contains $\Psi$ contains a non affine map. This is a contradiction.
\end{proof}

So, in the second alternative, $\Psi$ is  conjugate to a self-similar IFS $\Phi$ via a $C^\omega$ map $g$ that is not affine. In particular, 
$$\left| \lbrace x\in [0,1]: g''(x)=0 \rbrace \right| <\infty.$$
Since every self-conformal measure with respect to $\Psi$ can be written as $g \mu$ where $\mu$ is a self-similar measure with respect to $\Phi$, the Fourier decay bound in the second  alternative case  is a direct consequence of the following Theorem of Algom et al. \cite{Algom2023Wu}:
\begin{theorem} \label{Theorem meng} \cite[Corollary 1.3]{Algom2023Wu}  Let $\mu$ be a non-atomic self-similar measure with respect to $\Phi$, and let $g\in C^\omega ([0,1])$ be such that $g'\neq 0$, and such that  $g''\neq 0$ except for possibly  finitely many points in $[0,1]$. Then there exists some $\alpha>0$ such that
\begin{equation*}
\left| \mathcal{F}_q \left( g \mu \right) \right| = O\left(\frac{1}{|q|^\alpha} \right).
\end{equation*} 

\end{theorem}
The proof of Corollary \ref{Main Corollary} is complete. \hfill{$\Box$}

\section{Acknowledgements}
We thank Simon Baker, Tuomas Sahlsten,  Meng Wu, and Osama Khalil, for useful discussions and for their remarks on this project. We also thank Joey Veltri for pointing out some bugs in a previous version of this manuscript. This research was supported by Grant No. 2022034 from the United States - Israel Binational Science Foundation 
(BSF), Jerusalem, Israel. 
\bibliography{bib}{}
\bibliographystyle{plain}

\end{document}